\documentclass[a4paper]{amsart}
\pdfoutput=1

\usepackage[table]{xcolor}
\usepackage{subcaption}
\usepackage{tikz, fullpage, float, amssymb}
\usepackage{mathtools, hyperref}

\usepackage[T1]{fontenc}
\usepackage[utf8]{inputenc}
\usepackage{lmodern}
\usepackage{tikz-cd}

\usepackage{array}
\usepackage{varwidth} %

\definecolor{Gray}{gray}{0.9}

\newcolumntype{g}{>{\columncolor{Gray}}c}
\newcolumntype{M}{V{4cm}} %
\newcolumntype{U}{V{2.5cm}} %
\newcolumntype{T}{V{2cm}} %

\definecolor{col1}{RGB}{254,97,0}
\definecolor{col2}{RGB}{100,143,255}
\definecolor{col3}{RGB}{120, 94, 240}
\definecolor{col4}{RGB}{220, 38, 127}
\definecolor{col5}{RGB}{255, 176, 0}

\usetikzlibrary{positioning}
\usetikzlibrary{matrix}
\usetikzlibrary{shapes}
\usetikzlibrary{arrows, automata}
\usetikzlibrary{calc,3d}
\usetikzlibrary{decorations,decorations.pathmorphing,decorations.pathreplacing,decorations.markings}
\usetikzlibrary{through}

\tikzset{pbrace/.style={decorate, decoration = {brace}}}
\tikzset{snakeit/.style={decorate, decoration={snake, amplitude=.2mm,segment length=1mm}}}

\tikzset{ext/.style={circle, draw,inner sep=1pt}, int/.style={circle,draw,fill,inner sep=1pt},nil/.style={inner sep=1pt}}
\tikzset{cy/.style={circle,draw,fill,inner sep=2pt},scy/.style={circle,draw,inner sep=2pt},scyx/.style={draw,cross out,inner sep=2pt},scyt/.style={draw,regular polygon,regular polygon sides=3,inner sep=0.95pt}}
\tikzset{exte/.style={circle, draw,inner sep=3pt},inte/.style={circle,draw,fill,inner sep=3pt}}
\tikzset{diagram/.style={matrix of math nodes, row sep=3em, column sep=2.5em, text height=1.5ex, text depth=0.25ex}}
\tikzset{diagram2/.style={matrix of math nodes, row sep=0.5em, column sep=0.5em, text height=1.5ex, text depth=0.25ex}}
\tikzset{rowcolsep/.style={column sep=.2cm, row sep=.1cm}}

\tikzset{
  crossed/.style={
    decoration={markings,mark=at position .5 with {\arrow{|}}},
    postaction={decorate},
    shorten >=0.4pt}}

\tikzset{every picture/.style={baseline=-.65ex} }
\tikzset{every loop/.style={draw}}

  \tikzset{->-/.style={decoration={
    markings,
    mark=at position .5 with {\arrow{>}}},postaction={decorate}}}
  \tikzset{-<-/.style={decoration={
      markings,
      mark=at position .5 with {\arrow{<}}},postaction={decorate}}}
  \tikzset{-><-/.style={decoration={
    markings,
    mark=at position .2 with {\arrow{>}},
    mark=at position .9 with {\arrow{<}}},postaction={decorate}}}

\newcommand{\Ass}{\mathsf{Assoc}}
\newcommand{\Ger}{\mathsf{Ger}}

\newcommand{\Com}{\mathsf{Com}}
\newcommand{\Lie}{\mathsf{Lie}}
\newcommand{\BV}{\mathsf{BV}}
\newcommand{\tBV}{\widetilde{\BV}}
\newcommand{\Grav}{\mathsf{Grav}}
\newcommand{\HyCom}{\mathsf{HyCom}}
\newcommand{\Out}{\mathsf{Out}}
\newcommand{\M}{\mathcal{M}}
\newcommand{\MM}{\overline{\mathcal{M}}}
\newcommand{\Feyn}{\mathrm{Feyn}}
\newcommand{\AFeyn}{\mathrm{AFeyn}}

\newcommand{\gr}{\mathrm{gr}}
\newcommand{\Hbdy}{\mathcal H \mathsf{bdy}}
\newcommand{\cP}{\mathcal{P}}
\newcommand{\cM}{\mathcal{M}}
\newcommand{\cC}{\mathcal{C}}

\newcommand{\tGrav}{\widetilde{\Grav}}
\newcommand{\Z}{\mathbb{Z}}
\newcommand{\Q}{\mathbb{Q}}
\newcommand{\R}{\mathbb{R}}

\newcommand{\bs}{\mathbf{s}}
\newcommand{\Dc}{\mathbf{\Delta^*}}

\newcommand{\gst}{\mathrm{star}}
\newcommand{\Det}{\mathrm{Det}}
\newcommand{\LMod}{\mathbb L \mathrm{Mod}}
\newcommand{\HBdy}{\Hbdy}
\newcommand{\tBVt}{\tBV^!_{\geq 2}}
\newcommand{\LD}{\mathsf{LD}}
\newcommand{\sgn}{\mathrm{sgn}}
\newcommand{\Hy}{\mathsf{HyC}}

\newcommand{\vspan}{\mathrm{span}}
\DeclareMathOperator{\coker}{\mathrm{coker}}
\newcommand{\DD}{\mathfrak{D}}
\newcommand{\kk}{\mathfrak{k}}
\newcommand{\COp}{\mathsf{C}}
\newcommand{\Free}{\mathrm{Free}}

\newcommand{\mF}{\mathcal{F}}
\newcommand{\MFree}{\mathbb M\Free}

\newcommand{\HMod}{\mathrm{HMod}}
\newcommand{\Mod}{\mathrm{Mod}}
\newcommand{\Bij}{\mathrm{Bij}}
\newcommand{\rt}{\blacklozenge}
\newcommand{\feyncom}{\operatorname{F}}

\newcommand{\ch}{\mathrm{ch}}
\newcommand{\Ch}{\mathrm{Ch}}
\newcommand{\Sym}{\mathrm{Sym}}
\newcommand{\SW}{\mathbb W}
\newcommand{\gm}{\mathbf{m}}

\newtheorem{thm}{Theorem}[section]
\newtheorem{prop}[thm]{Proposition}

\newtheorem{cor}[thm]{Corollary}

\newtheorem{lemma}[thm]{Lemma}

\theoremstyle{definition}

\theoremstyle{remark}
\newtheorem{remark}[thm]{Remark}
\newtheorem{example}[thm]{Example}
\newtheorem{convention}[thm]{Convention}

\hypersetup{
  pdfauthor={Michael Borinsky, Benjamin Brück, Thomas Willwacher},
  pdftitle={Weight 2 cohomology of graph complexes of cyclic operads and the handlebody group},
  pdfsubject={},
  pdfkeywords={},
  colorlinks = true,
}

\title{Weight 2 cohomology of graph complexes of cyclic operads and the handlebody group}

\author{Michael Borinsky}
\address{
\parbox{\linewidth}{
Michael Borinsky\\
Institute for Theoretical Studies, ETH Z\"urich \\
Clausiusstrasse 47, 8006 Zurich, Switzerland}
}
\email{michael.borinsky@eth-its.ethz.ch}

\author{Benjamin Brück}
\address{
\parbox{\linewidth}{
Benjamin Brück\\
Department for Mathematical Logic and Foundational Research, University of M\"unster\\
Einsteinstraße 62,
48149 Münster, Germany}
}
\email{benjamin.brueck@uni-muenster.de}

\author{Thomas Willwacher}
\address{
\parbox{\linewidth}{
Thomas Willwacher\\
Department of Mathematics, ETH Z\"urich \\
R\"amistrasse 101,
8092 Zurich, Switzerland}
}
\email{thomas.willwacher@math.ethz.ch}

\begin{document}

\begin{abstract}
We compute the weight 2 cohomology of the Feynman transforms of the cyclic (co)operads $\mathsf{BV}$ and $\mathsf{HyCom}$, and the top$-2$ weight cohomology of the Feynman transforms of $D\mathsf{BV}$ and $\mathsf{Grav}$.  Using a result of Giansiracusa, we compute, in particular, the top$-2$ weight cohomology of the handlebody group.  We compare the result to the top$-2$ weight cohomology of the moduli space of curves $\mathcal{M}_{g,n}$, recently computed by Payne and the last-named author.  We also provide another proof of a recent result of Hainaut--Petersen identifying the top weight cohomology of the handlebody group with the Kontsevich graph cohomology.
\end{abstract}

\maketitle

\tableofcontents

\section{Introduction}

To every cyclic operad $\cP$ one may associate a graph complex $\Feyn(\cP)$ via the Feynman transform of Getzler-Kapranov \cite{GK}.
Elements of $\Feyn(\cP)$ can be seen as linear combinations of graphs whose vertices are decorated by elements of $\cP$.
Arguably the most common and prominent cyclic operads are the following:
\begin{itemize}
\item The commutative operad $\Com$ and its Koszul dual operad $\Lie$.
\item The associative operad $\Ass$, which is Koszul self-dual.
\item The Batalin-Vilkovisky operad $\BV$ and its dg dual operad $D\BV^*$. (See Section \ref{sec:dg dual def} below for our definition of $D(-)$.)
\item The hypercommutative operad $\HyCom=H_\bullet(\MM_{0,-})$ along with its dg dual operad, the gravity operad $\Grav$ \cite{Getzler0}.
\end{itemize}
For each of these classical operads $\cP$ we may ask two natural questions:
\medskip

  Question 1: What does the cohomology of $\Feyn(\cP)$ compute?

  \smallskip

Question 2: What is the cohomology of $\Feyn(\cP)$?

\medskip

Regarding the first question, satisfying answers have been found for 4 of the 7 cyclic operads above, see Table~\ref{tab:overview}.
These show that in many cases, $\Feyn(\cP)$ can be used to compute the cohomology of objects in algebraic geometry, topology and group theory and hence give motivation for studying Question 2.
Of particular interest for this paper will be the following result of Giansiracusa.\footnote{Only the first of the two isomorphisms has been stated in \cite{Giansiracusa}. But the second follows relatively easily, see Appendix \ref{sec:giansiracusa} for a brief discussion.}

\begin{thm}[Giansiracusa \cite{Giansiracusa}] \label{thm:giansiracusa}
    Let $V_{g,n}^m$ be a genus $g$ handlebody with $m$ distinct marked disks and $n$ marked points on the boundary. Let $\HMod_{g,n}^m:=\pi_0 \mathrm{Diff}(V_{g,n}^m)$ be the handlebody group.
    Then 
    \[
    H^\bullet(\HMod_{g,0}^m) \cong H^\bullet(\Feyn_\kk(D\BV^*)) (\!(g,m)\!)  
    \]
    as long as $(g,m)\neq (1,0)$, where $\Feyn_\kk(-)$ denotes the Feynman transform of 1-shifted cyclic operads. Furthermore, 
    \[
    H^\bullet(\HMod_{g,n}^0) \cong H^\bullet(\AFeyn_\kk(D\BV^*)) (\!(g,n)\!)  
    \]
    with $\AFeyn_\kk$ the amputated Feynman transform (see Section \ref{sec:amputated} below), 
    as long as $(g,n)\neq (1,0),(0,2)$.
\end{thm}

The second question above, pertaining to the actual value of $H(\Feyn(\cP))$, is complicated, and for none of the 7 cyclic operads $\cP$ above a complete answer is known.
However, there is a long list of partial results, in particular about the ``commutative'' graph cohomology $H(\Feyn(\Com))$, see \cite{AroneTurchin, Willwachergrt, KWZ1} and references therein.
The aim of this article is to study the second question for $\cP=D\BV^*$ and the closely related cases $\cP = \BV$, $\HyCom$ and $\Grav$.

\begin{table}[H]
\begin{tabular}{|c|p{0.8\textwidth}|}
\hline
\bf (Co)Operad $\cP$ & \bf $\Feyn(\cP)$ computes... \\
\hline
$\Com$ & Various objects, in particular the top-weight cohomology of $\M_{g,n}$ \cite{CGP1, CGP2} and homotopy groups of embedding spaces \cite{AroneTurchin, FTW}. \\
\hline
$\Lie$ & Cohomology of outer automorphism groups of free groups and the related groups $\Gamma_{g,n}$ \cite{CHKV}. \\
\hline
$\Ass$ & Cohomology of the moduli spaces of curves $\M_{g,n}$, see \cite[Theorem 9.4]{GK} (the result is due to Penner and Kontsevich).\\
\hline 
$\BV$ & \emph{Unknown} \\
\hline
$D\BV^*$ & Cohomology of the handlebody group \cite{Giansiracusa}, cf. Theorem \ref{thm:giansiracusa}\\
\hline
$\HyCom$ & \emph{Unknown}, but see Theorem \ref{thm:dbv_hycom} \\
\hline
$\Grav$ & Homotopy quotient of $\Feyn(D\BV)$ by the binary part of $H_\bullet(S^1)$
and the cohomology of the image of the gluing map $\MM_{0,n+2g} \to \MM_{g,n}$. 
\\
\hline
\end{tabular}
\caption{\label{tab:overview} Roles of Feynman transforms of cyclic operads.}
\end{table}

\subsection{Results about the Feynman transform of \texorpdfstring{$D\BV^*$}{DBV*}}
In this paper we relate parts of the cohomologies of the graph complexes $\Feyn(\cP)$ for the cases $\cP=\BV, \HyCom, \Grav$ and $D\BV^*$
to better understood objects.%
\footnote{We note that we use slightly non-standard notation and conventions for the Feynman transform. We refer to Section \ref{sec:feyn} below, also for the notation $\Feyn_{\kk}(-)$.}
To this end we use that the homological degrees on these operads
induce an additional grading on the complexes $\Feyn(\cP)$.
In particular, twice the cohomological degree of $\BV^*$ defines a grading on $\Feyn_{\kk}(D\BV^*)$ that we call the \emph{weight} grading.
This particular convention for the weight grading is convenient as it is compatible 
with the weight grading on $\HyCom$ which in turn is compatible with the mixed Hodge structure of $H^\bullet(\MM_{0,n})$. 
We denote the piece of weight $W$ by $\gr_{W}(-)$, so that 
\begin{align}\label{equ:degcomp Feyn} H^\bullet(\Feyn_{\kk}(D\BV^*)(\!(g,n)\!))&= \bigoplus_{W}\gr_{W} H^\bullet(\Feyn_{\kk}(D\BV^*)(\!(g,n)\!)) \\ \label{equ:degcomp Feyn 2} H^\bullet(\AFeyn_{\kk}(D\BV^*)(\!(g,n)\!))&= \bigoplus_{W}\gr_{W} H^\bullet(\AFeyn_{\kk}(D\BV^*)(\!(g,n)\!)). \end{align}

We first determine the values of $W$ for which the graded parts in \eqref{equ:degcomp Feyn} and \eqref{equ:degcomp Feyn 2} are non-trivial and describe the bottom and top-weight parts: 

\begin{thm}\label{thm:main}
    Let $(g,n)\neq (1,0),(1,1),(0,2)$.
\begin{enumerate}
\item The weights occurring nontrivially in \eqref{equ:degcomp Feyn} lie in the range $0\leq W\leq 6g-6+4n$, and the weights in \eqref{equ:degcomp Feyn 2} lie in the range $0\leq W\leq 6g-6+2n$.
\item The bottom weight parts in \eqref{equ:degcomp Feyn} and \eqref{equ:degcomp Feyn 2} are equal to the forested graph cohomology.
\[
    \gr_{0} H^\bullet(\Feyn_\kk(D\BV^*)(\!(g,n)\!))
    \cong \gr_{0} H^\bullet(\AFeyn_\kk(D\BV^*)(\!(g,n)\!))
    \cong H^\bullet(\Feyn_\kk(D\Com^*))(\!(g,n)\!).
\]
In particular, we have that 
\[
    \gr_{0} H^\bullet(\Feyn_\kk(D\BV^*)(\!(g,0)\!))
    \cong
    H^\bullet(\Out(F_g)).
\]
\item
The top-weight parts in \eqref{equ:degcomp Feyn} and \eqref{equ:degcomp Feyn 2} can be expressed through
the cohomology of the Kontsevich commutative graph complex
\begin{align*} \gr_{6g-6+2n} H^{6g-6+2n-k}(\AFeyn_\kk(D\BV^*)(\!(g,n)\!)) &\cong H^{-k}(\Feyn(\Com)(\!(g,n)\!)) \\
\gr_{6g-6+4n} H^{6g-6+3n -k}(\Feyn_\kk(D\BV^*)(\!(g,n)\!)) &\cong H^{-k}(\Feyn(\Com)(\!(g,n)\!))\otimes \sgn_n . \end{align*}
\end{enumerate}
\end{thm}

We can extend the above result to the part of top$-2$-weight. In the simplest case this reads as follows, for the amputated Feynman transform:
\begin{thm}\label{thm:dbv_hycom}
For $(g,n)\neq (1,0),(1,1),(0,2)$  we have that 
\[
\gr_{6g-8+2n} H^{6g-6+2n-k}(\AFeyn_\kk(D\BV^*)(\!(g,n)\!)) 
\cong 
\gr_2 H^{-k}(\Feyn(\HyCom)(\!(g,n)\!)).
\]
\end{thm}

The analogous result for the full Feynman transform of $D\BV^*$ is slightly more complicated to state.
In Section \ref{sec:thm dbv_hycom2 proof}
below, we construct a map
\[
  \Psi_\wedge^{k}\colon
  H^k(\Feyn(\HyCom)(\!(g,n)\!)) \to \bigoplus_{j=1}^n H^{k+2}(\Feyn(\Com)(\!(g,n)\!)).
\]

The map $\Psi_\wedge^*$ is built from the dual operations to the multiplications with $\psi$-classes at the $n$ markings, identifying $\HyCom(r)=H_\bullet(\MM_{0,r})$.

\begin{thm}\label{thm:dbv_hycom2}
Let $(g,n)\neq (1,1)$ and $n\geq 1$. Then we have that 
\[
  \gr_{6g-8+4n} H^{6g-6+3n-k}(\Feyn_\kk(D\BV^*)(\!(g,n)\!)) 
\cong 
 \ker \Psi_\wedge^{-k} \oplus \coker \Psi_\wedge^{-k-1}.
\]
\end{thm}

In order to obtain the above results, we introduce slightly smaller quasi-isomorphic subcomplexes
\begin{align*} \Feyn'(D\BV^*)(\!(g,n)\!) &\xhookrightarrow{\sim} \Feyn(D\BV^*)(\!(g,n)\!) & \AFeyn'(D\BV^*)(\!(g,n)\!) &\xhookrightarrow{\sim} \AFeyn(D\BV^*)(\!(g,n)\!) \end{align*}
and investigate natural spectral sequences
\begin{align*} H^\bullet(\Feyn'(H^\bullet(D\BV^*))(\!(g,n)\!)) &\Rightarrow H(\Feyn'(D\BV^*)(\!(g,n)\!)) \\
 H^\bullet(\AFeyn'(H^\bullet(D\BV^*))(\!(g,n)\!)) &\Rightarrow H(\AFeyn'(D\BV^*)(\!(g,n)\!)) . \end{align*}

\subsection{Results on \texorpdfstring{$\Feyn(\HyCom)$}{Feyn(HyCom)} and \texorpdfstring{$\Feyn(\BV)$}{Feyn(BV)}}

The above theorems allow us to express the top$-2$-weight part of the cohomology of $\Feyn_\kk(D\BV^*)$ through the weight 2 part of the Feynman transform of the hypercommutative operad.
This latter object can be simplified further as follows.

\begin{thm}\label{thm:hycom bv}
There is an isomorphism
\[
\gr_2H^\bullet(\Feyn(\HyCom^*))(\!(g,n)\!) 
\cong 
\gr_2H^\bullet(\Feyn(\BV^*))(\!(g,n)\!) 
\]
for all $(g,n)\neq (1,1)$.
\end{thm}

We note that $\gr_2\Feyn(\BV^*)$
is a fairly simple graph complex that identically arises in the computation of homotopy groups of embedding spaces \cite{FTW2}.
In particular, for $n=0$ its cohomology may be expressed completely through the cohomology of the commutative graph complex as follows.
\begin{prop}\label{prop:FeynBV_symmetric_product}
Let $\SW_{g}^k$ be the part of total genus $g$ and degree $k$ of the symmetric product 
\[
\SW = \Sym^2\left( \bigoplus_{g\geq 2} H^\bullet\left(\Feyn(\Com^*)(\!(g,1)\!)\right)[-1] \right).
\]
Then there is an isomorphism
\[
\gr_2 H^k(\Feyn(\BV^*)(\!(g,0)\!))
\cong 
H^{k-3}(\Feyn(\Com^*)(\!(g,2)\!)_{asymm})  
\oplus 
\SW_g^{k-1},
\]
with $\Feyn(\Com^*)(\!(g,2)\!)_{asymm}$ the antisymmetric part under the $S_2$-action on the two markings.
\end{prop}

Finally, the cyclic operad $\HyCom$ is the genus zero part of the larger modular operad $H_\bullet(\MM)$ consisting of the homology of the Deligne-Mumford compactifications of the moduli spaces of curves $\MM_{g,n}$.
We may hence use the canonical projection of modular operads $H_\bullet(\MM)\to \HyCom$ to obtain the map
\begin{equation}\label{equ:hycom mm}
    \Feyn(\HyCom^*)\to \Feyn(H^\bullet(\MM)).
\end{equation}
The resulting weight 0 comparison map 
\begin{gather*} H^\bullet(\Feyn(\Com^*))(\!(g,n)\!) \cong \gr_0H^\bullet(\Feyn(\HyCom^*))(\!(g,n)\!) \\
\quad \quad \to \gr_0 H^\bullet(\Feyn(H^\bullet(\MM)))(\!(g,n)\!) \cong \gr_0H_c^\bullet(\M_{g,n}) \end{gather*}
has been shown to be an isomorphism for all $(g,n)\neq (1,1)$ in \cite{CGP1, CGP2}.
Similarly, the weight 2 part of $\Feyn(H^\bullet(\MM))$ (computing $\gr_2 H^\bullet_c(\M_{g,n})$) has been studied by Payne and the third author in \cite{PayneWillwacher}. 
In particular they express the weight 2 cohomology in terms of a graph complex close to the commutative graph complex $\Feyn(\Com^*)$. 
Our methods allow for expressing their results in a more natural form through the cohomology of $\Feyn(\HyCom^*)$, yielding the following comparison result, which is slightly more complicated than its weight 0 analogue:
\begin{prop}[{cf. also \cite[Theorems 1.1 and 1.2]{PayneWillwacher}}]
  \label{prop:HyCom MM comparison}
There is a morphism 
\[
\nabla^k_{g,n}: \gr_2 H^{k-2}(\Feyn(\HyCom^*)(\!(g-1,n)\!)) 
\to  
\gr_2 H^{k}(\Feyn(\HyCom^*)(\!(g,n)\!)) 
\]
such that 
\[
\gr_2H^k(\Feyn(H^\bullet(\MM))(\!(g,n)\!))
\cong 
\mathrm{ker}(\nabla^{k+1}_{g,n})
\oplus
\mathrm{coker}(\nabla^{k}_{g,n})
.
\]
Furthermore, for $n=0$ we have $\nabla_{g,0}^k =0$ for all $g,k$ so that 
\[
\gr_2H^k(\Feyn(H^\bullet(\MM))(\!(g,0)\!))
\cong 
\gr_2 H^{k-1}(\Feyn(\HyCom^*)(\!(g-1,0)\!)) 
\oplus 
\gr_2 H^k(\Feyn(\HyCom^*)(\!(g,0)\!)) 
.
\]
\end{prop}

\subsection{Corollaries for the handlebody groups}

By Giansiracusa's Theorem~\ref{thm:giansiracusa} our results about the Feynman transform of $D\BV^*$ above (Theorem~\ref{thm:main} and Theorem~\ref{thm:dbv_hycom}) yield parallel results about the cohomology of the handlebody groups.

\begin{cor}\label{cor:main hbdy}
    Let $(g,n)\neq (1,0),(1,1),(0,2)$.
\begin{enumerate}
\item The bottom weight part of the cohomology of the handlebody group is equal to the forested graph cohomology.
\[
    \gr_{0} H^\bullet(\HMod_{g,0}^n)
    \cong \gr_{0} H^\bullet(\HMod_{g,n}^0)
    \cong H^\bullet(\Feyn_\kk(D\Com^*))(\!(g,n)\!).
\]
In particular, we have that 
\[
    \gr_{0} H^\bullet(\HMod_{g,0}^0)
    \cong
    H^\bullet(\Out(F_g)).
\]
\item (Hainaut--Petersen \cite{HainautPetersen})
The top-weight part is equal to the cohomology of the Kontsevich commutative graph complex
\begin{align*} \gr_{6g-6+2n} H^{6g-6+2n-k}(\HMod_{g,n}^0) &\cong H^{-k}(\Feyn(\Com)(\!(g,n)\!)) \cong \gr_{6g-6+2n} H^{6g-6+2n-k}(\M_{g,n}) \\
\gr_{6g-6+4n} H^{6g-6+3n -k}(\HMod_{g,0}^n) &\cong H^{-k}(\Feyn(\Com)(\!(g,n)\!))\otimes \sgn_n . \end{align*}
\item The top$-2$-weight part satisfies 
\[
\gr_{6g-8+2n} H^{6g-6+2n-k}(\HMod_{g,n}^0) 
\cong 
\gr_2 H^{-k}(\Feyn(\HyCom)(\!(g,n)\!)).
\]
\end{enumerate}
\end{cor}
Using Proposition~\ref{prop:FeynBV_symmetric_product} we then obtain:
\begin{cor}\label{cor:Hbdy Com comparison}
  For $g\geq 2$ have that 
  \[
  \gr_{6g-8} H^{6g-6-k}(\HMod_{g,0}^0) \cong
  H^{3-k}(\Feyn(\Com)(\!(g,2)\!)_{asymm})  
  \oplus 
  (\SW_g^{k-1})^*,
  \]
  with $\SW_g^{k-1}$ as in Proposition~\ref{prop:FeynBV_symmetric_product}.
\end{cor}

We note that Hirose \cite{Hirose} proved that $\HMod_{g,0}^0$ has vcd $4g-5$ (i.e.~that the cohomology is only supported in degrees $\leq 4g-5$).
Combining Corollary~\ref{cor:Hbdy Com comparison} with known results on the commutative graph cohomology, we obtain a lot of non-vanishing cohomology close to the vcd.

  \begin{cor}\label{cor:hbdy growth}
  The dimension of $H^{4g-5-k}(\HMod_{g,0}^0)$
  grows at least exponentially with $g$ for each 
  $k\in \{1,4,7,10,11,14\}$.
  \end{cor}

We would like to remark that just like the handlebody group $\HMod_{g,0}^0:=\pi_0 \mathrm{Diff}(V_{g,0}^0)$, the mapping class group of an unmarked surface $\Mod_{g,0}^0:=\pi_0 \mathrm{Diff}(\Sigma_{g,0}^0)$ has vcd $4g-5$. Church--Farb--Putman \cite{CFP12} showed that its rational cohomology $H^{\bullet}(\Mod_{g,0}^0)=H^{\bullet}(\M_g)$ vanishes in this vcd. 
It seems to be unknown whether the same is true for the handlebody group, i.e.~whether $H^{4g-5}(\HMod_{g,0}^0) = 0$.
Hence, we do not know whether the classes of the ``codimension-one'' case $k=1$ in Corollary \ref{cor:hbdy growth} are of the highest possible degree.

  Using the results of \cite{PayneWillwacher} we may compare the top$-2$-weight parts of the cohomology of the handlebody group and the mapping class group.
  \begin{cor}\label{cor:Hbdy Mg comparison new}
  For $g\geq 3$ we have that
  \[
      \gr_{6g-8} H^{k}(\M_g)
      \cong 
  \gr_{6g-8} H^k(\HMod_{g,0}^0) 
  \oplus 
  \gr_{6g-14} H^{k-5}(\HMod_{g-1,0}^0) 
  .
  \]
  \end{cor}

\subsection{Euler characteristics}

If an object has non-trivial Euler characteristic, then it needs to have non-trivial (co-)homology in some degree. 
Hence, computing the Euler characteristic can give a comparably accessible way of proving the existence or even quantifying the amount of (co-)homology classes.

For the handlebody group, this approach does not work on the nose because its Euler characteristic $\chi(\HMod_{g,0}^0)$ is equal to $0$ for all $g$. This was shown by Hirose in \cite{Hirose}.
Nonetheless, one can obtain non-trivial information by considering the Euler characteristic of its graded pieces:
In \cite{BV}, it was proven that the Euler characteristic $\chi(\Out(F_g)) = \sum (-1)^k \dim H^k(\Out(F_g))$ behaves as $- e^{-1/4} (g/e)^g/(g\log g)^2$ for large $g$. Hence, the total dimension of $H^\bullet(\Out(F_g))$ grows at least as fast. 
Furthermore, the Euler characteristic of $\Feyn(\Com)(\!(g,0)\!)$ also grows super exponentially \cite{MBeuler}.
By Corollary~\ref{cor:main hbdy}, these are the pieces of $H^\bullet(\HMod_{g,0}^0)$ of weight $0$ and $6g-6$, respectively. Hence, combining the above computations with Hirose's vanishing result for $\chi(\HMod_{g,0}^0)$, we obtain the following.

\begin{cor}
    The total dimension of $H^\bullet(\HMod_{g,0}^0)$ grows super exponentially. More explicitly, for any given $C >0$, the dimension of each of the subspaces $\gr_{0} H^\bullet(\HMod_{g,0}^0), \gr_{6g-6} H^\bullet(\HMod_{g,0}^0)$ and $\bigoplus_{W =2}^{6g-8} \gr_{W} H^\bullet(\HMod_{g,0}^0)$ is larger than $C^g$ for all but finitely many $g \geq 2$.
\end{cor}

Euler characteristic computations are also useful to identify relations between different objects. Using a procedure laid out by Getzler and Kapranov \cite{GK}, we will compute the weight-graded Euler characteristics of all
Feynman transforms listed in Table~\ref{tab:overview} in low genus and weight for $n=0$. The results are listed in Tables~\ref{tab:DBV}-\ref{tab:HyCom}. The weight-graded Euler characteristic of $\Feyn_\kk(D\BV^*)$ is listed in Table~\ref{tab:DBV}.
By Theorem~\ref{thm:giansiracusa} this is also the weight-graded Euler characteristic of the handlebody group. For instance, by summing the absolute values of all numbers in the $g=7$ row of Table~\ref{tab:DBV} we find that $H^\bullet(\HMod_{7,0}^0)$ has dimension at least 230.

\subsection{Acknowledgements}
We are indebted to Danica Kosanovi\'c with whom MB and BB collaborated during early stages of this project.
We would like to thank Sebastian Hensel and Oscar Randal-Williams for pointers to the literature on handlebody groups and Peter Feller for helpful conversations about mapping class groups. We also thank Dan Petersen for sharing an early version of \cite{HainautPetersen} with us.
MB was supported by Dr.\ Max Rössler, the Walter Haefner Foundation and the ETH Zürich Foundation.
TW was supported by the NCCR Swissmap, funded by the Swiss National Science Foundation.

\section{Recollections and setup}

\subsection{Basic notation}
We generally work with $\Q$-vector spaces, and (co)homology is taken with rational coefficients without further notice.
All graded vector spaces $V$ are $\Z$-graded and for a homogeneous element $v\in V$ we denote by $|v|$ the degree. We denote the $k$-fold downwards degree shift by $V[k]$. For example, if $V$ is concentrated in degree 0 then $V[k]$ is concentrated in degree $-k$.
We follow cohomological conventions unless otherwise noted. In other words, the differentials on differential graded (dg) vector spaces have degree $+1$.
In particular, this convention requires that dualization reverses degrees. That is, let $V$ be a graded vector space. We denote by $V^k\subset V$ the subspace of homogeneous elements of degree $k$.
Then we define the dual graded vector space
\[
  V^* = \bigoplus_k (V^{-k})^*,
\]
and in particular $(V^*)^k=(V^{-k})^*$.
We denote the $k$-th cohomology of differential graded vector space $(V,d)$ by $H^k(V)$ or $H^k(V,d)$.
The total cohomology is denoted by 
\[
H(V) =H^\bullet(V) = \bigoplus_{k}H^k(V,d).  
\]
Note in particular that we will henceforth omit the ``$\bullet$'', as long as no ambiguity arises.
Besides the cohomological grading most of our vector spaces will be equipped with a second (``weight'') grading.
By convention, we do not negate the weight grading upon dualization, and our weights will always be non-negative. 

For $V$ a graded vector space and $A$ a finite set of cardinality $|A|=n$ we define the set-wise tensor product 
\[
V^{\otimes A} = \left(\bigoplus_{f:A\xrightarrow{\cong} \{1,\dots,n\}} 
\underbrace{V\otimes \cdots \otimes V}_{n\times} \right)_{S_n},
\]
where the symmetric group $S_n$ acts by permuting the direct summands and simultaneously the factors of $V$ in the tensor product.

\subsection{A lemma for dg vector spaces}
We will need the following basic lemma from homological algebra.

\begin{lemma}\label{lem:simple tri}
Let $(V,d)$ be a dg vector space with a decomposition of graded vector spaces $V=V_1\oplus V_2$. Suppose that the differential has a lower triangular form with respect to this decomposition,
\[
d = \begin{pmatrix}
  d_1 & 0  \\ f & d_2 
\end{pmatrix},
\] 
so that 
\[
(V,d)  =
\begin{tikzcd}[column sep = 0em]
  V_{1} \ar[bend left]{rr}{f} \ar[loop above]{}{d_1} & \oplus  & V_{2}\ar[loop above]{}{d_2}
 \end{tikzcd}.
\]
Then $d_1$ is a differential on $V_1$, $d_2$ is a differential on $V_2$ and $f:(V_1,d_1)\to (V_2[1],-d_2)$ is a morphism of dg vector spaces. Furthermore, we have that 
\begin{equation}\label{equ:simp tri}
H^k(V) \cong \ker [f]^k \oplus \coker [f]^{k-1},  
\end{equation}
with 
\[
[f]^k : H^k(V_1,d_1) \to H^{k+1}(V_2,d_2)  
\]
the induced map on cohomology.
\end{lemma}
\begin{proof}
The statements on $d_1$, $d_2$ and $f$ are just a component-wise rewriting of the equation $d^2=0$.
For Equation \eqref{equ:simp tri} we consider the short exact sequence of dg vector spaces 
\[
0\to V_2\to V\to V_1\to 0.  
\]
It induces a long exact sequence on cohomology groups 
\[
\cdots \to H^{k-1}(V_1)\xrightarrow{[f]^{k-1}}
H^k(V_2) \to H^k(V)\to H^k(V_1) \xrightarrow{[f]^{k}}
H^{k+1}(V_2) \to \cdots,   
\]
from which \eqref{equ:simp tri} follows by choosing a splitting.
\end{proof}

\subsection{Operads and cyclic operads}
A symmetric sequence $\cP$ is a collection of right modules $\cP(r)$ for the symmetric groups $S_r$ for each $r=1,2,\dots$. The number $r$ is called the arity.

The category of symmetric sequences is equipped with a monoidal product (``plethysm'') $\circ$, and a (unital) operad is a monoid in this monoidal category, see \cite{LodayVallette}.
Concretely, an operad is determined by a unit element $1\in \cP(1)$ and partial composition morphisms 
\[
\circ_j : \cP(r)\otimes \cP(s) \to \cP(r+s-1)  
\]
for $j=1,\dots, r$, satisfying suitable compatibility relations.
Intuitively, one thinks of $\cP(r)$ as abstract avatar of a space of functions with $r$ inputs, and $\circ_j$ is the abstract version of composition of such functions at the $j$-th input.
A pseudo-operad \cite{MMS} is the data of a symmetric sequence $\cP$ together with the composition morphisms $\circ_j$, but without the unit element.

It is often more convenient to label symmetric sequences by sets rather than numbers. Concretely, for a set $A$ of cardinality $|A|=r$ we set 
\[
\cP(A) := \left(\bigoplus_{f: A\xrightarrow{\cong}\{1,\dots,r\}} \cP(r)  \right)_{S_r},
\]
where the direct sum is over bijections from $A$ to $\{1,\dots,r\}$ and the symmetric group acts diagonally on the set of such bijections and on $\cP(r)$.
Intuitively, if we think of $\cP(r)$ as functions with $r$ inputs, then $\cP(A)\cong \cP(r)$ is the same set of functions, but with their inputs labelled by the set $A$ instead of numbers $1,\dots,r$.
In particular, using this notation we may replace the composition morphisms $\circ_j$ by an equivalent set of morphisms 
\[
\circ_a : \cP(A) \otimes \cP(B) \to \cP(A\sqcup B\setminus\{a\})  
\]
for $A$, $B$ finite sets and $a\in A$.

Cyclic operads are operads for which the action of $S_r$ on the space $\cP(r)$ can be extended to an action of $S_{r+1}$, in a compatible manner \cite{GKcyclic}.
It is convenient to use the double-parenthesis-notation 
\[
\cP(\!(r+1)\!) := \cP(r)
\]
to remind of this extended symmetric group action, and call $\cP$ a cyclic sequence.
Similarly to above, we also use the notation $\cP(\!(A)\!)$ for $A$ a set of cardinality $r+1$.

For a cyclic operad $\cP$ it is convenient to replace the composition morphisms $\circ_j$ by an equivalent set of morphisms 
\[
  \circ_{a,b} : \cP(\!(A)\!)\otimes \cP(\!(B)\!)
  \to \cP(\!(A\sqcup B \setminus \{a,b\})\!)
\]
for $A$, $B$ finite sets with $a\in A$, $b\in B$.
A cyclic pseudo-operad is a cyclic sequence $\cP$ with the composition morphisms $\circ_{a,b}$, but without the unit element.

\subsection{Modular operads}
We define a modular sequence $\cP$ to be a collection of right $S_r$-modules 
\[
  \cP(\!(g,r)\!)
\]
for each $g\geq 0$ and $r\geq 0$ such that $2g+r\geq 2$.
We also use set-wise indexing of modular sequences, analogous to the case of symmetric sequences above, and write 
\[
  \cP(\!(g,A)\!) :=
  \left(\bigoplus_{f: A\xrightarrow{\cong}\{1,\dots,r\}} \cP(\!(g,r)\!)  \right)_{S_r} \cong   \cP(\!(g,r)\!)  
\]
for $A$ any finite set of cardinality $|A|=r$.

A modular operad is modular sequence $\cP$ together with composition morphisms 
\begin{align*} \circ_{i,j} &\colon \cP(\!(g,r)\!) \otimes \cP(\!(g',r')\!) \to \cP(\!(g+g',r+r'-2)\!) \\
\eta_{ij}&\colon \cP(\!(g,r)\!)\to \cP(\!(g+1,r-2)\!), \end{align*}
satisfying suitable compatibility relations.
For the precise formulation we refer to \cite{GK,Ward}.
From there we will also use the notion of $\DD$-modular operads, for $\DD$ a hyperoperad.

Any cyclic pseudo-operad $\cP$ is a modular operad by setting 
\[
\cP(\!(g,r)\!) := \begin{cases}
  \cP(\!(r)\!) &\text{for $g=0$} \\
  0 &\text{otherwise}
\end{cases},
\]
and defining the morphisms $\eta_{i,j}$ to be zero.

Finally, the above definition can be dualized to yield the notion of modular cooperad.
We use the following notation for the duals of the operations $\circ_{ij}$ and $\eta_{ij}$.
Let $A$ be a finite set and $A=A_1\sqcup A_2$ a decomposition of $A$ into two subsets $A_1,A_2\subset A$.
The modular cooperad $\cC$ is a modular sequence equipped with morphisms 
\begin{align*} \Delta_{h,A_1} &\colon \cC(\!(g,A)\!) \to \cC(\!(h,A_1\sqcup \{*\})\!)\otimes \cC(\!(g-h,A_2\sqcup \{*'\})\!) \\
\eta^* &\colon \cC(\!(g,A)\!) \to \cC(\!(g-1,A\sqcup \{*,*'\})\!) \end{align*}
satisfying suitable compatibility relations.

\subsection{Vector spaces of decorated graphs}
We use the Feynman transform introduced by Getzler--Kapranov \cite{GK}. A full treatment of the Feynman transform is outside of the scope of this paper, and we refer to op.\ cit.\ for that purpose.
However, we shall provide a short definition suitable for our needs, neglecting the modular operadic structures present on the Feynman transform. We will slightly deviate from the conventions of Getzler--Kapranov in one aspect.

We say that a modular graph with $n$ legs is a connected graph with $n$ legs together with the data of (i) a labeling of the legs by numbers $1,\dots,n$ and (ii) for every vertex $v$ a non-negative integer $g_v$, the genus of that vertex.
For $\gamma$ a modular graph we define its genus to be the number 
\[
g(\gamma)=\text{\#loops} + \sum_{v\in V\gamma} g_v,  
\]
and the arity of $\gamma$ to be $n(\gamma)=n$, the number of external legs.
Given any modular sequence $\cM$ and a modular graph $\gamma$ we set 
\[
  \otimes_{\gamma} \cM := \bigotimes_{v\in V\gamma}
  \cM(\!(g_v,\gst(v)\!)),
\]
where $\gst(v)$ is the set of half-edges incident at $v$.

For $\cM$ any modular sequence and $p$ an integer we define the modular sequence $\Free_p(\cM)$ such that 
\begin{equation}\label{equ:free1 def}
  \Free_p(\cM)(\!(g,n)\!)
  =\left(
  \bigoplus_{\substack{\gamma \\ g(\gamma)=g \\ n(\gamma)=n} }
   \left(\otimes_\gamma \cM\right)
  \otimes \Det_p(E\gamma)
  \right)/\sim
\end{equation}
where:
\begin{itemize}
\item The direct sum is over all modular graphs $\gamma$ of genus $g$ and arity $n$.
\item $E\gamma$ is the set of (non-leg-)edges of $\gamma$.
\item For a finite set $A$ we define the one-dimensional graded vector space $\Det_p(A):= \Q[-p]^{\otimes A}$, concentrated in cohomological degree $p|A|$.  
\item To specify the equivalence relation $\sim$ we use the following notation. A generating element of \eqref{equ:free1 def} in the $\gamma$-summand we denote by $(\gamma,m,o)$ with $m\in \otimes_\gamma \cM$ and $o\in \Det(E\gamma)$. 
Furthermore note that any isomorphism $\phi:\gamma\to \gamma'$ of modular graphs induces isomorphisms $\phi_* : \otimes_\gamma \cM\to \otimes_{\gamma'} \cM$ and 
$\phi_*: \Det_p(E\gamma)\to \Det_p(E\gamma')$.
Then the equivalence relations are generated by 
\[
(\gamma,m,o) \sim (\gamma' ,\phi_* m,\phi_* o),  
\]
for any isomorphism of modular graphs $\phi:\gamma\to \gamma'$. 
\end{itemize}
We think of a triple $\Gamma=(\gamma,m,o)$ 
as a decorated graph, with $m$ %
specifying decorations of the vertices by $\cM$.
We will sometimes abuse notation below and call a triple $\Gamma=(\gamma,m,o)$ a (decorated) graph.
The number $p$ is understood as the cohomological degree carried by an edge.

We endow $\Free_p$ with the differential 
\[
d = d_{\cM}  
\]
induced from the differential on $\cM$.
Note that
\[
  \Free(\cM) := \Free_0(\cM)
\]
can be identified with either the free modular operad generated by $\cM$, or the cofree modular cooperad cogenerated by $\cM$, hence the notation.

For later use we also introduce a technical variant of the above construction.
We say that a modular graph with marked legs is a modular graph $\gamma$ together with a subset $M\gamma\subset \{1,\dots,n\}$ of its legs, that we consider marked.
For a modular sequence $\cM$ we then define the modular sequence  
\begin{equation}\label{equ:mfree1 def}
  \MFree_p(\cM)(\!(g,n)\!)
  =\left(
  \bigoplus_{\substack{\gamma \\ g(\gamma)=g \\ n(\gamma)=n} }
 \left(\otimes_\gamma \cM\right)
  \otimes \Det_p(E\gamma\sqcup M\gamma)
  \right)/\sim,
\end{equation}
where the direct sum is now over all modular graphs with marked legs, and otherwise we are using the notation from above and the analogously defined equivalence relation.
As automorphisms preserve legs, they also preserve the set of marked legs.
Hence,
$\MFree_p(\cM)(\!(g,n)\!)$ has precisely $2^{n}$ generators for 
each generator of 
  $\Free_p(\cM)(\!(g,n)\!)$.
Again we equip $\MFree_p(\cM)(\!(g,n)\!)$ with the differential $d=d_{\cM}$ induced from the differential on $\cM$.

\subsection{Feynman transform}\label{sec:feyn}
Let $\cC$ be a modular cooperad. For example, as explained above any cyclic pseudo-cooperad is a modular cooperad.
Then we define the Feynman transform of $\cC$ as the modular sequence 
\[
  \Feyn(\cC) = \Free_1(\cC), 
\]
but with an altered differential reflecting the modular cooperad structure
\[
d = d_{\cC} + d_{s} + d_{\ell}.
\]
The term $d_{\cC}$ is induced by the differential on $\cC$. 
The term $d_s$ acts by splitting vertices, and the term $d_{\ell}$ introduces a tadpole at one vertex, i.e., an edge connecting the vertex to itself.
Concretely,
\[
d_s(\gamma, c, o)
= 
\frac 12
\sum_{v\in V\gamma}
\sum_{A\subset \gst(v)}
\sum_{h=0}^{g_v}
(-1)^{|c|}
(\gamma_{v,h,A}, \Delta^v_{h,A}c, e\wedge o),
\]
with $g_v$ the genus of vertex $v$ and $\gst(v)$ the set of incident half-edges at $v$, $\gamma_{v,h,A}$ obtained by splitting $v$
\[
  \gamma:
\begin{tikzpicture} \node[ext, label=90:{$v$}] (v) at (0,0) {$\scriptstyle g_v$}; \draw (v) edge +(-.5,0) edge +(-.5,-.5) edge +(-.5,.5) edge +(.5,0) edge +(.5,-.5) edge +(.5,.5); \end{tikzpicture}
\quad\quad\to\quad\quad
\gamma_{v,h,A}:
\begin{tikzpicture} \node[ext] (v1) at (0,0) {$\scriptstyle h$}; \node[ext] (v2) at (1.2,0) {$\scriptscriptstyle g_v-h$}; \draw (v1) edge +(-.5,0) edge +(-.5,-.5) edge +(-.5,.5) (v2) edge +(.5,0) edge +(.5,-.5) edge +(.5,.5) edge node[below] {$ e$} (v1); \draw[pbrace, thick] (-.55,-.5) -- (-.55,.5); \node at (-.9,0) {$A$}; \end{tikzpicture},
\]
$\Delta^v_{h,A}$ is the cocomposition $\Delta_{h,A}$ applied to the $v$-tensor factor in $c$, and $e\wedge o$ is the natural element of $\Det_1(E\gamma_{v,h,A})$ obtained by attaching a factor corresponding to the new edge to $o$ from the left.
Similarly, 
\[
d_\ell (\gamma,c,o)
=
\sum_{v\in V\gamma}
(-1)^{|c|}
(\gamma_{v}, \eta^{*,v}c, e\wedge o),
\]
with 
\[
  \gamma_v : \begin{tikzpicture} \node[ext, label=-90:{$v$}] (v) at (0,0) {$\scriptscriptstyle g_v-1$}; \draw (v) edge +(-.5,0) edge +(-.5,-.5) edge +(-.5,.5) edge +(.5,0) edge +(.5,-.5) edge +(.5,.5) edge [loop above] (v); \end{tikzpicture}
\]
and $\eta^{*,v}$ being the operation $\eta^*$ applied to the tensor factor corresponding to $v$.

Similarly, let $\cC$ be a 1-shifted modular cooperad, or a $\kk$-modular cooperad in the notation of Getzler--Kapranov. 
This means that the structure operations raise the cohomological degree by one.
Then we define 
\[
\Feyn_\kk(\cC) = \Free_0(\cC),
\]
to be the free modular operad generated by $\cC$, equipped with the differential $d = d_{\cC} + d_{s} + d_{\ell}$ defined analogously to above.

Next let $\cP$ be a modular operad.
Then we define the Feynman transform of $\cP$ to be the modular sequence 
\[
\Feyn(\cP):= \Free_{-1}(\cP)  
\]
with the differential 
\[
d = d_{\cP} + d_{c} + d_{t}.
\]
The term $d_{\cP}$ is induced by the differential on $\cP$. 
The term $d_c$ dual to $d_s$ acts by contracting an edge, and the term $d_{t}$ is dual to $d_\ell$ and removes a tadpole at one vertex, i.e., an edge connecting the vertex to itself.

Similarly, for $\cP$ a 1-shifted modular operad, meaning the structure operations have degree $+1$, we set 
\[
  \Feyn(\cP):= \Free_{0}(\cP) ,
\]
with the analogously defined differential $d = d_{\cP} + d_{c} + d_{t}$.

Note that we abuse the notation, using $\Feyn(-)$ to denote the Feynman transform of both modular operads and cooperads.
Note also that here we deviate from the conventions of Getzler and Kapranov \cite{GK}.
First, our modular (co)operads can have unstable operations, and second our definition of the Feynman transform does not involve a dualization.
These two deviations are connected: Allowing unstable vertices means that the set of modular graphs of fixed arity and genus is infinite, which is why one needs to be careful with dualization.

\begin{convention}
We will often apply the Feynman transform to augmented cyclic operads $\cP$ or coaugmented cooperads $\cC$.
In this case we define 
\begin{align*} \Feyn(\cP)&:= \Feyn(\overline \cP) & \Feyn(\cC)&:= \Feyn(\overline \cC) \end{align*}
as the Feynman transform of the augmentation ideal
$\overline \cP$ of $\cP$ (respectively, the coaugmentation coideal $\overline \cC$ of $\cC$), understood as a pseudo-(co)operad and hence a modular (co)operad.
We use the analogous convention also for the $1$-shifted version $\Feyn_{\kk}(-)$.

We also note that for cyclic operads $\cP$ such that the part of arity one and degree zero is one-dimensional, i.e., $\cP(1)^0\cong \Q$, we have that 
\[
  \overline \cP(r)^k =
  \begin{cases}
    \cP(r)^k & \text{if $r\neq 1$ or $k\neq 0$,} \\
    0 & \text{if $r= 1$ and $k= 0$.}
  \end{cases}
\]
Since all our cyclic operads will satisfy $\cP(1)^0\cong \Q$ this can be taken as the definition of $\overline \cP$. The same formula may also be used to define the coaugmentation coideal $\overline \cC$.
\end{convention}

\begin{remark}\label{rem:feyn minus diff}
The signs in the definition of the differential are to some degree conventional.
First, note that any dg vector space $(V,d)$ is isomorphic to $(V,-d)$, the isomorphism sending a homogeneous element $v$ to $(-1)^{|v|}v$.
Furthermore, in the definition of the differential $d$ on the Feynman transform we may change the relative sign of the terms to 
\[
\tilde d= d_{\cP}-d_c-d_t.  
\]
The isomorphism relating $d$ and $\tilde d$ in this case sends a graph $\Gamma$ with $e$ edges to $(-1)^e\Gamma$.
\end{remark}

\subsection{Feynman transform for weight-graded cooperads}

In particular, we will consider the situation of a 1-shifted-modular pseudo-cooperad $\COp$ that is equipped with an additional $\Z$-grading.
We call this grading a weight grading if the following conditions are satisfied:
\begin{itemize}
\item The grading is non-negative, and positive on $\COp(\!(0,2)\!)$.
\item The grading is (additively) compatible with the modular cooperad structure.
\item Each graded piece of fixed arity $\gr_W \COp(\!(g,n)\!)$ is finite dimensional.
\end{itemize}

We note that the conditions ensure that the Feynman transform of $\COp$ inherits the weight grading from $\COp$, and that each weight-graded piece $\gr_W\Feyn_{\kk}(\COp)(\!(g,n)\!)$ is finite dimensional as well.
We then define 
\[
I_{3} \subset \Feyn_{\kk} (\COp)(\!(0,2)\!)
\]
to be the subspace of elements of weight $\geq 3$, and we define the quotient\footnote{The number $3$ here is somewhat arbitrary, and is chosen with our applications in mind.}
\[
\Feyn_{\kk}'(\COp) := \Feyn_{\kk}(\COp) / \langle I_3\rangle
\]
by the modular operadic ideal generated by $I_3$.
Note that the differential on the Feynman transform descends to the quotient by compatibility of the weight grading with the cooperad structure.

\begin{lemma}\label{lem:reduced Feyn}
Let $\COp$ be a coaugmented weight-graded 1-shifted-modular cooperad as above. 
Suppose that $I_3$ is acyclic, i.e., $H(I_3)=0$.
Then the natural morphism of modular operads 
\begin{equation*}%
    \Feyn_{\kk}(\COp)\to \Feyn_{\kk}'(\COp)    
\end{equation*}
is a quasi-isomorphism in each genus and arity $(g,n)$ except possibly $(g,n)=(1,0)$.
\end{lemma}
\begin{proof}
We have to check that the ideal $\langle I_3\rangle$ is acyclic.
Recall that elements of $\Feyn_{\kk}(\COp)$ are linear combinations of $\cC$-decorated graphs, possibly with bivalent vertices.
The differential has the form $d=d_{\cC}+d_{s}+d_\ell$, see above.
We say that a vertex $v$ of genus $g_v$ and valence $n_v$ is a $(g_v,n_v)$-vertex.
The ideal $\langle I_3\rangle $ is spanned by all graphs that have a consecutive chain of $(0,2)$-vertices of total weight $\geq 3$.
To show that $H(\langle I_3\rangle)=0$ it suffices to endow $\langle I_3\rangle$ with a filtration such that the associated graded complex is acyclic.
Note that convergence of the spectral sequence here is automatic by finite dimensionality for each fixed combination of weight, genus and arity $(g,n)$.

We filter $\langle I_3\rangle$ by the total weight on non-$(0,2)$-vertices, minus the total valence of non-$(0,2)$-vertices.
The differential on the associated graded complex is then $d_{\cC}+d_s'$, with $d_s'$ the part of the vertex splitting differential that only splits $(0,2)$-vertices.
Here we use in particular that the weight grading on $\COp(\!(0,2)\!)$ is positive.  
The associated graded complex is then a direct summand of a tensor product of complexes, one for each non-$(0,2)$-vertex and one for each chain of $(0,2)$-vertices. Of the latter complexes at least one factor is $I_3$, which is acyclic, hence so is the total complex. 
\end{proof}

\subsection{Cyclic (co)bar construction and dg dual (co)operads}
\label{sec:dg dual def}
For a cyclic operad or cooperad $\cP$
we introduce the 1-shifted cyclic pseudo-cooperad or pseudo-operad $D\cP$, which is given by the genus zero part of the Feynman transform of $\cP$,
\[
  D\cP(\!(n)\!) := \Feyn(\cP)(\!(0,n)\!).
\]
The operation $D(-)$ should be considered as the cyclic version of the (co)operadic (co)bar construction.
We will also call $D\cP$ the dg dual cooperad respectively operad of $\cP$, following earlier literature. However, this terminology is not optimal, and prone to confusion with the linear dual $\cP$, in our opinion.

\subsection{Amputated Feynman transform}\label{sec:amputated}

Let again $\COp$ be a 1-shifted modular cooperad with operations of arity $2$, and consider the Feynman transform $\Feyn_\kk(\COp)$. 
For $(g,r)\neq (0,2)$ let $B(\!(g,r)\!)\subset \Feyn_\kk(\COp)(\!(g,r)\!)$ be the dg subspace spanned by all elements obtained by the modular composition $\circ_{i,j}$ of an element in $\Feyn_\kk(\COp)(\!(g,r)\!)$ with an element in $\Feyn_\kk(\COp)(\!(0,2)\!)$.
If we think of elements of $ \Feyn_{\kk}(\COp)(\!(g,r)\!)$ as linear combinations of $\COp$-decorated graphs with $r$ legs, then $B(\!(g,r)\!)$ is spanned by linear combinations of decorated graphs that have at least one leg attached to a $(0,2)$-vertex, for example
\[
  \begin{tikzpicture} \node[ext] (v1) at (-.5,.5) {}; \node[ext] (v2) at (0,.5) {}; \node[ext] (v3) at (.7,.5) {}; \node[ext] (v4) at (0,-.3) {}; \node (w1) at (-1.2,.5) {$\scriptstyle 1$}; \node (w2) at (1.2,.5) {$\scriptstyle 2$}; \node (w3) at (0,-1) {$\scriptstyle 3$}; \draw (v1) edge (v2) edge (w1) (v3) edge (v2) edge (v4) edge (w2) (v4) edge (w3) edge (v2); \end{tikzpicture}
  \in B(\!(1,3)\!),
  \quad\quad\quad\quad
  \begin{tikzpicture} \node[ext] (v1) at (-.7,.5) {}; \node[ext] (v2) at (0,.5) {}; \node[ext] (v3) at (.7,.5) {}; \node[ext] (v4) at (0,-.3) {}; \node (w1) at (-1.2,.5) {$\scriptstyle 1$}; \node (w2) at (1.2,.5) {$\scriptstyle 2$}; \node (w3) at (0,-1) {$\scriptstyle 3$}; \draw (v1) edge (v2) edge (v4) edge (w1) (v3) edge (v2) edge (v4) edge (w2) (v4) edge (w3); \end{tikzpicture}
  \notin B(\!(1,3)\!).
\]
Note in particular that the $B(\!(g,r)\!)$ together do not form a modular operadic ideal. By convention, we also set $B(\!(0,2)\!):=\COp(\!(0,2)\!)$.

Then we define the amputated Feynman transform as 
\[
\AFeyn_\kk(\COp)(\!(g,r)\!) :=   \Feyn_\kk(\COp)(\!(g,r)\!)/B(\!(g,r)\!).
\]
This is not a modular operad in general.
In other words, $\AFeyn_{\kk}(\COp)(\!(g,r)\!)$ is spanned by just those (``amputated'') graphs, that do not have a leg attached to a bivalent genus zero vertex.

Likewise, we also define 
\[
  \AFeyn_{\kk}'(\COp)(\!(g,r)\!)
  := 
  \Feyn_{\kk}'(\COp)(\!(g,r)\!)/(\langle I_3\rangle(\!(g,r)\!) + B(\!(g,r)\!) ).
\]

\section{Examples of cyclic and modular (co)operads}

\subsection{Hypercommutative and gravity operad}\label{sec:hypercom_grav_basics}
The hypercommutative operad is the cyclic operad such that 
\[
  \HyCom(\!(r)\!) = H_{\bullet}(\MM_{0,r}).
\]
The operadic structure is given by the pushforward along the natural gluing operations 
\[
  \MM_{0,r} \times \MM_{0,s} \to \MM_{0,r+s-2}.
\]
We denote the linear dual cyclic cooperad by $\HyCom^*$, such that 
\[
  \HyCom^*(\!(r)\!) = H^{\bullet}(\MM_{0,r}).
\]
We also consider the cyclic operadic cobar construction, i.e., the genus zero part of the Feynman transform
\[
D\HyCom^* (\!(r)\!) = \Feyn(\HyCom^*)(\!(0,r)\!).
\]
This is a 1-shifted cyclic pseudo-operad.
Its cohomology is the (degree shifted) gravity (pseudo-)operad 
\[
\tGrav :=  H(D\HyCom^* ).
\]
E. Getzler has shown \cite{Getzler0} that $D\HyCom^*$ is formal, that is, $D\HyCom^*\simeq \tGrav$, and 
\[
\tGrav(\!(r)\!) \cong  H^{\bullet}_c(\M_{0,r}).
\]

\begin{remark}
We use different sign and degree conventions for the gravity operad than other literature. One can obtain from the one-shifted cyclic pseudo-operad $\tGrav$ the standard gravity (cyclic pseudo-) operad $\Grav$ by a degree shift, namely 
\[
\Grav(\!(r)\!) := \tGrav(\!(r)\!)[r-3]\otimes \sgn_r .
\]
In this paper we shall not use $\Grav$, only $\tGrav$.
\end{remark}

\subsection{The hypercommutative operad \texorpdfstring{$\HyCom$}{HyCom} in weights \texorpdfstring{$\leq 2$}{<=2}}
\label{sec:hycom1}
We endow the hypercommutative operad $\HyCom$ with a non-negative weight grading, equal to the homological degree, or equivalently minus the cohomological degree. That is, we have
\begin{equation*}
	\gr_k \HyCom(\!(r)\!)^l = \begin{cases}
	H_{k}(\MM_{0,r}) & \text{if $l=k$} \\
	0 & \text{else.}
	\end{cases}
\end{equation*}

Here we recall in particular an explicit combinatorial description of $\HyCom$ in weights $\leq 2$.

First, the weight 0 just equals the commutative cyclic operad.
\[
  \gr_0\HyCom = \Com.
\]
The weight 1 part vanishes, and the weight 2 part can be explicitly described, see \cite{Getzler0, ArbarelloCornalba98}. We recall here the description.
Consider first the dual object $\gr_2\HyCom^*(\!(n)\!)$.
This vector space is generated by symbols $\delta_A=\delta_{A^c}$, for $A\subset \{1,\dots,n\}$, with $2\leq |A|\leq n-2$.
For three distinct elements $i,j,k\in \{1,\dots,n\}$ we set 
\[
\psi_{i;jk} := \sum_{A\atop {i\in A \atop j,k\notin A}}\delta_A.  
\]
Then the Keel relations state that $\psi_{i;jk}$ is independent of the choice of $j,k$, i.e., 
\[
  \psi_{i;jk} = \psi_{i;j'k'}
\]
for all $j'$, $k'$ such that $i,j',k'$ are distinct.
It is sufficient to require these relations for one particular choice of $i$, but all $j,k,j',k'$.
Since $\psi_{i;jk}$ is independent of $j$, $k$ we define the $\psi$ classes 
\[
\psi_i :=  \psi_{i;jk} 
\]
for an arbitrary choice of $j,k\neq i$. Under the identification $\HyCom^*(\!(n)\!) = H^{\bullet}(\MM_{0,n})$, the classes $\psi_i$ are given by the usual $\psi$ classes in $H^2(\MM_{0,n})$, i.e.~the first Chern classes of the $n$ cotangent line bundles over $\MM_{0,n}$.

We dualize this description to $\gr_2\HyCom$.
In fact, we will break the cyclic invariance and describe the underlying non-cyclic object.
We set 
\[
  \gr_2\HyCom(n) = \gr_2\HyCom(\!(\{0,\dots,n\})\!),
\]
and more generally, for a finite set $S$,
\[
  \gr_2\HyCom(S) = \gr_2\HyCom(\!(\{0\}\sqcup S)\!).
\]
Elements of $\gr_2\HyCom(S)$ are formal linear combinations
\[
\sum_{A\subset S }
c_A \delta_A^* 
\]
such that $c_A=0$ if $|A|\leq 1$ or $A=S$ and
\begin{equation}\label{equ:Hycom condition}
 \sum_{A\atop j,k\in A} c_A  =
 \sum_{A\atop j',k'\in A} c_A  
\end{equation}
for each choice of distinct elements $j,k\in S$ and $j',k'\in S$.
Note that we break here the $S_{n+1}$-equivariance, arbitrarily singling out the input ``0''. While this might seem less appealing, it will be convenient later.

Below we shall need the following technical lemma.

\begin{lemma}\label{lem:hycom technical}
Let $\underline c=(c_A)_{A\subset \{1,\dots,n\}}$ be a collection of numbers satisfying \eqref{equ:Hycom condition}.
Then for any fixed $i,j,k\in \{1,\dots,n\}$ such that $i\neq j,k$ the following hold:
\begin{enumerate}
\item 
\[
  \sum_{A\atop {i,j \in A \atop k\in A^c}} c_A 
  = 
\sum_{A\atop {i,k \in A \atop j\in A^c}} c_A.
\] 
\item 
\[
  \sum_{A\atop {i \in A \atop j\in A^c}} c_A
  =
  \sum_{A\atop {i \in A \atop k\in A^c}} c_A
\] 
\end{enumerate}
\end{lemma}
\begin{proof}
We compute, using \eqref{equ:Hycom condition} for the second equality:
\begin{align*} &\sum_{A\atop {i,j \in A \atop k\in A^c}} c_A = \sum_{A\atop {i,j \in A}} c_A - \sum_{A\atop {i,j,k \in A}} c_A = \sum_{A\atop {i,k \in A}} c_A - \sum_{A\atop {i,j,k \in A}} c_A = \sum_{A\atop {i,k \in A \atop j\in A^c}} c_A. \end{align*}

For the second assertion we proceed similarly, using the first assertion in the second equality:
\[
  \sum_{A\atop {i \in A \atop j\in A^c}} c_A
  =
  \sum_{A\atop {i,k \in A \atop j\in A^c}} c_A
  +
  \sum_{A\atop {i \in A \atop j,k\in A^c}} c_A
  =
  \sum_{A\atop {i,j \in A \atop k\in A^c}} c_A
  +
  \sum_{A\atop {i \in A \atop j,k\in A^c}} c_A
  =
  \sum_{A\atop {i \in A \atop k\in A^c}} c_A.
\]
\end{proof}

In this language the $\psi$-classes can be identified with the following linear functions
\begin{equation}\label{equ:psistar def}
\begin{gathered}
\psi_i : \gr_2\HyCom(\!(n)\!) \to \Q \\
\psi_i(\sum_{A\subset S }
c_A \delta_A^* ) =: \psi_i(\underline c) =
\begin{cases}
 \sum_{A\atop j,k\in A} c_A & \text{if $i=0$} \\
 \sum_{A\atop {i\in A \atop j,k\notin A}} c_A & \text{if $i\neq 0$}
\end{cases}
\end{gathered}
\end{equation}
for $i=1,\dots,n$, with $j,k\neq i$ arbitrary. Note that it follows from \eqref{equ:Hycom condition} and Lemma \ref{lem:hycom technical} that $\psi_i$ is well-defined, i.e., independent of the choice of $j,k$.

Finally we consider the operadic structure. In weights $\leq 2$ this means we only need to describe the $\Com$-module structure on $\gr_2\HyCom$. This has been described explicitly in \cite[Lemma 3.3]{ArbarelloCornalba98}.

The composition is
\begin{equation}\label{equ:Hycom com action}
\begin{gathered}
\circ_s : \gr_2\HyCom(S\sqcup \{s\} )\cong 
\gr_2\HyCom(S\sqcup \{s\}) \otimes \Com(S') \to \gr_2\HyCom(S\sqcup S')
\\
\sum_{A\subset S\sqcup \{s\} }
c_A \delta_A^* 
\mapsto 
\sum_{A\subset S}
c_{A\sqcup \{s\}} \delta_{A\sqcup S'}^* 
+
\sum_{A\subset S}
c_A \delta_A^* 
-
\psi_s(\underline c)
\delta_{S'} 
\\
\circ_s : \gr_2\HyCom(S)\cong 
\Com(S'\sqcup \{s\})
\otimes
\gr_2\HyCom(S)  \to \gr_2\HyCom(S\sqcup S')
\\
\sum_{A\subset S\sqcup \{s\} }
c_A \delta_A^* 
\mapsto 
\sum_{A\subset S}
c_{A} \delta_{A}^* 
-
\psi_0(\underline c)
\delta_{S} 
\end{gathered}.
\end{equation}

\subsection{The Batalin-Vilkovisky operad \texorpdfstring{$\BV$}{BV}}
\label{sec:BV}

The Batalin-Vilkovisky operad is the homology operad $\BV = H_\bullet(\LD_2^{fr})$ of the framed little disks operad.
As an operad, it has the following presentation.
\begin{itemize}
\item It is generated by:\footnote{Note that we are using cohomological degree conventions, and the cohomological degrees are minus the homological ones.}
\begin{itemize}
  \item An operation $\Delta\in \BV(1)$ of arity $1$ and cohomological degree $-1$, the BV operator.
  \item A symmetric degree 0 operation $-\wedge -\in \BV(2)$, the commutative product.
  \item A symmetric degree $-1$ operation $[-, -]\in \BV(2)$, the bracket.
\end{itemize}
\item These operations satisfy a list of relations.
First, $-\wedge -$ is a commutative and associative product and $[-, -]$ is a Lie bracket of degree $-1$.
Additionally, one has:
\begin{align*} \Delta^2 &= 0 \\
\Delta(x\wedge y ) - (\Delta x) \wedge y - x\wedge(\Delta y) - [x,y] &= 0 \\
[x,y\wedge z] - [x,y]\wedge z - [x,z] \wedge y &= 0. \end{align*}
\end{itemize}

The operad $\BV$ is cyclic. The natural $S_{r}$-action
on $\BV(\!(r)\!):=\BV(r-1)$ is determined by specifying the action of the cyclic permutation $\tau_r=(12\cdots r)$ on each of the generators.
\begin{align*} \tau_2 \Delta &= \Delta \\
\tau_3 (x\wedge y) &= x\wedge y \\
\tau_3 [x,y] &= -[x,y] -2x\wedge \Delta y. \end{align*}

We will mostly work with the (linear) dual cooperad $\BV^*=H^\bullet(\LD_2^{fr})$.
We then study the dg dual $D\BV^*$ of the cyclic cooperad $\BV^*$,
\[
  D\BV^*(\!(r)\!) := \Feyn(\BV^*)(\!(0,r)\!),
\]
and its part $D'\BV^*$ of arity $\geq 3$
\[
  D'\BV^*(\!(r)\!)
  =\begin{cases}
    D\BV^*(\!(r)\!) & \text{for $r\geq 3$} \\
    0 & \text{otherwise}
  \end{cases}.
\]
Recall that we defined the weight of a $\BV$ generator to be twice its homological degree. We therefore also obtain a weight grading on $D\BV$.

\begin{thm}[{Drummond-Cole--Vallette \cite[Theorem 2.21 and Proposition 3.9]{DCV}}]\label{thm:DCV}
There is a weight-preserving isomorphism of (non-cyclic) 1-shifted pseudo-operads 
\[
  H(D'\BV^*) \cong \tGrav
\]
\end{thm}

We can strengthen this as follows:
\begin{prop}\label{prop:DpDBV formal}
  The 1-shifted cyclic pseudo-operad $D'\BV^*$ is formal.
\end{prop}
The statement follows from the general principle that purity implies formality. 
\begin{proof}
We first argue that the isomorphism $H(D\BV^*) \cong \tGrav$ of Theorem \ref{thm:DCV} is actually an isomorphism of cyclic sequences.
Indeed, both sides are cyclic pseudo-operads and the isomorphism respects the non-cyclic pseudo-operad structure. Hence it is sufficient to check that the isomorphism
preserves the symmetric group action on the operadic generators. The gravity operad is generated by its top-weight part \cite{Getzler0}, and the top-weight part is one-dimensional in each arity. Now, consider the isomorphism between the top-weight parts 
\[
  \gr_{2r-6} \tGrav(\!(r)\!) \cong\gr_{2r-6}  H(D'\BV^*)(\!(r)\!).
\]
We know that the left-hand side is a one-dimensional trivial representations of $S_r$ concentrated in degree $2r-6$, and the isomorphism is $S_{r-1}$-equivariant.
Hence the right-hand side must be a one-dimensional representation of $S_r$ that restricts to the trivial representation of $S_{r-1}$. But the trivial representation is the only such, and hence the isomorphism of Theorem \ref{thm:DCV} is indeed an isomorphism of cyclic sequences.

Furthermore, the cohomology $H^k(D'\BV^*(\!(n)\!)) \cong \tGrav(\!(n)\!)^k \cong H^\bullet_c(\M_{0,n})$ is concentrated in weight $2k-2n+6$.
Or, equivalently, the weight $2W$-part of the cohomology is concentrated in degree $n-3+W$.
By homotopy transfer (see e.g., \cite{Ward}) we may endow $\tGrav$ with a homotopy (1-shifted) cyclic operad structure that respects the weight grading, such that $\tGrav$ is quasi-isomorphic to $D'\BV^*$.
We claim that the weight and degree constraints forbid all higher compositions.
Concretely, the $r$-ary $\infty$-operations are maps 
\[
  \tGrav(\!(n_1)\!)^{k_1}\otimes \cdots \otimes \tGrav(\!(n_r)\!)^{k_r}
  \to \tGrav(\!(n_1+\dots+n_r-2r+2)\!)^{k_1+\dots+k_r+1}
\]
that (additively) respects the weight gradings.
But the left-hand side is concentrated in weight 
\[
  \sum_i (2k_i - 2n_i+6)= 6r+ 2\sum_i (k_i - n_i)
\]
and the right-hand side is concentrated in weight 
\[
  2\sum_i k_i +2 - 2\sum_i n_i +4r-4  +6
  =
  4r+4+ 2\sum_i (k_i - n_i).
\]
Hence equality can only hold for $r=2$ and all higher $\infty$-operations must vanish, thus showing the formality claim.

\end{proof}

Note that $D\BV^*(\!(2)\!)$ carries the zero differential. Hence the cohomology of $D\BV^*$ satisfies:
\[
  H(D\BV^*(\!(r)\!))
  =
  \begin{cases}
    \bigoplus_k \Q \underbrace{\Dc \circ\cdots \circ \Dc}_{k\times} & \text{if $r=2$} \\
    \tGrav(\!(r)\!) & \text{if $r\geq 3$}\\
    0 & \text{otherwise}
  \end{cases}.
\]
Here $\Dc$ represents a binary operation of degree 1 and weight 2, dual to the BV operator $\Delta$. Note that the composition $\circ$ has degree $+1$ so that $\Dc \circ\cdots \circ \Dc$
has degree $2k-1$ and weight $2k$.

Via homotopy transfer one may also endow $H(D\BV^*)$ with a cyclic 1-shifted $\infty$-operad structure. This structure has been studied in the (non-cyclic) operadic setting in \cite{DCV}, but is unfortunately not very explicitly known.

\begin{lemma}\label{lem:feynp dbv}
  We have that 
  \[
H(\Feyn_{\kk}(D\BV^*)(\!(0,2)\!)) \cong \Q \Dc.
\]
In particular, the cyclic operad $D\BV^*$ with the above weight grading satisfies the conditions of Lemma \ref{lem:reduced Feyn}, and hence the inclusion 
\[
\Feyn_{\kk}'(D\BV^*)(\!(g,r)\!)  \subset \Feyn_{\kk}(D\BV^*)(\!(g,r)\!)
\]
is a quasi-isomorphism as long as $(g,r)\neq (1,0)$.
\end{lemma}
\begin{proof}
We have to check that the dg subspace $I_3\subset\Feyn_{\kk}(D\BV^*)(\!(0,2)\!)$ spanned by elements of weight $\geq 3$ is acyclic.
It suffices to check that the cohomology of
\[
  A := \Feyn_{\kk}(D\BV^*)(\!(0,2)\!)
\]
is concentrated in weight $\leq 2$.
But we have that 
\[
  \Feyn_{\kk}(D\BV^*)(\!(0,2)\!)
  =
  \bigoplus_p (D\BV^*(\!(2)\!) )^{\otimes p}
\]
is identified (up to degree shift) with the reduced bar construction of a free associative algebra in a single generator. But the cohomology of the reduced bar construction of a free associative algebra is the space of generators.
This implies that 
\[
H(\Feyn_{\kk}(D\BV^*)(\!(0,2)\!)) \cong \Q \Dc
\]
is concentrated in degree 1 and weight 2.
\end{proof}

\subsection{Koszul dual of \texorpdfstring{$\BV$}{BV}}
The (non-cyclic) Koszul dual operad $\BV^!$ of the operad $\BV$ has been computed in \cite[Section 1.3]{GCTV}.
The operad $\BV^!$ has the following presentation:
\begin{itemize}
\item Generators: $\BV^!$ is generated by a unary operation $u$ of degree 2, a symmetric binary operation (product) $\cdot$ of degree 2 and a symmetric binary operation (bracket) $[-,-]$ of degree 1.
\item Relations: $\cdot$ is a commutative associative product of degree 2. $[-,-]$ is a Lie bracket of degree 1. Additionally we have the Leibniz identity
\[
[x,y\cdot z] = [x,y]\cdot z + [x,z]\cdot y.  
\]
Furthermore, we have that 
\begin{align*} (ux)\cdot y& =u(x\cdot y) & [ux,y] = u[x,y]. \end{align*}
  This implies that 
\[
\BV^!(r) \cong \Ger(r)[-2r+2]\otimes \Q[u]
\]
as graded $S_r$-modules, with $\Q[u]$ the polynomials in a formal variable $u$ of degree 2.
\item Differential: $d=u\nabla$ with 
\begin{align*} \nabla (-\cdot-) &= [-,-] & \nabla [-,-] &= 0. \end{align*}
\end{itemize}
We may equip $\BV^!$ with an additional weight grading by declaring $\cdot$ and $u$ to be of weight $2$ and $[-,-]$ to be of weight 0.

We consider the 1-shifted version 
\[
  \tBV^! := \overline{\BV^!}[1],
\]
which is a 1-shifted non-cyclic pseudo-operad, that is, the composition in $\tBV^!$ has degree 1.
Notationally, it will be important to keep track of the degree shift. We denote the shift by using a symbol $\bs$ of degree $-1$ and weight 0. Hence if $X\in \BV^!(r)$ is an element then $\bs X$ denotes the corresponding element of one lower degree in $\tBV^!(r)$.
To fix the signs, we define the differential and composition on $\tBV^!$ by the formulas 
\begin{align*} d(\bs X) &= -\bs dX & \bs X \circ_j \bs Y = \circ_j(\bs X,\bs Y) = (-1)^{|X|+1} \bs (X\circ_j Y). \end{align*}

There is a weak equivalence of 1-shifted non-cyclic pseudo-operads
\[
D\BV^* \to \tBV^!.
\]
It sends all operadic generators of arity $\geq 3$ to zero, sends the product to $\bs [-,-]$ and the bracket to $\bs (-\cdot -)$, and the dual BV operator $\Dc$ to $\bs u$.

For future use we also introduce the truncated version $\tBV^!_{\geq 2}$ defined by removing the (non-cyclic) arity 1 operations,
\[
  \tBVt(r)
  :=
  \begin{cases}
    \tBV^!(r) & \text{for $r\geq 2$} \\
    0 & \text{otherwise}
  \end{cases}.
\]
By restriction we have a quasi-isomorphism of 1-shifted non-cyclic pseudo-operads
\[
  D'\BV^* \to \tBVt.
\]

\begin{example}\label{ex:top cohom tbv}
Explicitly, let us compute the top-weight and top$-2$-weight pieces of $H(\tBV^!(r)\!)\cong \tGrav(r)$ for $r\geq 2$.
The top-weight piece of the cohomology is a copy of $\Com$ concentrated in weight and degree $2r-4$.
The generator is the element 
\begin{equation}\label{equ:BVgen}
\bs \nabla (x_1\cdots x_r) = 
\sum_{i<j} [x_i,x_j]x_1\cdots\hat x_i\cdots \hat x_j \cdots x_r.  
\end{equation}
Here we omit the product $\cdot$ from the notation.

The top$-2$-weight piece of the cohomology $\gr_{2r-6}H(\tBV^!(r)\!)$ is as follows.
Generators are the elements $\bs \nabla \Omega_{ij}$, $1\leq i<j\leq r$, with 
\begin{equation}\label{equ:Omegaij}
\Omega_{ij} := [x_i,x_j] x_1\cdots\hat x_i\cdots \hat x_j \cdots x_r, 
\end{equation}
subject to the relation 
\[
\sum _{i<j}\bs\nabla \Omega_{ij} = 0.
\]
\end{example}

\section{Feynman transform of the gravity operad}
In order to understand the Feynman transform of $D\BV^*$ we first need some statements about the Feynman transform of its stable part $D'\BV^*$.
By Proposition \ref{prop:DpDBV formal} 
this is equivalent to studying the Feynman transform of the gravity operad $\tGrav$, i.e., we have a quasi-isomorphism 
\[
\Feyn_\kk(D'\BV^*) \xrightarrow{\simeq} \Feyn_\kk(\tGrav).  
\]

\begin{lemma}\label{lem:Feyngrav weights}
  The complex $\Feyn_{\kk}(\tGrav)(\!(g,n)\!)$ is concentrated in even weights $0,\dots, 4g-6+2n$. 
  \end{lemma}
  \begin{proof}
  Consider a decorated graph $\Gamma\in \Feyn_{\kk}(\tGrav)(\!(g,n)\!)$. Consider a vertex of $\Gamma$ with $h$ incident half-edges.
  The vertex is then decorated by an element of $H_c^\bullet(\M_{0,h})$, and the maximum weight occurring in the latter complex is the real dimension of $\M_{0,h}$, that is $2h-6$.
  Summing the weights of all vertices hence yields the total weight of $\Gamma$
  \[
    4 \#\text{edges} - 6 \#\text{vertices} +2n
    =
    6g -6 -2\#\text{edges} +2n.
  \]
  Since the number of edges of a genus $g$ graph must be at least $g$ the result follows.
\end{proof}

The cohomology in the top weight and top$-2$ weight can be computed explicitly.
\begin{thm}\label{thm:Feyn grav top}
We have that 
\begin{equation}\label{equ:Feyn grap top}
  \gr_{4g-6+2n} H^k(\Feyn_{\kk}(\tGrav)(\!(g,n)\!))
\cong 
\begin{cases}
  \Q &\text{for $k=4g-6+2n$} \\
0 &\text{otherwise}
\end{cases}
\end{equation}
and
\begin{equation}\label{equ:Feyn grap top 2}
  \gr_{4g-8+2n} H^k(\Feyn_{\kk}(\tGrav)(\!(g,n)\!))
\cong 
\begin{cases}
  \HyCom(\!(2g+n)\!)^{2}_{S_2\wr S_g} &\text{for $k=4g-8+2n$} \\
0 &\text{otherwise}
\end{cases}.
\end{equation}
\end{thm}
Note that in particular, in each case the weight $W$ cohomology is  concentrated in degree $W$.
Furthermore, we may understand the 1-dimensional vector space \eqref{equ:Feyn grap top} as 
\[
\Q \cong \HyCom(\!(2g+n)\!)^{0}_{S_2\wr S_g},
\]
so as to display the similar form of \eqref{equ:Feyn grap top} and \eqref{equ:Feyn grap top 2}. 
\begin{proof}
Consider first the top weight $4g-6+2n$.
By the proof of Lemma \ref{lem:Feyngrav weights}
the only graphs that can contribute to the complex $\gr_{4g-6+2n} \Feyn_{\kk}(\tGrav)(\!(g,n)\!)$ have exactly $g$ edges. But a genus $g$ graph with $g$ edges necessarily has the form 
\begin{equation}\label{equ:multi tad}
T_{g,n}:=
\begin{tikzpicture}[yshift=-.5cm] \node[ext] (v) at (0,1) {}; \node at (-.7,0) {$1$}; \node at (.7,0) {$n$}; \node at (0,0) {$\cdots$}; \node at (0,1.3) {$\scriptscriptstyle \cdots$}; \draw (v) edge +(-.7,-.7) edge +(-.35,-.7)edge +(+.35,-.7)edge +(.7,-.7) edge +(0,-.7) edge[in=0,out=30,loop] (v) edge[in=35,out=65,loop] (v) edge[in=180,out=150,loop] (v) edge[in=145,out=115,loop] (v) ; \end{tikzpicture},
\end{equation}
that is, it consists of a single vertex with $g$ tadpoles.
Furthermore, this vertex needs to be decorated by the top-weight part $\gr_{4g+2n-6}\tGrav(\!(2g+n)\!)\cong \Q[-4g-2n+6]$.
Hence the complex $\gr_{4g-6+2n} \Feyn_{\kk}(\tGrav)(\!(g,n)\!)$ is one-dimensional and \eqref{equ:Feyn grap top} follows.

Next consider weight $4g-8+2n$. 
Again, by the proof of Lemma \ref{lem:Feyngrav weights}
the graphs that contribute to $\gr_{4g-8+2n} \Feyn_{\kk}(\tGrav)(\!(g,n)\!)$
have either $g$ or $g+1$ edges.
The graphs with $g$ edges are of the form 
\eqref{equ:multi tad}, decorated by an element 
$x\in \gr_{4g+2n-8}\tGrav(\!(2g+n)\!)$.
Let us denote such graph by $T_{g,n}(x)$.

The graphs with $g+1$ edges are of the form 
\[
\delta^*_{a,b,A}
=
\begin{tikzpicture}[yshift=-.5cm] \node[ext] (v) at (0,1) {}; \node[ext] (w) at (1,1) {}; \node at (0,0) {$\scriptstyle A$}; \node at (1,0) {$\scriptstyle A^c$}; \node at (0,1.3) {$\scriptscriptstyle \cdots$}; \node at (0,1.6) {$\scriptstyle a\times$}; \node at (1,1.6) {$\scriptstyle b\times$}; \node at (1,1.3) {$\scriptscriptstyle \cdots$}; \node at (.5,1) {$\scriptscriptstyle \cdots$}; \draw (v) edge[bend left] (w) edge[bend right] (w) edge +(-.35,-.7)edge +(+.35,-.7) edge +(0,-.7) edge[in=35,out=65,loop] (v) edge[in=180,out=150,loop] (v) edge[in=145,out=115,loop] (v) (w) edge +(-.35,-.7)edge +(+.35,-.7) edge +(0,-.7) edge[in=0,out=30,loop] (v) edge[in=35,out=65,loop] (v) edge[in=145,out=115,loop] (v) ; \end{tikzpicture}.
\]
Here $A\subset \{1,\dots,n\}$ and there are $a$ tadpoles on the left-hand vertex, $b$ tadpoles on the right-hand vertex and $g-a-b+1$ edges connecting both vertices. We require that $a+b\leq g$ and that $|A|+g+a-b\geq 2$ and $|A^c|+g-a+b\geq 2$ to ensure that the vertices are at least trivalent.
By symmetry we furthermore identify 
\[
  \delta^*_{a,b,A}
  =
  \delta^*_{b,a,A^c}.
\] 
All vertices are implicitly decorated by the generator of the (one-dimensional) top-weight part of $\tGrav$.

Hence $\gr_{4g-8+2n} \Feyn_{\kk}(\tGrav)(\!(g,n)\!)$ is a simple two-term complex of the form 
\begin{equation}\label{equ:fg complex}
  \underbrace{\vspan\left(\delta^*_{a,b,A}\right)}
  _{\text{in degree }4g-8+2n}
\xrightarrow{d}
\underbrace{\vspan\left(T_{g,n}(x)\right)}
_{\text{in degree }4g-9+2n}
\cong \gr_{4g-8+2n} \tGrav(\!(2g+n)\!)_{S_2\wr S_g}
.
\end{equation}
On the other hand, we know from Getzler's proof \cite{Getzler0} of the Koszulness of the hypercommutative operad that 
$H(\Feyn_\kk(\tGrav)(\!(0,N)\!))=\HyCom^*(\!(N)\!)$.
In particular this implies that 
$$H(\gr_{4g-8+2n}\Feyn_\kk(\tGrav)(\!(0,2g+n)\!)_{S_2\wr S_g})=\gr_{4g-8+2n}\HyCom^*(\!(2g+n)\!)_{S_2\wr S_g}.$$

{\bf Claim A:} The complex $\gr_{4g-8+2n} \Feyn_{\kk}(\tGrav)(\!(g,n)\!)$ given by \eqref{equ:fg complex} is isomorphic to the complex 
\begin{equation}\label{equ:fg complex 2}
  \gr_{4g-8+2n}\Feyn_\kk(\tGrav)(\!(0,2g+n)\!)_{S_2\wr S_g}.
\end{equation}

From Claim A, Theorem \ref{thm:Feyn grav top} follows using that $\gr_{4g-8+2n}\HyCom^*(\!(2g+n)\!)$ is concentrated in degree $4g-8+2n$ and that $\gr_{4g-8+2n}\HyCom^*(\!(2g+n)\!)\cong \gr_{2}\HyCom(\!(2g+n)\!)$ by Poincar\'e duality.

Let us prove Claim A. 
Setting the genus in $\gr_{4g-8+2n} \Feyn_{\kk}(\tGrav)(\!(g,n)\!)$ to be zero and replacing $n$ by $2g+n$, \eqref{equ:fg complex} shows that \eqref{equ:fg complex 2} is of the form 
\begin{equation}\label{equ:fg complex 3}
\vspan\left(\delta^*_{0,0,A}\right)
_{S_2\wr S_g}
\to
\gr_{4g-8+2n} \tGrav(\!(2g+n)\!)_{S_2\wr S_g}.
\end{equation}
Now the isomorphism from \eqref{equ:fg complex} to \eqref{equ:fg complex 3} is defined on the right-hand term $\gr_{4g-8+2n} \tGrav(\!(2g+n)\!)_{S_2\wr S_g}$ as the identity and on the left-hand term by sending 
\[
  \delta^*_{a,b,A}
  \mapsto
  (g-a-b)
  \delta^*_{0,0,A\cup \{n+1,n+2,\dots,n+2a, n+2a+1, n+2a+3,\dots,n+2g-2b-3\} }. \qedhere
\]
\end{proof}

For later use we extract the following special case of the results of Theorem \ref{thm:Feyn grav top} and Lemma \ref{lem:Feyngrav weights}.
\begin{cor}
  \label{cor:Feyngrav}
Let $g,n$ be such that $2g+n\geq 3$. Then we have that 
\begin{multline}\label{equ:Fgrav hw}
\gr_{6g-6+2n} H^k(\Feyn_{\kk}(\tGrav)(\!(g,n)\!))\cong \gr_{6g-6+2n} H^k(\Feyn_{\kk}(D'\BV^*)(\!(g,n)\!))
\\\cong 
\begin{cases}
  \HyCom(\!(n)\!)^0 =
  \Com(\!(n)\!)\cong  
  \Q &\text{for $g=0$, $k=2n-6$} \\
0 &\text{otherwise}
\end{cases}
\end{multline}
and 
\begin{multline}\label{equ:Fgrav hw2}
\gr_{6g-8+2n} H^k(\Feyn_{\kk}(\tGrav)(\!(g,n)\!))\cong \gr_{6g-8+2n} H^k(\Feyn_{\kk}(D'\BV^*)(\!(g,n)\!))
\\\cong 
\begin{cases}
  \HyCom^*(\!(n)\!)^{2n-8}
  \cong 
  \HyCom(\!(n)\!)^2 &\text{for $g=0$, $k=2n-8$} \\
  \Q\cong \Com(\!(n+2)\!)_{S_2} 
  &\text{for $g=1$, $k=2n-2$} \\
0 &\text{otherwise}
\end{cases}
\end{multline}
\end{cor}
\begin{proof}
\eqref{equ:Fgrav hw} follows by setting $g=0$ in \eqref{equ:Feyn grap top}. In \eqref{equ:Fgrav hw2}, the case $g=0$, $k=2n-8$ follows by setting $g=0$ in \eqref{equ:Feyn grap top 2} and the case $g=1$, $k=2n-2$ follows by setting $g=1$ in \eqref{equ:Feyn grap top}.
\end{proof}

\begin{remark}\label{rem:feyngrav top 2}
In the special case $n=0$ one can deduce from the proof of Theorem \ref{thm:Feyn grav top} by counting the numbers of graphs involved that 
\[
   \mathrm{dim} \, \gr_{4g-8+2n} H^{4g-8+2n}(\Feyn_{\kk}(\tGrav)(\!(g,0)\!))
  =
  \begin{cases}
   \frac 14 (g+2)^2-3& \text{for $g$ even}
   \\
   \frac 14 (g+1)(g+3) -3& \text{for $g$ odd}
  \end{cases}.
\]
\end{remark}

\subsection{Representatives}\label{sec:top cohomology feyn grav}
We have a quasi-isomorphism of 1-shifted non-cyclic pseudo-operads
\[
\tGrav \to \tBVt.
\]
There is hence a quasi-isomorphism of non-cyclic pseudo-cooperads
\[
  \Feyn_\kk(\tGrav)(\!(0,-)\!) \to \Feyn_\kk(\tBVt)(\!(0,-)\!).  
\]
Here we abuse the notation a bit and apply the Feynman transform to $\tBVt$, even though this is not a cyclic operad. However, in genus zero, and if we consider only the non-cyclic cooperad structure on the result this is still well-defined, as it is equivalent to applying the operadic bar construction.

For later use, we shall need to write down explicit representatives of the top- and top$-2$-weight cohomology classes identified in Corollary \ref{cor:Feyngrav}, in genus zero.

First consider the top-weight $2n-6$ part in genus $g=0$.
The representative of a generator of $$\gr_{2n-6}H^{2n-6}(\Feyn_\kk(\tGrav)(\!(0,n)\!))\cong \Q$$ is a graph with a single vertex, decorated by an element of the one-dimensional space $$\gr_{2n-6}\tGrav(\!(0,n)\!)^{2n-6}\cong \Q.$$

In $\tBVt\simeq \tGrav$ this top-weight generator is represented by \eqref{equ:BVgen}.
Hence the cohomology $$\gr_{2n-4}H^{2n-4}(\Feyn_\kk(\tBVt)(\!(0,n+1)\!))\cong \Q$$ is generated by a graph with a single vertex, decorated by the element $\bs\nabla(x_1\cdots x_{n})$.

Next consider weight $2n-8$ in genus $g=0$.
Here the representatives are certain linear combinations of graphs with two vertices and one edge.
Each vertex is decorated by a top-weight element of $\tGrav$.
To be more explicit, label the inputs of some element of $\gr_{2n-6}\Feyn_\kk(\tBVt)(\!(0,n+1)\!)$ by $0,\dots,n$.
We break cyclic symmetry by considering $0$ the root of our trees.
Then for $A\subset \{1,\dots,n\}$ such that $n-1\geq |A|\geq 2$ we build an element 
\[
  T_A \in \gr_{2n-6}\Feyn_\kk(\tBVt)(\!(0,n+1)\!),
\]
that is a graph with two vertices.
The lower vertex is decorated by $\bs\nabla (\prod_{i\in A}x_i)$, and the upper by $\bs\nabla(x_t\prod_{i\in A^c}x_i)$.
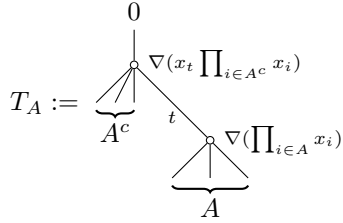
\[
T_A :=\, 
\begin{tikzpicture} \node[ext,label=0:{$\scriptstyle \nabla(x_t\prod_{i\in A^c}x_i)$}] (v1) at (0,.5) {}; \node[ext,label=0:{$\scriptstyle \nabla(\prod_{i\in A}x_i)$}] (v2) at (1,-.5) {}; \node (v0) at (0,1.2) {0}; \draw (v1) edge node[below] {$\scriptstyle t$} (v2) edge (v0) edge +(-.5,-.5) edge +(-.25,-.5) edge +(0,-.5) (v2) edge +(-.5,-.5) edge +(0,-.5) edge +(.5,-.5); \draw[pbrace, thick] (0,-.1) -- (-.5,-.1); \node at (-.25, -.4) {$A^c$}; \draw[pbrace, thick] (1.5,-1.1) -- (.5,-1.1); \node at (1, -1.4) {$A$}; \end{tikzpicture}
\]

We want to consider only those linear combinations $\sum_A c_A T_A$ of such graphs that are closed under the edge contraction differential.
Concretely, $d_c T_A$ is a graph with one vertex, decorated by the element 
\[
\bs\nabla\left(\!(\prod_{i\in A^c}x_i) \nabla(\prod_{i\in A}x_i)  \right)  
=
\bs\nabla\left(\sum_{i,j\in A\atop i<j} \Omega_{ij} \right),  
\]
using the notation \eqref{equ:Omegaij}.
But the kernel of $\nabla$ is one-dimensional, spanned by $\sum_{i<j} \Omega_{ij}$.
Hence we find that 
\[
  d_c \left(\sum_A c_A T_A\right) =0
\]
if and only if for each choice of indices $i\neq j$, $k\neq l$ we have 
\begin{equation}\label{equ:ca eqn}
\sum_{A\atop i,j\in A} c_A =  \sum_{A\atop k,l\in A} c_A.
\end{equation}
Hence our explicit representatives for $\gr_{2n-8}H^{2n-8}(\Feyn_\kk(\tBVt))(\!(0,n+1)\!)$ are those linear combinations $\sum_A c_A T_A$ whose coefficients satisfy the linear equations \eqref{equ:ca eqn}.
The identification with $\HyCom$ is such that the element $T_A$ here corresponds to the element $\delta_A^*$ in Section \ref{sec:hycom1}.

Finally consider the genus 1 part.
Here the representative of weight $2n-2$ and degree $2n-2$ of $\Feyn_\kk(\tGrav)(\!(0,n)\!)$ is a genus 1 graph with one vertex 
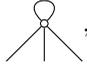
\[
\begin{tikzpicture} \node[ext] (v) at (0,0) {}; \draw (v) edge[loop] (v) edge +(-.5,-.5) edge +(0,-.5) edge +(.5,-.5); \end{tikzpicture},
\]
with the vertex being decorated by a (top-weight) element of $\gr_{2n-2}\tGrav(\!(n+2)\!)^{2n-2}\cong \Q$.

\subsection{Operations on \texorpdfstring{$H(\Feyn_{\kk}(\tGrav))$}{H(Feyn(Grav))}}
\label{sec:operations ftgrav}
We define three operations on the cohomology of the Feynman transform $H(\Feyn(\tGrav))$ that are derived using the interpretation of $\tGrav$ as the cohomology of the part $D'\BV^*$ of arity $\geq 3$ of $D\BV^*$, cf. Theorem \ref{thm:DCV}.

More precisely, by formality of $D'\BV^*$ (see Proposition \ref{prop:DpDBV formal}) we have an isomorphism
\[
  H(\Feyn_{\kk}(\tGrav)) \cong H(\Feyn_{\kk}(D'\BV^*)).
\]
We shall define operations on $H(\Feyn_{\kk}(\tGrav))$ by specifying cochain level operations on $\Feyn_{\kk}(D'\BV^*)$.

The first set of operations are unary operations of degree $+2$ and weight $+2$
\[
\Delta_j : \Feyn_{\kk}(D'\BV^*)(\!(g,n)\!) \to \Feyn_{\kk}(D'\BV^*)(\!(g,n)\!)
\]
for $j=1,\dots, n$ and $2g+n\geq 3$.
To define $\Delta_j$ consider a graph $\Gamma\in \Feyn_{\kk}(D'\BV^*)(\!(g,n)\!)$ and the vertex $v$ incident to the leg $j$.
The vertex $v$ carries a decoration $A_v\in D'\BV^*(\!(g_v, \gst(v))\!)$, and we assume that the half-edge at the vertex that connects to leg $j$ is $h$.
Then we define $\Delta_j\Gamma$ to be the same graph $\Gamma$, but with decoration $A_v$ changed to the cyclic 1-shifted-operadic composition 
\[
A_v' = A_v \circ_h \Dc 
\]
with the dual BV operator $\Dc\in D\BV^*(\!(0,2)\!)$.

The second operation we shall need is a binary ``operadic composition'' of weight $+2$ and degree $+2$.
\begin{gather*} \circ_{i,j} : \Feyn_{\kk}(D'\BV^*)(\!(g,n)\!) \otimes \Feyn_{\kk}(D'\BV^*)(\!(g',n')\!) \to \Feyn_{\kk}(D'\BV^*)(\!(g+g',n+n'-2)\!)\\
\Gamma\otimes \Gamma' \mapsto \Gamma \circ_{i,j} \Gamma', \end{gather*}
with $1\leq i\leq n$, $1\leq j\leq n'$.
Concretely, let $\Gamma\in \Feyn(D'\BV^*)(\!(g,n)\!)$, $\Gamma'\in \Feyn_{\kk}(D'\BV^*)(\!(g',n')\!)$ be two graphs. Then we set 
\[
  \Gamma \circ_{i,j} \Gamma' := \mu_{i,j}( \Delta_i \Gamma, \Gamma' ) + \mu_{i,j}(  \Gamma, \Delta_j\Gamma' ),
\]
where $\mu_{i,j}(A,B)$ takes graphs $A$ and $B$ and glues them together with a new edge 
\[
  \mu_{i,j}\left(
\begin{tikzpicture} \node[draw, circle] (v) at (0,0) {$\Gamma$}; \node (i) at (.7,0) {$i$}; \draw (v) edge (i) edge +(-.5,.5) edge +(-.5,-.5); \end{tikzpicture},
\begin{tikzpicture} \node[draw, circle] (v) at (0,0) {$\Gamma'$}; \node (i) at (-.7,0) {$j$}; \draw (v) edge (i) edge +(.5,.5) edge +(.5,-.5); \end{tikzpicture}
  \right)
  \quad = \quad 
  \begin{tikzpicture} \node[draw, circle] (v) at (0,0) {$\Gamma$}; \node[draw, circle] (vv) at (1,0) {$\Gamma'$}; \draw (v) edge (vv) edge +(-.5,.5) edge +(-.5,-.5) (vv) edge +(.5,.5) edge +(.5,-.5); \end{tikzpicture}
\]

We also have a similar operation of weight and degree $+2$
\begin{gather*} \eta_{i,j} : \Feyn_{\kk}(D'\BV^*)(\!(g,n)\!) \to \Feyn_{\kk}(D'\BV^*)(\!(g+1,n-2)\!)\\
 \Gamma \mapsto \eta_{i,j}(\Gamma), \end{gather*}
defined such that 
\[
\eta_{i,j} (\Gamma) =\nu_{i,j}(\Delta_i\Gamma) + \nu_{i,j}(\Delta_j\Gamma) 
\]
with 
\[
  \nu_{i,j}\left(
\begin{tikzpicture} \node[draw, circle] (v) at (0,0) {$\Gamma$}; \node (i) at (.7,.7) {$i$}; \node (j) at (.7,-.7) {$j$}; \draw (v) edge (i) edge (j) edge +(-.5,.5) edge +(-.5,-.5); \end{tikzpicture}
  \right)
  \quad = \quad 
  \begin{tikzpicture} \node[draw, circle] (v) at (0,0) {$\Gamma$}; \draw (v) edge +(-.5,.5) edge +(-.5,-.5) edge[in=45,out=-45, loop ](v) ; \end{tikzpicture}
\]

\begin{lemma}
The operations $\Delta_j$, $\circ_{i,j}$ and $\eta_{i,j}$ are well defined cochain maps and hence induce morphisms on cohomology that we denote by the same letters.
\end{lemma}
\begin{proof}
This is clear for $\Delta_j$, since the operadic composition is compatible with the differentials.

For $\mu_{i,j}$ one has to check that 
\[
A:= d\mu_{i,j}(\Gamma,\Gamma') - 
\mu_{i,j}(d\Gamma,\Gamma')-(-1)^{|\Gamma|}\mu_{i,j}(\Gamma,d\Gamma')=0.
\] 
The differential can be decomposed as $d=d_c+d_{D\BV^*}$, with $d_{D\BV^*}$ the part induced from the differential on $D\BV^*$, and $d_c$ the edge contraction.
The operation $\mu_{i,j}$ is clearly compatible with the internal part of the differential $d_{D\BV^*}$:
\begin{align*} &d_{D\BV^*}\mu_{i,j}(\Gamma,\Gamma') \\
 &= d_{D\BV^*} (\mu_{i,j}(\Delta_i\Gamma,\Gamma' )+\mu_{i,j}(\Gamma,\Delta_j\Gamma' )) \\&= \mu_{i,j}(d_{D\BV^*}\Delta_i\Gamma,\Gamma' ) +(-1)^{|\Gamma|} \mu_{i,j}(\Delta_i\Gamma,d_{D\BV^*}\Gamma' ) \\& +\mu_{i,j}(d_{D\BV^*}\Gamma,\Delta_j\Gamma' ) +(-1)^{|\Gamma|} \mu_{i,j}(\Gamma,d_{D\BV^*}\Delta_j\Gamma' ) \\&= \mu_{i,j}(\Delta_i d_{D\BV^*}\Gamma,\Gamma' ) +(-1)^{|\Gamma|} \mu_{i,j}(\Delta_i\Gamma,d_{D\BV^*}\Gamma' ) \\& +\mu_{i,j}(d_{D\BV^*}\Gamma,\Delta_j\Gamma' ) +(-1)^{|\Gamma|} \mu_{i,j}(\Gamma,\Delta_jd_{D\BV^*}\Gamma' ) \\&= \mu_{i,j}(d_{D\BV^*}\Gamma.\Gamma')-(-1)^{|\Gamma|}\mu_{i,j}(\Gamma.d_{D\BV^*}\Gamma'). \end{align*}
In the analogous computation with $d_c$ there appears a new term, from contracting the new edge introduced by $\mu_{i,j}$.
Denote by $\gamma\in D\BV^*$ the decoration at the vertex connected to leg $i$ in $\Gamma$ and denote by $\gamma'\in D\BV^*$ the decoration at the vertex connected to leg $j$ in $\Gamma'$.
Then the obstruction $A$ to compatibility with the differential is a graph with one vertex decorated by 
\[
  (\gamma \circ_h \Dc) \circ_{h'} \gamma'
  +
  \gamma \circ_{h} (\Dc \circ_{h'} \gamma')
  =0,
\]
with the terms cancelling since the composition ``$\circ$'' is an operation of odd degree.
This shows that $\circ_{i,j}$ is a well-defined cochain map.

The proof of well-definedness of $\eta_{i,j}$ is analogous to that of $\circ_{i,j}$.
\end{proof}

\subsection{The three operations in high weight}
In this section we shall understand the three operations $\Delta_j$, $\circ_{i,j}$ and $\eta_{i,j}$ of the previous subsection in the case of high weights, that is, weights $6g-6+2n$ and $6g-8+2n$.
Taking into account the computation of those high weight parts in Corollary \ref{cor:Feyngrav}, there is the following list of nontrivial cases to consider:
\begin{align*} &\Delta_j : \gr_{2n-8} H^{2n-8}( \Feyn_{\kk}(\tGrav)(\!(0,n)\!) ) \\&\quad\quad\quad\quad \to \gr_{2n-6} H^{2n-6}( \Feyn_{\kk}(\tGrav)(\!(0,n)\!) ) \\
&\circ_{i,j}: \gr_{2n-6} H^{2n-6}( \Feyn_{\kk}(\tGrav)(\!(0,n)\!) )\otimes \gr_{2n'-6} H^{2n'-6}( \Feyn_{\kk}(\tGrav)(\!(0,n')\!) ) \\&\quad\quad\quad\quad \to \gr_{2n+2n'-10} H^{2n+2n'-10}( \Feyn_{\kk}(\tGrav)(\!(0,n+n'-2)\!) ) \\
&\circ_{i,j}: \gr_{2n-8} H^{2n-8}( \Feyn_{\kk}(\tGrav)(\!(0,n)\!) ) \otimes \gr_{2n'-6} H^{2n'-6}( \Feyn_{\kk}(\tGrav)(\!(0,n')\!) ) \\&\quad\quad\quad\quad \to \gr_{2n+2n'-12} H^{2n+2n'-12}( \Feyn_{\kk}(\tGrav)(\!(0,n+n'-2)\!) ) \\
&\circ_{i,j}: \gr_{2n-2} H^{2n-2}( \Feyn_{\kk}(\tGrav)(\!(1,n)\!) ) \otimes \gr_{2n'-6} H^{2n'-6}( \Feyn_{\kk}(\tGrav)(\!(0,n')\!) ) \\&\quad\quad\quad\quad \to \gr_{2n+2n'-6} H^{2n+2n'-6}( \Feyn_{\kk}(\tGrav)(\!(1,n+n'-2)\!) ) \\
&\eta_{i,j}: \gr_{2n-8} H^{2n-8}( \Feyn_{\kk}(\tGrav)(\!(0,n)\!) ) \\&\quad\quad\quad\quad \to \gr_{2n-6} H^{2n-6}( \Feyn_{\kk}(\tGrav)(\!(1,n-2)\!) ) . \end{align*}
Here we also used the symmetry of $\circ_{i,j}$ to remove one case.

\begin{prop}\label{prop:circ is comp Com}
The operation 
\begin{multline*}
  \circ_{i,j}:
\gr_{2n-6} H^{2n-6}( \Feyn_{\kk}(\tGrav)(\!(0,n)\!) ) \otimes 
\gr_{2n'-6} H^{2n'-6}( \Feyn_{\kk}(\tGrav)(\!(0,n')\!) )
\cong \Com(\!(n)\!) \otimes \Com(\!(n')\!)
\\
\to \gr_{2n+2n'-10} H^{2n+2n'-10}( \Feyn_{\kk}(\tGrav)(\!(0,n+n'-2)\!) ) \cong \Com(\!(n+n'-2)\!)
\end{multline*}
agrees with minus the cyclic operadic composition of the cyclic operad $\Com$, using the identification \eqref{equ:Fgrav hw}.
\end{prop}
\begin{proof}
We consider an arbitrary leg as the root, and define the corresponding non-cyclic sequences
\begin{align}\label{equ:Ydef} Y(S) := \gr_{2r-4} H^{2r-4}( \Feyn_{\kk}(\tBVt)(\!(0,\{0\}\sqcup S)\!) ) \cong \gr_{2r-4} H^{2r-4}( \Feyn_{\kk}(D'\BV^*)(\!(0,\{0\}\sqcup S)\!) ), \end{align}
with $S$ running over finite sets and $r:=|S|$.
Then the above morphisms $\circ_{i,j}$ are equivalent to morphisms 
\[
\hat\circ_*: Y(S\sqcup \{*\}) \otimes Y(S') \to Y(S\sqcup S') .
\]
We know that $Y(S)\cong \Com(S)\cong \Q$, and we want to show that $\circ_j$ corresponds to the composition in the operad $\Com$, or in other words to the usual isomorphism $\Q\otimes\Q\to\Q$.

From Section \ref{sec:top cohomology feyn grav} and Example \ref{ex:top cohom tbv} we see that $Y(S)$ is generated by the 1-vertex graph decorated with the element
\[
 \bs \nabla( \prod_{j\in S} x_j) \in \tBVt(S).
\]
Let us denote the graph in $Y(S)$ with one vertex decorated by $x\in \tBVt(S)$ by 
\[
T(x) =
\begin{tikzpicture} \node[ext,label=0:{$\scriptstyle x$}] (v1) at (0,0) {}; \node (v0) at (0,.7) {0}; \draw (v1) edge (v0) edge +(-.5,-.5) edge +(0,-.5) edge +(.5,-.5); \end{tikzpicture}
\]
and the two-vertex graph with vertices decorated by $x$ and $y$ by 
\begin{equation}\label{equ:TAxydef}
T_A(x,y) = 
\begin{tikzpicture} \node[ext,label=0:{$\scriptstyle x$}] (v1) at (0,.5) {}; \node[ext,label=0:{$\scriptstyle y$}] (v2) at (1,-.5) {}; \node (v0) at (0,1.2) {0}; \draw (v1) edge node[below] {$\scriptstyle *$} (v2) edge (v0) edge +(-.5,-.5) edge +(-.25,-.5) edge +(0,-.5) (v2) edge +(-.5,-.5) edge +(0,-.5) edge +(.5,-.5); \draw[pbrace, thick] (0,-.1) -- (-.5,-.1); \node at (-.25, -.4) {$A^c$}; \draw[pbrace, thick] (1.5,-1.1) -- (.5,-1.1); \node at (1, -1.4) {$A$}; \end{tikzpicture},
\end{equation}
with $A\subset S$.
We then compute 
\begin{align*} &T\left(\bs\nabla( x_* \prod_{i\in S}x_i)\right) \hat\circ_j T\left(\bs\nabla( \prod_{j\in S'} x_j)\right) \\
 =~& T_{S'}\left(\bs u\nabla( x_* \prod_{i\in S}x_i), \bs\nabla( \prod_{j\in S'} x_j)\right) - T_{S'}\left(\bs \nabla( x_* \prod_{i\in S}x_i), \bs u\nabla( \prod_{j\in S'} x_j)\right) \\
 =~& -d_{\tBVt}\underbrace{\left(T_{S'}\left(\bs x_* \prod_{i\in S}x_i, \bs \nabla( \prod_{j\in S'} x_j)\right) - T_{S'}\left(\bs \nabla( x_* \prod_{i\in S}x_i), \bs \prod_{j\in S'} x_j\right)  \right)}_{=:X}. \end{align*} 
Using the notation \eqref{equ:Omegaij} the result is hence cohomologous to 
\begin{align*} d_c X &= -T\left( \sum_{i<j\in S'}\bs\Omega_{ij} + \sum_{i<j \in S\sqcup S' \atop i\text{ or }j\in S} \bs\Omega_{ij} \right) \\&= -T\left( \sum_{i<j \in S\sqcup S'}\bs\Omega_{ij} \right) = -T\left(\bs \nabla \prod_{i\in S\sqcup S'} x_i\right).                  \end{align*}

Hence we have shown that the result is minus the standard generator of $Y(S\sqcup S')\cong \Q$, and we have thus shown the Proposition.

Also note that the notation ``$i<j$'' in the above computation is slightly abusive and shall indicate that we sum over each pair only once.
\end{proof}

\begin{prop}\label{prop:circ is comp HyCom}
  The operation 
  \begin{multline*}
    \circ_{i,j}:
    \gr_{2n-8} H^{2n-8}( \Feyn_{\kk}(\tGrav)(\!(0,n)\!) )
    \otimes \gr_{2n'-6} H^{2n'-6}( \Feyn_{\kk}(\tGrav)(\!(0,n')\!) )
\\
\cong \HyCom(\!(n)\!)^2 \otimes \Com(\!(n')\!)
\cong \HyCom(\!(n)\!)^2 \otimes \HyCom(\!(n')\!)^0
    \\ \to 
    \gr_{2n+2n'-8} H^{2n+2n'-8}( \Feyn_{\kk}(\tGrav)(\!(0,n+n'-2)\!) )
    \cong \HyCom(\!(n)\!)^2
  \end{multline*}
  agrees with minus the cyclic operadic composition in $\HyCom$, using the identifications \eqref{equ:Fgrav hw}, \eqref{equ:Fgrav hw2}.
\end{prop}
\begin{proof}
Again, we can show that statement in the non-cyclic setting, using the model $\tBVt$ instead of $D'\BV^*$.
Representatives of 
\[
Z(S) := \gr_{2r-8} H^{2r-8}( \Feyn_{\kk}(\BV^{!}_{\geq 3})(\!(0,\{0\}\sqcup S)\!) )
\cong \gr_{2r-8} H^{2r-8}( \Feyn_{\kk}(\tGrav)(\!(0,\{0\}\sqcup S)\!)) ,
\]
with $S$ a finite set of cardinality $r$,
have been computed in Section \ref{sec:top cohomology feyn grav}. Using the notation therein, they are given by formal linear combinations of trees 
\[
  \sum_{A\subset S\atop 2\leq |A|\leq n-1} c_A T_A
\]
satisfying the condition \ref{equ:Hycom condition}.
Comparing to Section \ref{sec:hycom1} the symbol $T_A$ corresponds to $\delta_A^*$ there.

We need to consider the composition morphisms $\circ_\rt:=\circ_{\rt,0}$
\begin{align}\label{equ:circ ch 1} \circ_\rt \colon Z(S\sqcup \{\rt\})\otimes Y(S') \to Z(S\sqcup S') \\
\label{equ:circ ch 2} \circ_\rt \colon Y(S'\sqcup \{\rt\})\otimes Z(S) \to Z(S\sqcup S'), \end{align}
using the notation \eqref{equ:Ydef}.
We have to check that these compositions adhere to the description \ref{equ:Hycom com action} of the $\Com$-action on $\HyCom$.

We continue to use the notation \eqref{equ:TAxydef} for specifying trees with vertices decorated by $\tBVt$-elements and extend it as follows:
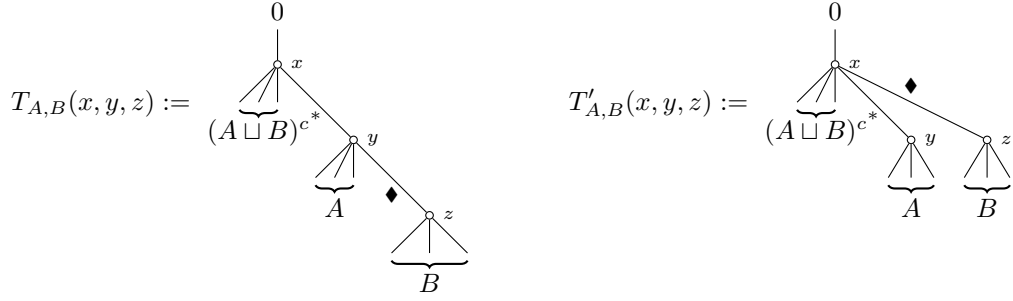
\begin{align*}              T_{A,B}(x,y,z) &:= \begin{tikzpicture} \node[ext,label=0:{$\scriptstyle x$}] (v1) at (0,.5) {}; \node[ext,label=0:{$\scriptstyle y$}] (v2) at (1,-.5) {}; \node[ext,label=0:{$\scriptstyle z$}] (v3) at (2,-1.5) {}; \node (v0) at (0,1.2) {0}; \draw (v1) edge node[below] {$\scriptstyle *$} (v2) edge (v0) edge +(-.5,-.5) edge +(-.25,-.5) edge +(0,-.5) (v2) edge +(-.5,-.5) edge +(0,-.5) edge +(-.25,-.5) edge node[below] {$\scriptstyle \rt$} (v3) (v3) edge +(-.5,-.5) edge +(0,-.5) edge +(.5,-.5); \draw[pbrace, thick] (0,-.1) -- (-.5,-.1); \node at (-.25, -.4) {$(A\sqcup B)^c$}; \draw[pbrace, thick] (1,-1.1) -- (.5,-1.1); \node at (.75, -1.4) {$A$}; \draw[pbrace, thick] (2.5,-2.1) -- (1.5,-2.1); \node at (2, -2.4) {$B$}; \end{tikzpicture} & T_{A,B}'(x,y,z) &:= \begin{tikzpicture} \node[ext,label=0:{$\scriptstyle x$}] (v1) at (0,.5) {}; \node[ext,label=0:{$\scriptstyle y$}] (v2) at (1,-.5) {}; \node[ext,label=0:{$\scriptstyle z$}] (v3) at (2,-.5) {}; \node (v0) at (0,1.2) {0}; \draw (v1) edge node[below] {$\scriptstyle *$} (v2) edge (v0) edge +(-.5,-.5) edge +(-.25,-.5) edge +(0,-.5) edge node[above] {$\scriptstyle \rt$} (v3) (v2) edge +(-.3,-.5) edge +(0,-.5) edge +(.3,-.5) (v3) edge +(-.3,-.5) edge +(0,-.5) edge +(.3,-.5); \draw[pbrace, thick] (0,-.1) -- (-.5,-.1); \node at (-.25, -.4) {$(A\sqcup B)^c$}; \draw[pbrace, thick] (1.3,-1.1) -- (.7,-1.1); \node at (1, -1.4) {$A$}; \draw[pbrace, thick] (2.3,-1.1) -- (1.7,-1.1); \node at (2, -1.4) {$B$}; \end{tikzpicture} \end{align*}

In particular, the elements $T_A$ above are expressed through \eqref{equ:TAxydef} as
\[
  T_A= T_A\left(\bs\nabla(x_*\prod_{i\in A^c}x_i),\bs\nabla(\prod_{i\in A}x_i)\right).
\]
We first consider the composition \eqref{equ:circ ch 1} and compute the image of $\sum_A c_A T_{A}$ under $\circ_{\rt}$.
This is 
\begin{align*} &\sum_{A\atop \rt \in A} c_A T_{A\setminus\{\rt\} , S'} \left(\bs\nabla(x_*\prod_{i\in A^c}x_i), \bs u\nabla(\prod_{i\in A}x_i), \bs \nabla(\prod_{i\in S'}x_i)\right) \\
 -~& \sum_{A\atop \rt \in A} c_A T_{A\setminus\{\rt\} , S'} \left(\bs \nabla(x_*\prod_{i\in A^c}x_i), \bs \nabla(\prod_{i\in A}x_i), \bs u\nabla(\prod_{i\in S'}x_i)\right) \\
 +~& \sum_{A\atop \rt \notin A} c_A T_{A, S'}' \left(\bs u\nabla(x_*\prod_{i\in A^c}x_i), \bs \nabla(\prod_{i\in A}x_i), \bs \nabla(\prod_{i\in S'}x_i)\right) \\
 -~& \sum_{A\atop \rt \notin A} c_A T_{A, S'}' \left(\bs\nabla(x_*\prod_{i\in A^c}x_i), \bs\nabla(\prod_{i\in A}x_i), \bs u\nabla(\prod_{i\in S'}x_i)\right) \\= -~&d_{\tBVt} \left( \sum_{A\atop \rt \in A} c_A T_{A\setminus\{\rt\} , S'} \left(\bs\nabla(x_*\prod_{i\in A^c}x_i), \bs\prod_{i\in A}x_i, \bs\nabla(\prod_{i\in S'}x_i)\right) \right. \\
 -~& \sum_{A\atop \rt \in A} c_A T_{A\setminus\{\rt\} , S'} (\bs\nabla(x_*\prod_{i\in A^c}x_i), \bs\nabla(\prod_{i\in A}x_i), \bs\prod_{i\in S'}x_i) \\
 +~& \sum_{A\atop \rt \notin A} c_A T_{A, S'}' \left(\bs x_*\prod_{i\in A^c}x_i, \bs\nabla(\prod_{i\in A}x_i), \bs\nabla(\prod_{i\in S'}x_i)\right) \\
 -~& \left. \sum_{A\atop \rt \notin A} c_A T_{A, S'}' \left(\bs\nabla(x_*\prod_{i\in A^c}x_i), \bs\nabla(\prod_{i\in A}x_i), \bs\prod_{i\in S'}x_i\right) \right) \\=: -~&d_{\tBVt} (X), \end{align*}
where we abbreviated the last term in brackets by $X$.
The result is cohomologous to $d_cX$.
The edge contraction $d_c$ can act on each tree occurring in $X$, and contract one of two edges.
Some of the terms yield a trivial contribution due to the closedness $d_c\sum c_A T_A=0$. 
Concretely, contracting an edge between two vertices with degree 1 decorations yields a trivial contribution for that reason.
Next, the terms contracting an edge towards the $S'$-vertex can be handled as in the proof of Proposition \ref{prop:circ is comp Com} above. 
We obtain
\begin{align*} d_cX = & \sum_{A\atop \rt \in A} c_A T_{ S'} \left(\sum_{i<j\in S\atop i\text{ or }j\in A^c} \bs\Omega_{ij}, \bs\nabla(\prod_{i\in S'}x_i)\right) - \sum_{A\atop \rt \in A} c_A T_{A\sqcup S'\setminus\{\rt\} } \left(\bs\nabla(x_*\prod_{i\in A^c}x_i), \bs\nabla(\prod_{i\in A\sqcup S'\setminus\{\rt\}}x_i)\right) \\-& \sum_{A\atop \rt \notin A} c_A T_{S'} \left(\sum_{i<j\in A} \bs\Omega_{ij}, \bs\nabla(\prod_{i\in S'}x_i)\right) - \sum_{A\atop \rt \notin A} c_A T_{A} \left(\bs\nabla(x_*\prod_{i\in A^c \sqcup S'\setminus \{\rt\} } x_i), \bs\nabla(\prod_{i\in A}x_i)\right) \end{align*}
The first and third term together have the form
\begin{gather*} \sum_{A\atop \rt \in A} c_A T_{ S'} \left(\sum_{i<j\in S\atop i\text{ or }j\in A^c} \bs\Omega_{ij}, \bs\nabla(\prod_{i\in S'}x_i)\right) - \sum_{A\atop \rt \notin A} c_A T_{S'} \left(\sum_{i<j\in A} \bs\Omega_{ij}, \bs\nabla(\prod_{i\in S'}x_i)\right) \\= T_{S'}\left(\sum_{i<j\in S} \alpha_{ij}\bs \Omega_{ij}, \bs \nabla(\prod_{i\in S'}x_i)\right) \end{gather*}
with 
\begin{equation}\label{equ:alpha ij def}
\begin{aligned}
\alpha_{ij}
&=
\sum_{A\atop {\rt \in A \atop i\text{ or }j\in A^c}} c_A 
-
\sum_{A\atop {\rt \notin A \atop i,j\in A}} c_A 
\\&=
\sum_{A\atop {\rt \in A \atop i\in A^c}} c_A 
+
\sum_{A\atop {\rt,i \in A \atop j\in A^c}} c_A 
-
\sum_{A\atop {\rt \notin A \atop i,j\in A}} c_A.
\end{aligned}
\end{equation}
By symmetry in $i,j$ we may assume that $i\neq \rt$, and $i\neq j$ holds anyway.
Then by the first assertion of Lemma \ref{lem:hycom technical} the last two summands above cancel.
Hence we have that 
\[
  \alpha_{ij}= \sum_{A\atop {\rt \in A \atop i\in A^c}} c_A
  = \psi_{\rt}(\underline c),
\]
using the definition of the dual $\psi$-classes \eqref{equ:psistar def} and Lemma \ref{lem:hycom technical}.
Furthermore, comparing with \eqref{equ:Hycom com action}, $\circ_{i,j}$ is indeed minus the $\Com$-action.

Finally, we consider the other composition \eqref{equ:circ ch 2}, that can be handled similarly.
We obtain the linear combination of decorated trees
\begin{gather*} \sum_A c_A \left(T_{A^c,A}\left(\bs u\nabla(x_*\prod_{i\in S}x_i), \bs\nabla(x_\rt\prod_{i\in A^c}x_i),\bs\nabla(\prod_{i\in A}x_i) \right) \right. \\
\left. - T_{A^c,A}\left(\bs\nabla(x_*\prod_{i\in S}x_i), \bs u\nabla(x_\rt\prod_{i\in A^c}x_i),\bs \nabla(\prod_{i\in A}x_i) \right) \right) \\= -d_{\BV'_{\geq 3}} \left( \sum_A c_A \left(T_{A^c,A}\left(\bs x_*\prod_{i\in S}x_i, \bs\nabla(x_\rt\prod_{i\in A^c}x_i),\bs\nabla(\prod_{i\in A}x_i) \right) \right. \right. \\
 - \left. \left. T_{A^c,A}\left(\bs \nabla(x_*\prod_{i\in S}x_i),\bs x_\rt\prod_{i\in A^c}x_i,\bs \nabla(\prod_{i\in A}x_i) \right) \right) \right) \\=: -d_{\BV'_{\geq 3}} X'. \end{gather*}
This is cohomologous to $d_cX'$. Again the edge contraction on the edge $\rt$ of the first summand does not yield a contribution due to the closedness $d_c\sum_Ac_AT_A=0$. The other terms give:
\begin{align*} d_c X' &= \sum_A c_A\left( -T_{A}\left(\sum_{i<j \in A^c\sqcup\{*\}} \bs\Omega_{ij},\bs\nabla(\prod_{i\in A}x_i) \right) - T_{A}\left(\sum_{i<j\in S\sqcup A^c \sqcup\{*\}\atop \text{$i$ or $j$}\in S} \bs\Omega_{ij}, \bs\nabla(\prod_{i\in A}x_i) \right) \right. \\&\quad\quad\quad\left. +T_{A\sqcup A^c}\left(\bs\nabla(x_*\prod_{i\in S}x_i), \bs(\prod_{i\in A^c}x_i)\nabla(\prod_{i\in A}x_i) \right) \right) \\&= -\sum_A c_A T_{A}\left(\bs\nabla( x_*\prod_{i\in S\sqcup A^c}x_i), \bs\nabla(\prod_{i\in A}x_i) \right) + \sum_{j<k} T_{A\sqcup A^c}\left(\bs\nabla(x_*\prod_{i\in S}x_i), \bs\Omega_{ij} \sum_{A\atop j,k\in A}c_A\right) . \end{align*}
We note that $\sum_{A\atop j,k\in A}c_A=\psi_0(\underline c)$ is independent of the choice of $j,k$, see Lemma \ref{lem:hycom technical} and the definition of $\psi_0$ in \eqref{equ:psistar def}.
Comparing to \eqref{equ:Hycom com action} we get the desired formula.
\end{proof}

\begin{prop}\label{prop:circ is comp ComCom}
  The operation 
  \begin{multline*}
    \circ_{i,j}:
    \gr_{2n-2} H^{2n-2}( \Feyn_{\kk}(\tGrav)(\!(1,n)\!) )
    \otimes
    \gr_{2n'-6} H^{2n'-6}( \Feyn_{\kk}(\tGrav)(\!(0,n')\!) )
\\
\cong \Com(\!(n+2)\!)_{S_2} \otimes \Com(\!(n')\!)
   \\ \to 
    \gr_{2n+2n'-6} H^{2n+2n'-6}( \Feyn_{\kk}(\tGrav)(\!(1,n+n'-2)\!) )
    \cong \Com(\!(n+n')\!)_{S_2}
  \end{multline*}
  agrees with minus the cyclic operadic composition in $\Com$, using the identification \eqref{equ:Fgrav hw}.
\end{prop}
\begin{proof}
The cohomology $\gr_{2n-2} H^{2n-2}( \Feyn_{\kk}(\tGrav)(\!(1,n)\!) )$ is spanned by graphs 
with one vertex decorated by the top-weight part of $\tGrav$, which is isomorphic to $\Com$, and a tadpole. 
Hence we have that 
$$
\gr_{2n-2} H^{2n-2}( \Feyn_{\kk}(\tGrav)(\!(1,n)\!) )
\cong 
\gr_{2n-2} H^{2n-2}( \Feyn_{\kk}(\tGrav)(\!(0,n+2)\!) )_{S_2}.
$$
We can then just use the (genus 0) result of Proposition \ref{prop:circ is comp Com} to conclude that the composition is indeed minus the composition of $\Com$.
\end{proof}

\begin{prop}\label{prop:Delta is psi}
  The operation 
\begin{gather*} \Delta_j: \gr_{2n-8} H^{2n-8}( \Feyn_{\kk}(\tGrav)(\!(0,n)\!) ) \cong \HyCom(\!(n)\!)^{-2} \\
 \qquad \qquad \qquad \qquad \to \gr_{2n-6} H^{2n-6}( \Feyn_{\kk}(\tGrav)(\!(0,n)\!) ) \cong \Com(\!(n)\!)\cong \Q \intertext{ is dual to the map } \Q \cong \Com(\!(n)\!)^* \cong H^0(\MM_{0,n}) \xrightarrow{\psi_j\wedge -} \gr_2 H^2(\MM_{0,n}) \end{gather*}
  of multiplication with the $\psi$-class at the $j$-th marking.
\end{prop}
\begin{proof}
As before, we begin with a representative 
\[
y := \sum_A c_A T_A  
\]
of a cohomology class in $ \gr_{2n-6} H^{2n-6}( \Feyn_{\kk}(\tBVt)(\!(0,n+1)\!) \cong \gr_{2n-6} H^{2n-6}( \Feyn_{\kk}(\tGrav)(\!(0,n+1)\!) )$.
In particular, the coefficients $c_A$ satisfy \eqref{equ:Hycom condition}.

First assume that $j=0$. 
Then 
\[
  \Delta_0 y = -\sum_A c_A T_A\left(\bs u\nabla(x_*\prod_{i\notin A}x_i), \bs\nabla(\prod_{i\in A}x_i)\right)
  =
  d_{\tBVt}
  \sum_A c_A T_A\left(\bs x_*\prod_{i\notin A}x_i, \bs \nabla(\prod_{i\in A}x_i)\right).
\]
This is cohomologous to 
\[
-d_c  \sum_A c_A T_A\left(\bs x_*\prod_{i\notin A}x_i, \bs \nabla(\prod_{i\in A}x_i)\right)
=
\sum_A c_A T\left(\sum_{i<j\in A}\bs \Omega_{ij}\right)
=
\sum_{i<j}\left(\sum_{A\atop i,j\in A} c_A \right) T(\bs\Omega_{ij})
\]
But the term in parentheses is independent of $i,j$ by \eqref{equ:Hycom condition}, and equal to $\psi_0(y)$.
Hence $\Delta_0 y$ is cohomologous to 
\[
  \psi_0(x) T\left(\bs\nabla(\prod_{i<j\in S} x_i)\right)
\]
and the result is shown in this case.

Next suppose that $j>0$. 
Then 
\begin{align*} \Delta_j y &= \sum_{A\atop j\in A} c_A T_A\left(\bs\nabla(x_*\prod_{i\notin A}x_i), \bs u\nabla(\prod_{i\in A}x_i)\right) + \sum_{A\atop j\notin A} c_A T_A\left(\bs u\nabla(x_*\prod_{i\notin A}x_i), \bs \nabla(\prod_{i\in A}x_i)\right) \\&= -d_{\tBVt} \underbrace{ \left( \sum_{A\atop j\in A} c_A T_A\left(\bs \nabla(x_*\prod_{i\notin A}x_i), \bs \prod_{i\in A}x_i\right) + \sum_{A\atop j\notin A} c_A T_A\left(\bs x_*\prod_{i\notin A}x_i, \bs \nabla(\prod_{i\in A}x_i)\right) \right) }_{=:y'} . \end{align*}
Hence $\Delta_j y$ is cohomologous to $d_c y'$. We compute further:
\begin{align*} d_c y' &= + \sum_{A\atop j\in A} c_A T\left(\sum_{i<k\atop i\text{ or }k\in A^c}\bs\Omega_{ik}\right) - \sum_{A\atop j\notin A} c_A T\left(\sum_{i<k\in A}\bs\Omega_{ik}\right) \\&= \sum_{i<k} \alpha_{ik} T(\bs\Omega_{ik}) \end{align*}
with 
$$
\alpha_{ik}= \sum_{A\atop {j\in A \atop i\text{ or }k\in A^c }} c_A
-
\sum_{A\atop {j\in A^c \atop i, k\in A }} c_A.
$$
Analogously to the derivation below \eqref{equ:alpha ij def}
we see that $\alpha_{ik}=\psi_j(y)$, independently of $i,k$, and we are done.

\end{proof}

\begin{remark}
We do not consider the explicit form of $\eta_{i,j}$ here, because, as we shall see, the map will not be important for our cohomology computations.
\end{remark}

\section{Proof of Theorems \ref{thm:main}, \ref{thm:dbv_hycom} and \ref{thm:dbv_hycom2}}

\subsection{Combinatorial Description of \texorpdfstring{$\Feyn_{\kk}'(D\BV^*)(\!(g,r)\!)$}{Feyn'(DBV*)((g,r))} and \texorpdfstring{$\AFeyn_{\kk}'(D\BV^*)(\!(g,r)\!)$}{AFeyn'(DBV*)}}
\label{sec:feinp combinatorial}
We shall describe the complex $\Feyn_{\kk}'(D\BV^*)(\!(g,r)\!)$ combinatorially.
Elements can be seen as linear combinations of decorated connected graphs of loop order $g$ with $r$ legs.
The vertices $v$ of valence $r_v\geq 3$ are decorated by elements of $D\BV^*(\!(r_v)\!)$. Vertices of valence two are implicitly decorated by the dual BV operator $\Dc$.
No bivalent vertex can have a neighbour that is also a bivalent vertex.
Hence, instead of drawing bivalent vertices, we may equivalently consider $\geq $trivalent graphs some of whose edges and legs are marked.
A marked edge represents a pair of edge with a $\Dc$-decorated bivalent vertex in the middle, as indicated in the following picture.
\[
\begin{tikzpicture} \node[ext] (v1) at (-.7,-.7) {}; \node[ext] (v2) at (+.7,-.7) {}; \node[ext] (v3) at (-.7,+.7) {}; \node[ext] (v4) at (+.7,+.7) {}; \node[ext] (b1) at (0,+.7) {$\scriptscriptstyle \Dc$}; \node[ext] (b2) at (-1.2,-1.2) {$\scriptscriptstyle \Dc$}; \node (n1) at (-1.8,-1.8) {$1$}; \node (n2) at (-1.2,+1.2) {$2$}; \node (n3) at (+1.2,+1.2) {$3$}; \draw (v1) edge (v2) edge (v3) (b1) edge (v3) edge (v4) (b2) edge (v1) edge (n1) (v4) edge (v2) edge (n3) (v2) edge (v3) (v3) edge (n2); \end{tikzpicture}
\quad
=
\quad
\begin{tikzpicture} \node[ext] (v1) at (-.7,-.7) {}; \node[ext] (v2) at (+.7,-.7) {}; \node[ext] (v3) at (-.7,+.7) {}; \node[ext] (v4) at (+.7,+.7) {}; \node (n1) at (-1.5,-1.5) {$1$}; \node (n2) at (-1.2,+1.2) {$2$}; \node (n3) at (+1.2,+1.2) {$3$}; \draw (v1) edge (v2) edge (v3) edge[very thick] (n1) (v4) edge (v2) edge (n3) edge[very thick] (v3) (v2) edge (v3) (v3) edge (n2); \end{tikzpicture}
\]

The only difference between $\Feyn_{\kk}'(D\BV^*)(\!(g,r)\!)$ and $\AFeyn_{\kk}'(D\BV^*)(\!(g,r)\!)$ is that the graphs generating $\AFeyn_{\kk}'(D\BV^*)(\!(g,r)\!)$ cannot have marked legs, only marked internal edges.

The marked edges carry degree $+1$ and weight $2$, which are the degree and weight of $\Dc$.
The normal edges carry degree and weight 0.

We also note that the decorations by $D\BV^*(\!(r_v)\!)$ on vertices $v$ can be understood as given by a linear combination of trees. 
Hence, we may even consider $\Feyn_{\kk}'(D\BV^*)(\!(g,r)\!)$ as generated by graphs with three different types of edges.
In this picture, the tree-edges carry degree $+1$ and weight $0$.
The decorations by $\BV^*$ on the vertices of the trees carry the cohomological degree of $\BV^*$, and contribute twice that degree to the weight.
But for simplicity, we shall restrict to the two types of edges above in this section, and keep the tree-edges implicit in the decoration of the vertices.

\subsection{Spectral sequence for \texorpdfstring{$\Feyn_{\kk}'(D\BV^*)$}{Feyn'(DBV*)}}

The theorems are shown by studying a natural spectral sequence converging to $\Feyn_{\kk}'(D\BV^*)$, and similarly on $\AFeyn_{\kk}'(D\BV^*)$. We only state explicitly the former case, with the understanding that one can proceed analogously for $\AFeyn_{\kk}'(D\BV^*)$.

Let us define 
\[
  \mF_p \Feyn_{\kk}'(D\BV^*)(\!(g,r)\!) \subset \Feyn_{\kk}'(D\BV^*)(\!(g,r)\!)
\]
to be the subspace spanned by graphs that have at most $p$ bivalent vertices.
This subspace %
is closed under the differential, since the differential (---contracting edges---) cannot increase the number of bivalent vertices.
We may hence consider the spectral sequence $E^\bullet$ associated to this filtration.
By finite dimensionality this spectral sequence converges to the cohomology 
\[
  E^\bullet \Rightarrow H\left( \Feyn_{\kk}'(D\BV^*)(\!(g,r)\!) \right).
\]
We shall consider the various pages of this spectral sequence.
To this end, note that the differential on $\Feyn_{\kk}'(D\BV^*)(\!(g,r)\!)$ has the form 
\[
d = d_{D\BV^*} + d_{0} + d_1
\]
where $d_{D\BV^*}$ is the differential induced from $D\BV^*$, $d_0$ contracts an edge between two vertices of valence $\geq 3$ and $d_1$ contracts an edge between a bivalent and a higher valent vertex.

Similarly, we endow $\AFeyn_{\kk}'(D\BV^*)(\!(g,r)\!)$ with the identically defined filtration, and by the same reasoning we obtain a spectral sequence
\[
  E_A^\bullet \Rightarrow H\left( \AFeyn_{\kk}'(D\BV^*)(\!(g,r)\!) \right).
\]

\subsection{\texorpdfstring{$E^0$}{E0}-page and convergence}
The associated graded complex ($E^0$-page) can be identified with $\Feyn_{\kk}'(D\BV^*)(\!(g,r)\!)$, but equipped with the differential $d_{D\BV^*} + d_{0}$, that is, the part of the original differential that leaves invariant the number of bivalent vertices in graphs.
Recall from Section \ref{sec:feinp combinatorial} that we may understand $\Feyn_{\kk}'(D\BV^*)(\!(g,r)\!)$ as a vector space generated by decorated graphs with two types of edges -- normal and marked.
But any graph with edges of two colors can be considered a nested graph, with the marked edges the edges of the outer graph, and the normal edges being part of graphs decorating every vertex of the outer graph.
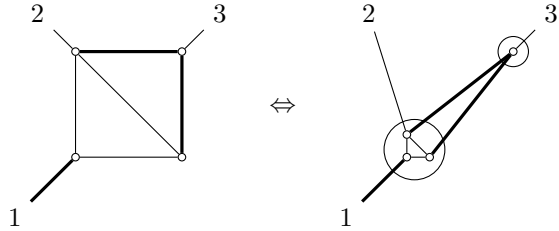
\[
\begin{tikzpicture} \node[ext] (v1) at (-.7,-.7) {}; \node[ext] (v2) at (+.7,-.7) {}; \node[ext] (v3) at (-.7,+.7) {}; \node[ext] (v4) at (+.7,+.7) {}; \node (n1) at (-1.5,-1.5) {$1$}; \node (n2) at (-1.2,+1.2) {$2$}; \node (n3) at (+1.2,+1.2) {$3$}; \draw (v1) edge (v2) edge (v3) edge[very thick] (n1) (v4) edge[very thick] (v2) edge (n3) edge[very thick] (v3) (v2) edge (v3) (v3) edge (n2); \end{tikzpicture}
  \quad
  \Leftrightarrow
  \quad
  \begin{tikzpicture} \node[ext] (v1) at (-.7,-.7) {}; \node[ext] (v2) at (-0.4,-.7) {}; \node[ext] (v3) at (-.7,-0.4) {}; \node[ext] (v4) at (+.7,+.7) {}; \node (n1) at (-1.5,-1.5) {$1$}; \node (n2) at (-1.2,+1.2) {$2$}; \node (n3) at (+1.2,+1.2) {$3$}; \draw (v1) edge (v2) edge (v3) edge[very thick] (n1) (v4) edge[very thick] (v2) edge (n3) edge[very thick] (v3) (v2) edge (v3) (v3) edge (n2); \draw (-.6,-.6) circle (4mm) (.7,.7) circle (2mm); \end{tikzpicture}
\]
Hence we also have the identification of modular sequences
\[
  \Feyn_{\kk}'(D\BV^*)
  = \MFree_1(\Feyn_\kk(D'\BV^*)),
\]
using the construction $\MFree_1$ of \eqref{equ:mfree1 def}.
The differential on the $E^0$ page $d_{D\BV^*} + d_{0}$ only acts on the inner object $\Feyn_{\kk}(D'\BV^*)$, and agrees with the differential there.
Hence by the K\"unneth formula we have 
\[
E^1 := H(\Feyn_{\kk}'(D\BV^*)(\!(g,r)\!), d_{D\BV^*} + d_{0})
\cong \MFree_1(H(\Feyn_{\kk}(D'\BV^*))).
\]
\begin{prop}\label{prop:E2 convergence}
Suppose that there is a $p\geq 0$ such that for all $g,n$ and weights $W\geq 6g-6+2n-2p$ the weight $W$ cohomology of $\Feyn_{\kk}(\tGrav)(\!(g,n)\!)$ is concentrated in degree $W$. 
Then the weight $6g-6+2n-2p$-parts of the spectral sequences above converge to the cohomology on the $E^2$-page, i.e.,
\begin{align*} \gr_{6g-6+2n-2p} H(\Feyn_{\kk}'(D\BV^*)(\!(g,n)\!)) &\cong H( \gr_{6g-6+2n-2p} \MFree_1(H(\Feyn_{\kk}(D'\BV^*))), d_1) \\
\gr_{6g-6+2n-2p} H(\AFeyn_{\kk}'(D\BV^*)(\!(g,n)\!)) &\cong H( \gr_{6g-6+2n-2p} \Free_1(H(\Feyn_{\kk}(D'\BV^*))), d_1) . \end{align*}
\end{prop}
By Corollary \ref{cor:Feyngrav} we hence immediately obtain the following.
\begin{cor}\label{cor:E2 convergence}
The spectral sequences converges on the $E^2$ page in top-  and top$-2$-weight, that is, 
\begin{align*} \gr_{6g-6+4n} H(\Feyn_{\kk}'(D\BV^*)(\!(g,n)\!)) \cong H( \gr_{6g-6+4n} \MFree_1(H(\Feyn_{\kk}(D'\BV^*))), d_1) \\
 \gr_{6g-8+4n} H(\Feyn_{\kk}'(D\BV^*)(\!(g,n)\!)) \cong H( \gr_{6g-8+4n} \MFree_1(H(\Feyn_{\kk}(D'\BV^*))), d_1) \\
 \gr_{6g-6+4n} H(\AFeyn_{\kk}'(D\BV^*)(\!(g,n)\!)) \cong H( \gr_{6g-6+4n} \Free_1(H(\Feyn_{\kk}(D'\BV^*))), d_1) \\
 \gr_{6g-8+4n} H(\AFeyn_{\kk}'(D\BV^*)(\!(g,n)\!)) \cong H( \gr_{6g-8+4n} \Free_1(H(\Feyn_{\kk}(D'\BV^*))), d_1). \end{align*}
\end{cor}

\begin{proof}[Proof of Proposition \ref{prop:E2 convergence}]
We only consider the case of the Feynman transform -- for the amputated Feynman transform the argument is analogous.

  Consider a graph $\Gamma\in \MFree_1(H(\Feyn_{\kk}(D'\BV^*)))(\!(g,n)\!)$.
Suppose it has $N$ vertices, with each vertex $v$ decorated by an element of some 
$\gr_{6g_v-6+2n_v-2p_v}H(\Feyn_{\kk}(D'\BV^*)(\!(g_v,n_v)\!))$.
Suppose that $\Gamma$ has $e$ internal edges and there are markings on $n'$ of the $n$ external legs.
Then the genus of $\Gamma$ is 
\[
  g =   e+1 + \sum_v (g_v-1)
\]
and the weight of $\Gamma$ is 
\begin{equation}\label{equ:WGamma}
  W_\Gamma := 2e + 2n' + \sum_v (6g_v-6+2n_v-2p_v)
  =
  6g-6+4n - 2\left(n-n' + \sum_v p_v\right).
\end{equation}
Here we also use that $\sum_v n_v= n+2e$, the number of half-edges.
That means that to make a weight $6g-6+4n -2p$ graph, all the numbers $p_v$ at the vertices must be $\leq p$.
Hence we can use the assumption that the cohomology of $H(\Feyn(D'\BV^*))$ in those weights is concentrated in degree equal to the weight. 
Concretely, suppose that $\Gamma$ has weight $W_\Gamma \geq 6g-6+4n -2p$.
Then it has degree 
\[
  e + n' + \sum_v (6g_v-6+2n_v-2p_v)
  = W_\Gamma - e - n'.
\]
The differential on the $E^q$-page of the spectral sequence reduces the number (of bivalent vertices) $e + n'$ by exactly $q$.
But the same differential leaves invariant the weight, and increases the degree by +1. Hence by the above formulas for weight and degree this differential needs to vanish unless $q=1$.
\end{proof}

\subsection{Proof of parts (1) and (2) of Theorem \ref{thm:main}}
The proof of part (2) of Theorem  \ref{thm:main} trivially follows from the fact that $\gr_0\BV^*=\Com^*$.

Part (1) for the non-amputated case $\Feyn(D\BV^*)$ follows immediately from the formula \eqref{equ:WGamma} of the weight $W_\Gamma$ of a graph $\Gamma\in \MFree_1(H(\Feyn_{\kk}(D'\BV^*)))(\!(g,n)\!)$, from which it is clear that $W_\Gamma\leq 6g-6+4n$.

The analogous formula in the amputated case reads 
\[
  W_\Gamma = 6g-6+2n - \sum_v p_v,
\]
since no leg can be marked. From this it is clear that the nontrivial weights in $H(\AFeyn(D\BV^*))$ must be $\leq 6g-6+2n$.

\subsection{\texorpdfstring{$E^1$}{E1} page in top-weight and proof of part (3) of Theorem \ref{thm:main} }
\label{sec:thm main 3 proof}
By Corollary \ref{cor:E2 convergence} we have that the top-weight part $\gr_{6g-6+4n} H(\Feyn_{\kk}'(D\BV^*)(\!(g,n)\!))$ agrees with the cohomology of the $E^1$-page $\gr_{6g-6+2n-2p} \MFree_1(H(\Feyn_{\kk}(D'\BV^*)))$ with differential $d_1$. 
By the proof of Proposition \ref{prop:E2 convergence} above we know that graphs contributing to the top-weight piece of $\MFree_1(H(\Feyn_{\kk}(D'\BV^*)(\!(g,n)\!)))$ must have all external legs marked and each vertex $v$ must be decorated by an element of the top-weight cohomology $\gr_{6g_v-6+2n_v}H^{k_v}(\Feyn_{\kk}(D'\BV^*)(\!(g_v,n_v)\!))$.
But by \eqref{equ:Fgrav hw} 
this cohomology vanishes unless each $g_v=0$ and $k_v=2n_v-6$, and then it can be identified with $\Com(\!(n_v)\!)\cong \Q$.

Hence, as a graded vector space we may identify 
\begin{equation}\label{equ:top wt proof aux}
  \gr_{6g-6+4n} \MFree_1(H(\Feyn_{\kk}(D'\BV^*)(\!(g,n)\!)))
  \cong \Feyn(\Com)(\!(g,n)\!) [-6g+6-3n] \otimes \sgn_n
\end{equation}
with the commutative graph complex, up to a degree shift and an extra sign factor if one also wants to preserve the $S_n$ action. 
To see that the degree shift is indeed by $6g-6+3n$
consider a graph $\Gamma$ of genus $g$ with $n$ external legs and $e$ (non-leg-)edges.
The cohomological degree of this graph in $\Feyn(\Com)(\!(g,n)\!)$ is $-e$.
The degree of the same graph $\Gamma$ interpreted as an element of $\MFree_1(H(\Feyn_{\kk}(D'\BV^*)(\!(g,n)\!))$ is 
\[
\sum_{v\in V\Gamma}(2n_v-6) + e + n
=4e+2n - 6|V\Gamma| +e+n
=
6g-6+3n-e,
\]
so that the degree shift is indeed $6g-6+3n$.

Next, we study the differential $d_1$ on the $E_1$ page.
The differential acts by removing one bivalent vertex, contracting an edge incident to that bivalent vertex.
There are three cases to consider:
\begin{enumerate}
\item[(i)] The bivalent vertex sits between two different vertices.
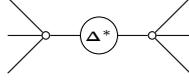
\[
\begin{tikzpicture} \node[ext] (u) at (0,0) {$\scriptscriptstyle \Dc$}; \node[ext] (v1) at (-.7,0) {}; \node[ext] (v2) at (0.7,0) {}; \draw (u) edge (v1) edge (v2) (v1) edge +(-.5,-.5) edge +(-.5,.5) edge +(-.5,0) (v2) edge +(.5,-.5) edge +(.5,.5) edge +(.5,0); \end{tikzpicture}
\]
\item[(ii)] The bivalent vertex connects twice to the same vertex.
 
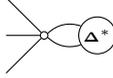
\[
  \begin{tikzpicture} \node[ext] (u) at (0,0) {$\scriptscriptstyle \Dc$}; \node[ext] (v1) at (-.7,0) {}; \draw (u) edge[bend left] (v1) edge[bend right] (v1) (v1) edge +(-.5,-.5) edge +(-.5,.5) edge +(-.5,0) ; \end{tikzpicture}
  \]
\item[(iii)] The bivalent vertex is at an external leg.
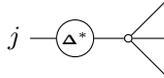
\[
\begin{tikzpicture} \node[ext] (u) at (0,0) {$\scriptscriptstyle \Dc$}; \node (v1) at (-.8,0) {$j$}; \node[ext] (v2) at (0.7,0) {}; \draw (u) edge (v1) edge (v2) (v2) edge +(.5,-.5) edge +(.5,.5) edge +(.5,0); \end{tikzpicture}
\]
\end{enumerate}

We shall accordingly split our differential $d_1$ into three terms 
\[
d_1 = d_1^{(i)} + d_1^{(ii)} + d_1^{(iii)}
\]
that we will discuss one-by-one.
First, in the top-weight case we have $d_1^{(ii)}=0$, since it produces a vertex of positive genus, and there are no such in top-weight.
Similarly, $d_1^{(iii)}$ produces a vertex $v$ with decoration of weight $6g_v-6+2n_v+2$, and again there are no possible decorations of that weight, the maximum being $6g_v-6+2n_v$.

Hence the only remaining piece is $d_1 = d_1^{(i)}$.
Comparing the definitions, this piece agrees with the operation $\circ_{i,j}$ of Section \ref{sec:operations ftgrav} applied to the two vertices $v,w$ neighboring our bivalent vertex.
But by Proposition \ref{prop:circ is comp Com} this operation agrees with minus the operadic composition in $\Com$.
Hence the identification \eqref{equ:top wt proof aux} is indeed an identification of dg vector spaces, if one considers the Feynman transform with minus the usual differential -- but that we can do, see Remark \ref{rem:feyn minus diff}.
Putting all pieces together we have shown that 
\[
  \gr_{6g-6+4n} H( \Feyn_{\kk}(D'\BV^*)(\!(g,n)\!) )
  \cong H(\Feyn(\Com) (\!(g,n)\!) ) [-6g+6-3n] \otimes \sgn_n
\]
as desired.

The analogous amputated case is shown in the same manner. One just does not have the markings on external legs, and hence misses the corresponding degree shift of $n$ and weight shift of $2n$, and the multiplication by the sign representation of $S_n$. This then finishes the proof of Theorem \ref{thm:main}.
\hfill\qed

\subsection{Top\texorpdfstring{$-2$}{-2} weight and proof of Theorem \ref{thm:dbv_hycom2}}
\label{sec:thm dbv_hycom2 proof}
We may proceed similarly for the case of top$-2$-weight.
That is, consider a graph $\Gamma\in \gr_{6g-8+4n} \MFree_1(H(\Feyn_{\kk}(D'\BV^*)(\!(g,n)\!)))$.
Again by the analysis of the proof of Proposition \ref{prop:E2 convergence} there are three cases to be considered:
\begin{enumerate}
\item[A.] The graph $\Gamma$ has a single vertex $v$ decorated by the top$-2$-weight cohomology 
$$\gr_{6g_v-8+2n_v}H(\Feyn_{\kk}(D'\BV^*)(\!(g_v,n_v)\!)).$$ All other vertices $v'$ are decorated by the top-weight cohomology.
That is, arguing as in the previous subsection, all the other vertices are of genus zero and trivially decorated by $\Q$.
All legs of $\Gamma$ are marked.
By \eqref{equ:Fgrav hw2} we also know that the special vertex must have genus either zero or one, and we may accordingly distinguish two sub-cases:
\begin{enumerate}
  \item[A0.] Here the special vertex $v$ has genus zero, i.e., a decoration in $$\gr_{2n_v-8}H^{2n_v-8}(\Feyn_{\kk}(D'\BV^*)(\!(0,n_v)\!))\cong \HyCom(\!(n_v)\!)^2.$$
  \item[A1.] The special vertex $v$ has genus one, i.e., a decoration in $\gr_{2n_v-2}H(\Feyn_{\kk}(D'\BV^*)(\!(1,n_v)\!))\cong \Com(\!(n_v+2)\!)_{S_2}$.
\end{enumerate}
\item[B.] All vertices of $\Gamma$ have top weight, and are hence of genus zero and trivially ($\Q$-)decorated.
All legs of $\Gamma$ except for one, say the $j$-th, are marked.
\end{enumerate}

Accordingly, we split 
\[
  \gr_{6g-8+4n} \MFree_1(H(\Feyn_{\kk}(D'\BV^*)(\!(g,n)\!))) := V_{A0}(\!(g,n)\!)\oplus V_{A1}(\!(g,n)\!) \oplus V_{B}(\!(g,n)\!)
\]
into graded subspaces, each spanned by graphs of the respective type.
As in the previous subsection we split the differential as $d_1=d_1^{(i)}+d_1^{(ii)}+d_1^{(iii)}$.
These pieces of the differential map between the subspaces as follows:
\[
\begin{tikzcd}
  V_{A0}(\!(g,n)\!) \ar[loop above]{}{d_1^{(i)}} 
  \ar{r}{d_1^{(ii)}} \ar{d}{d_1^{(iii)}}
  &
  V_{A1}(\!(g,n)\!) \ar[loop above]{}{d_1^{(i)}} 
  \\
  V_{B}(\!(g,n)\!) \ar[loop below]{}{d_1^{(i)}}
\end{tikzcd}
\]

Again we may use Propositions \ref{prop:circ is comp Com}-\ref{prop:circ is comp ComCom} to obtain a very explicit combinatorial description of the different pieces of the differential:
In particular, elements of $V_{A1}(\!(g,n)\!)$ can be seen as undecorated graphs (as in $\Feyn(\Com)$), but with one special vertex. All vertices are required to have valence $\geq 3$, except the special vertex which may also have valence $1$ or $2$.
The differential $d_1^{(i)}$ acts by contracting an edge between an arbitrary pair of distinct vertices.

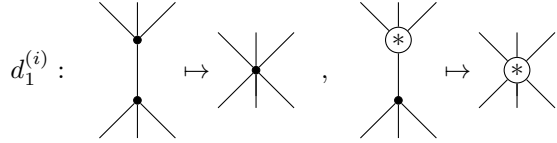
\begin{align*} d_1^{(i)}: \quad \begin{tikzpicture} \node[int] (v) at (0,0.4) {}; \node[int] (w) at (0,-0.4) {}; \draw (v) edge (w) edge +(.5,.5) edge +(0,.5) edge +(-.5,.5) (w) edge +(.5,-.5) edge +(0,-.5) edge +(-.5,-.5); \end{tikzpicture} \mapsto \begin{tikzpicture} \node[int] (v) at (0,0) {}; \draw (v) edge (w) edge +(.5,.5) edge +(0,.5) edge +(-.5,.5) edge +(.5,-.5) edge +(0,-.5) edge +(-.5,-.5); \end{tikzpicture} \quad, \quad \begin{tikzpicture} \node[ext] (v) at (0,0.4) {$*$}; \node[int] (w) at (0,-0.4) {}; \draw (v) edge (w) edge +(.5,.5) edge +(0,.5) edge +(-.5,.5) (w) edge +(.5,-.5) edge +(0,-.5) edge +(-.5,-.5); \end{tikzpicture} \mapsto \begin{tikzpicture} \node[ext] (v) at (0,0) {$*$}; \draw (v) edge (w) edge +(.5,.5) edge +(0,.5) edge +(-.5,.5) edge +(.5,-.5) edge +(0,-.5) edge +(-.5,-.5); \end{tikzpicture} \end{align*}

\begin{lemma}
The complex $(V_{A1}(\!(g,n)\!),d_1^{(i)})$ is acyclic except for the case of $(g,n)=(1,1)$, when it has one-dimensional cohomology spanned by the graph 
\begin{equation}\label{equ:VA1 gen}
\begin{tikzpicture} \node[ext] (v) at (0,0) {$*$}; \node (e) at (0,-.8) {$1$}; \draw (v) edge (e); \end{tikzpicture}
\end{equation}
\end{lemma}
\begin{proof}
This statement has appeared in the literature on graph complexes in many closely analogous forms. We reproduce here the standard argument, originally due to Lambrechts-Volic \cite{LambrechtsVolic}, for the reader's convenience.

First, consider the decomposition
\[
  V_{A1}(\!(g,n)\!) =
  \begin{tikzcd}[column sep = 0em]
   V_{A1}' \ar[bend left]{rr}{f} \ar[loop above]{} & \oplus  & V_{A1}''\ar[loop above]{}
  \end{tikzcd}
\]
with $V_{A1}'\subset V_{A1}(\!(g,n)\!)$ the subspace spanned by graphs for which the special vertex has valence $1$, and $V_{A1}''\subset V_{A1}$ the subspace spanned by graphs for which the special vertex has valence $\geq 2$. The arrows indicate the different pieces of the differential $d_1^{(i)}$ between the subspaces.
In particular the piece $f$ of the differential contracts the unique edge at the special vertex 
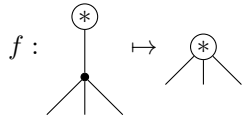
\[
f : 
\begin{tikzpicture} \node[ext] (v) at (0,.4) {$*$}; \node[int] (e) at (0,-.4) {}; \draw (v) edge (e) (e) edge +(-.5,-.5) edge +(0,-.5) edge +(.5,-.5); \end{tikzpicture}
\mapsto 
\begin{tikzpicture} \node[ext] (v) at (0,0) {$*$}; \draw (v) edge +(-.5,-.5) edge +(0,-.5) edge +(.5,-.5); \end{tikzpicture}
\]
The map $f$ is surjective for all $(g,n)$, and injective for $(g,n)\neq (1,1)$. For $(g,n)=(1,1)$ the kernel is spanned by the graph \eqref{equ:VA1 gen}. From these observations and simple homological algebra considerations (see \cite[Lemma 2.1]{PayneWillwacher}) the lemma follows. 
\end{proof}

\begin{cor}\label{cor:pre dbvhycom 2}
For $(g,n)\neq (1,1)$ the projection 
\[
  V_{A0}(\!(g,n)\!)\oplus V_{A1}(\!(g,n)\!) \oplus V_{B}(\!(g,n)\!) \to V_{A0}(\!(g,n)\!) \oplus V_{B}(\!(g,n)\!)
\]
is a quasi-isomorphism of dg vector spaces and hence 
\[
  \gr_{6g-8+4n} H(\Feyn'(D\BV^*)(\!(g,n)\!)) \cong  H(V_{A0}(\!(g,n)\!) \oplus V_{B}(\!(g,n)\!), d_1^{(i)}+d_1^{(iii)}).
\]
\end{cor}

We next consider the individual summands of the complex in the corollary.
\begin{lemma} \label{lem:VAVB iso}
  We have the isomorphisms of complexes 
\begin{align} \label{equ:VAVB iso 1} (V_{A0}(\!(g,n)\!),d_1^{(i)}) &\cong \gr_2 \Feyn(\HyCom)(\!(g,n)\!)[-6g+6-3n] \\
 \label{equ:VAVB iso 2} (V_{B}(\!(g,n)\!),d_1^{(i)}) &\cong \oplus_{j=1}^n \Feyn(\Com)(\!(g,n)\!)[-6g+7-3n]. \end{align}
\end{lemma}
\begin{proof}
Case $V_{A0}$: The graded vector space $V_{A0}(\!(g,n)\!)$ is generated by graphs with one special vertex $v$ that is decorated by an element of $\gr_{2n_v-8}H^{2n_v-8}(\Feyn_{\kk}(D'\BV^*)(\!(0,n_v)\!))\cong \HyCom^2(\!(n_v)\!)$, while all other (at least trivalent) vertices $w$ are decorated by elements of $\gr_{2n_w-6}H^{2n_v-6}(\Feyn_{\kk}(D'\BV^*)(\!(0,n_w)\!))\cong \Q$.
The edges and legs are all marked, and hence carry degree $+1$ and weight $2$.

On the other hand, $\gr_2 \Feyn(\HyCom)(\!(g,n)\!)$ is generated by graphs with one special vertex decorated by $\HyCom^2(\!(n_v)\!)$, and all other vertices decorated by $\Q$. 
The edges carry degree $+1$.

Hence we can identify $V_{A0}(\!(g,n)\!)$ and $\gr_2 \Feyn(\HyCom)(\!(g,n)\!)$ as graded vector spaces, up to a degree shift.
To compute the degree shift, consider a graph $\Gamma\in \gr_2 \Feyn(\HyCom)(\!(g,n)\!)$ with $e$ (non-leg-)edges. Its degree in $\gr_2 \Feyn(\HyCom)(\!(g,n)\!)$ is $-e-2$.
On the other hand, the same graph considered as a generator of 
$V_{A0}(\!(g,n)\!)$ has degree 
\[
  2n_v-8+
\sum_{w\in V\Gamma\atop w\neq v} (2n_w-6)
+ 
e+ n 
=
6g-6 +3n -e -2,
\]
so that the degree shift is indeed $6g-6+3n$.

It remains to check that the differential on both complexes is the same. But this is a direct consequence of Proposition \ref{prop:circ is comp HyCom}, arguing analogously to the proof of part (3) of Theorem \ref{thm:main} in Section \ref{sec:thm main 3 proof}. 

\smallskip

Case $V_{B}$: Here we proceed similarly. 
Generators of $V_{B}(\!(g,n)\!)$ are graphs with all vertices $w$ decorated by an element of $\gr_{2n_v-8}H^{2n_v-6}(\Feyn_{\kk}(D'\BV^*)(\!(0,n_w)\!))\cong \Q$.
All edges and legs are marked, except for one leg.
On the other hand, each summand of $\oplus_{j=1}^n \Feyn(\Com)(\!(g,n)\!)$ is generated by graphs all of whose vertices are decorated by $\Com(\!(n_w)\!)\cong \Q$. 
We think of the $j$-th summand of the direct sum as generated by graphs with the $j$-th leg distinguished.
Hence we have an identification of graded vector spaces 
\[
  V_{B} \cong \oplus_{j=1}^n \Feyn(\Com)(\!(g,n)\!)[-6g+7-3n],
\]
with the $j$-th summand on the right-hand side corresponding to the sub-vector space of $V_{B}$ spanned by graphs whose leg $j$ is not marked. (Remember that the legs are numbered $1,\dots, n$.)  
To check the degree shift above, we consider again a generating graph $\Gamma\in \Feyn(\Com)(\!(g,n)\!)$ with $e$ edges. It has degree $-e$ in $\Feyn(\Com)(\!(g,n)\!)$.
Considered as an element of $V_B(\!(g,n)\!)$ it has degree
\[
\sum_{w\in V\Gamma} (2n_v-6) + e+n-1
=
6g-6+3n -e-1,
\]
so that the degree shift is $6g-7+3n$.
Again, we need to check that the differentials on both sides of \eqref{equ:VAVB iso 2} agree. But this follows from Proposition \ref{prop:circ is comp ComCom}, again as in Section \ref{sec:thm main 3 proof}.
\end{proof}

By Lemma \ref{lem:simple tri} the previous lemma implies that
\[
  \gr_2 H^k(\Feyn'(D\BV^*(\!(g,n)\!)))
  \cong \ker [d_1^{(iii)}]^k \oplus \coker [d_1^{(iii)}]^{k-1}
\]
with 
\[
  [d_1^{(iii)}]^k
  :
  H^k(V_{A0}) \to H^{k+1}(V_B)
\]
the cohomology map induced by $d_1^{(iii)}$.
We investigate further the map $[d_1^{(iii)}]^k$.
To this end let $\Gamma\in \Feyn(\HyCom)(\!(g,n)\!)^k$ be a generator and fix some $j\in\{1,\dots,n\}$.
Then we define 
\[
\psi_j^* \Gamma \in \Feyn(\HyCom)(\!(g,n)\!)^{k+2}
\]
as follows.
\begin{itemize}
\item If the vertex $v$ of $\Gamma$ adjacent to leg $j$ is not the special vertex, we set $\psi_j^* \Gamma :=0$.
\item Suppose that the vertex $v$ of $\Gamma$ adjacent to leg $j$ is the special vertex, and decorated by $x\in \HyCom(\!(n_v)\!)^{-2}$. The we define $\psi_j^* \Gamma$ as the same graph $\Gamma$, except that the decoration $x$ at $v$ is replaced by $\psi_j(x)$, see \eqref{equ:psistar def} above. This is the dual operation of multiplying with the $\psi$-class at marking $j$, hence we use the notation ``$\psi_j^*\Gamma$''.
\end{itemize}

Finally, we assemble the maps $\psi_j^*$ for $j=1,\dots,n$ into one morphism
\begin{gather*} \Psi^k_{\wedge} \colon \gr_2H^k(\Feyn(\HyCom)(\!(g,n)\!)) \to \bigoplus_{j=1}^n H^{k+2}(\Feyn(\Com)(\!(g,n)\!)) \\
\Psi_\wedge^k(\Gamma) = (\psi_1^* \Gamma,\dots, \psi_n^*\Gamma) . \end{gather*}

We then have:
\begin{lemma}\label{lem:d1iii is Psiwedge}
The following diagram commutes:
\[
\begin{tikzcd}
  H^{6g-6+3n-k}(V_{A0}(\!(g,n)\!)) \ar{r}{[d_1^{(iii)}]}
  \ar{d}{\cong}
  &
  H^{6g-5+3n-k}(V_{B}(\!(g,n)\!)) \ar{d}{\cong}
  \\
  \gr_2 H^{-k}(\Feyn(\HyCom)(\!(g,n)\!)) \ar{r}{\Psi_\wedge^{-k}}
  &
  \bigoplus_{j=1}^n
  H^{-k+2}(\Feyn(\Com)(\!(g,n)\!)).
\end{tikzcd}
\]
\end{lemma}
\begin{proof}
This follows directly from Proposition \ref{prop:Delta is psi}.
\end{proof}

\begin{proof}[Proof of Theorem \ref{thm:dbv_hycom2}]

  We just have to assemble the results above.
  By Corollary \ref{cor:E2 convergence} we have that 
  \begin{gather*} \gr_{6g-8+4n} H^{6g-6+3n-k}(\Feyn_{\kk}'(D\BV^*)(\!(g,n)\!)) \\
 \cong H^{6g-6+3n-k}\left(V_{A0}(\!(g,n)\!)\oplus V_{A1}(\!(g,n)\!)\oplus V_{B}(\!(g,n)\!), d_1^{(i)}+d_1^{(ii)}+d_1^{(iii)} \right). \end{gather*}
  By Corollary \ref{cor:pre dbvhycom 2} this is isomorphic to 
  \[
    H^{6g-6+3n-k}\left(V_{A0}(\!(g,n)\!)\oplus V_{B}(\!(g,n)\!), d_1^{(i)}+d_1^{(iii)} \right).
  \]
  By Lemma \ref{lem:simple tri} this is isomorphic to 
  \[
  \ker [d_1^{(iii)}]^{6g-6+3n-k} \oplus \coker [d_1^{(iii)}]^{6g-7+3n-k}.
  \]
  By Lemma \ref{lem:d1iii is Psiwedge} this is in turn isomorphic to 
  \[
  \ker \Psi_\wedge^{-k} \oplus \coker \Psi_\wedge^{-k-1}.\qedhere
  \]
\end{proof}

\subsection{Proof of Theorem \ref{thm:dbv_hycom}}
The amputated analogue of the previous subsection is again easier: Since there are no markings on the external legs, one just omits the summand $V_B$ from the above discussion, as well as the map $d_1^{(iii)}$.
The analogous versions of Lemma \ref{lem:VAVB iso} and Corollary \ref{cor:pre dbvhycom 2} then state that for $(g,n)\neq (1,1)$ we have that 
\[
  \gr_2 H^{6g-6+2n-k}(\AFeyn'(D\BV^*)(\!(g,n)\!))
  \cong 
  \gr_2H^{-k}(\Feyn(\HyCom)(\!(g,n)\!)).
\]
Just mind that in comparison to \eqref{equ:VAVB iso 1} we here have a degree shift of $n$ less, and a weight shift of $2n$ less, due to the legs not being marked in the amputated case.
\hfill\qed

\subsection{Proof of Corollary \ref{cor:main hbdy}}

Corollary \ref{cor:main hbdy} follows immediately by combining Theorem \ref{thm:giansiracusa} with Theorem \ref{thm:main} and Theorem \ref{thm:dbv_hycom}.

\section{Resolutions of \texorpdfstring{$\HyCom^*$}{HyCom*} in weight \texorpdfstring{$\leq 2$}{<=2}}

To show Theorem \ref{thm:hycom bv} we will need to use a resolution of the weight $\leq 2$-part of $\HyCom^*$ that is more suitable for combinatorial arguments.
We also provide a similar presentation of $\BV^*$.

\subsection{\texorpdfstring{$\HyCom^*$}{HyCom*}}
Recall that the weight grading on $\HyCom^*$ is the same as grading by the cohomological degree.
We consider the weight-truncated cyclic cooperad 
\[
    \HyCom^*_{\leq 2}
    :=
    \bigoplus_{W\leq 2} \gr_W\HyCom^*
    =
    \gr_0\HyCom^*
    \oplus 
    \gr_2\HyCom^*  
    \subset \HyCom^*
\]
consisting of elements of weight $\leq 2$.

We need a combinatorial model $\Hy^*$ for $\HyCom^*_{\leq2}$ from \cite{PayneWillwacher} that we briefly recall.
\begin{itemize}
\item The weight 0 part is just the commutative cyclic cooperad 
\[
    \gr_0\Hy^* =\gr_0 \HyCom^* =\Com^*.
\]
\item A basis of the weight 2 part $\gr_2\Hy^*(\!(r)\!)$ is given by elements 
$E_{ij}=E_{ji}$ ($1\leq i\neq j\leq r$) of degree $1$, elements $\psi_i$ ($1\leq i\leq r$) of degree 2 and elements $\delta_A=\delta_{A^c}$ of degree 2 with $A$ ranging over subsets of $\{1,\dots,r\}$ such that $|A|,|A^c|\geq 2$. 
\end{itemize}

The differential is such that 
\[
dE_{ij} = 
\psi_i+\psi_j 
-
\sum_{A\subset \{1,\dots,r\} \atop
i\in A,j\in A^c} \delta_A.    
\]

Describing the cyclic cooperad structure on $\Hy^*$ amounts to describing the $\Com^c$-coaction on the weight 2 part.
For $S$ a finite set and $B\subset S$ a subset we seek to define the coaction
\[
\Delta_B: \gr_2\Hy^*(\!(S)\!) \to \gr_2\Hy^*(\!(B\sqcup \{*\}))\otimes \Com^*(\!(B^c \sqcup \{*'\}))
\cong \gr_2\Hy^*(\!(B\sqcup \{*\})).
\]
This will be defined as follows on the basis elements:
\begin{align} \label{equ:Delta Eij} \Delta_B E_{ij} &= \begin{cases} E_{ij} & \text{if $i,j\in B$} \\
 E_{i*} & \text{if $i\in B$, $j\in B^c$} \\
 E_{j*} & \text{if $j\in B$, $i\in B^c$} \\
 0 & \text{otherwise} \end{cases} \\
\label{equ:Delta psi i} \Delta_B \psi_i &= \begin{cases} \psi_{i} & \text{if $i\in B$} \\
 0 & \text{otherwise} \end{cases} \\
\label{equ:Delta delta A} \Delta_B \delta_A &= \begin{cases} \delta_A & \text{if $A\subsetneq B$} \\
 \delta_{A^c} & \text{if $B\subsetneq A$} \\
 \psi_* & \text{if $B=A$ or $B=A^c$} \\
 0 & \text{otherwise} \end{cases}. \end{align}

\subsection{\texorpdfstring{$\BV^*$}{BV*}}\label{sec:BV gr2 basis}
Recall that on $\BV^*$ we define the weight grading as twice the grading by cohomological degree.
We consider the weight-truncated version
\[
    \BV^*_{\leq 2}
    :=
    \bigoplus_{W\leq 2} \gr_W\BV^*
    =
    \gr_0\BV^*
    \oplus 
    \gr_2\BV^*  
    \subset \BV^*.
\]
Concretely, we have that $\gr_0\BV^*=\Com^*$.
We want to describe explicitly the part $\gr_2\BV^*$.
Dualizing the discussion of Section \ref{sec:BV} we see that an explicit basis of $\gr_2\BV^*(r)$ is given by the elements $\omega_{ij}=\omega_{ji}$ for $1\leq i,j\leq r$.
In this basis the cyclic structure takes the following form.
A permutation $\sigma\in S_{r}$ acts on the basis element $\omega_{ij}$ (from the left) as 
\[
\sigma \omega_{ij} = \omega_{\sigma(i)\sigma(j)}.
\]
The action of $S_{r}$ on $\gr_2\BV^*(r)$ extends to an action of $S_{r+1}$. Let $\tau=(01)\in S_{r+1}$ be the transposition of symbols $0$ and $1$. Then we have that for $i\leq j$
\[
\tau \omega_{ij} = 
\begin{cases}  
\omega_{ij} - \omega_{1j}-\omega_{1i} + \omega_{11} & \text{for $i\geq 2$} \\
-\omega_{1j}+\omega_{11} & \text{for $i=1$, $j\geq 2$} \\
\omega_{11} & \text{for $i=j=1$}
\end{cases}.
\] 

Unfortunately, the above basis is not very canonical if one considers $\BV^*$ as a cyclic operad.
Hence we define here the alternative basis elements
of $\gr_2\BV^*(r)\cong \gr_2\BV^*(\!(r+1)\!)$:
\begin{align*} E_{ij} &:= 2\omega_{ij} -\omega_{ii} - \omega_{jj} =E_{ji}\\
 E_{0i} &:= -\omega_{ii}, \end{align*}
where we use indices $0,\dots, r$ to label the inputs in the cyclic setting.

\begin{lemma}
The above elements $E_{ij}$ with $0\leq i<j\leq r$ form a basis of 
\[
    \gr_2\BV^*(\!(r+1)\!) \cong \gr_2\BV^*(r).
\]
The $S_{r+1}\cong \Bij(\{0,\dots,r\})$-action on $\gr_2\BV^*(\!(r+1)\!)$ operates on the basis elements as
\begin{equation}\label{equ:sigma Eij} 
  \sigma \cdot E_{ij} = E_{\sigma(i)\sigma(j)},
\end{equation}
where we identify $E_{ij}=E_{ji}$, and where $\sigma\in S_{r+1}$.
The $\Com^*$-comodule structure on the basis elements obeys the formula \eqref{equ:Delta Eij} for all $B\subset S$, in particular also $B$ of cardinality $|B|=2$.
\end{lemma}
\begin{proof}
    It is clear that the elements $E_{ij}$ form a basis.

    We next check the formula \eqref{equ:sigma Eij} for the $S_{r+1}$-action.
    If $\sigma\in S_r\subset S_{r+1}$, only permuting numbers $\geq 1$, then the formula is obvious.
    Hence we just need to check that the morphism is compatible with the permutation $\tau=(01)$.
    First suppose that $i,j\geq 2$. 
    Then we compute
\begin{align*} \tau E_{ij} &= \tau (2\omega_{ij} -\omega_{ii} - \omega_{jj}) \\
&= 2\omega_{ij}-2\omega_{i1}-2\omega_{1j} + 2\omega_{11} -\omega_{ii} +2\omega_{i1} - \omega_{11} - \omega_{jj}+2\omega_{j1} - \omega_{11} \\
&=2\omega_{ij} -\omega_{ii} - \omega_{jj} =E_{ij} = E_{\tau(i)\tau(j)} \\
\tau E_{0i} &= -\tau \omega_{ii} = -\omega_{ii} + 2\omega_{1i} - \omega_{11} \\&= E_{1i} = E_{\tau(0)\tau(i)}          \end{align*}
as desired.
Similarly, for $j\geq 2$ we obtain:
\begin{align*} \tau E_{1j} &= \tau (2\omega_{1j} -\omega_{11} - \omega_{jj}) \\
 &= -2\omega_{1j}+2\omega_{11} -\omega_{11} - \omega_{jj}+2\omega_{j1} - \omega_{11} \\
 &=- \omega_{jj} = E_{0j}=E_{\tau(1)\tau(j)}\\
 \tau E_{01} &= -\tau \omega_{11} = -\omega_{11} = E_{01}=E_{\tau(0)\tau(1)} \end{align*}
Hence \eqref{equ:sigma Eij} is established.

Finally, we have to check formula \eqref{equ:Delta Eij} describing the $\Com^*$-coaction 
\[
   \Delta_B: \BV^*(\!(S)\!) \to \BV^*(\!(B\sqcup \{*\})).
\]
We take the indexing sets to be $S=\{0,\dots,r\}$ and 
$B=\{0,\dots,s\}\subset S$ with $1\leq s\leq r-2$, without loss of generality.
Then we distinguish several cases:
\begin{itemize}
    \item For $1\leq i,j \leq s$:
\begin{align*} \Delta_B E_{ij} = \Delta_B (2\omega_{ij} -\omega_{ii} - \omega_{jj}) &= 2\omega_{ij} -\omega_{ii} - \omega_{jj} =E_{ij} \end{align*}
\item For $1\leq i\in S$:
\begin{align*} \Delta_B E_{0i} = -\Delta_B (\omega_{ii}) &= -\omega_{ii}=E_{0i} \end{align*}
\item For $1\leq i\in B$, $j\in B^c$:
\begin{align*} \Delta_B E_{ij} = \Delta_B (2\omega_{ij} -\omega_{ii} - \omega_{jj}) &= 2\omega_{i*} -\omega_{ii} - \omega_{**} = E_{i*} \end{align*}
\item For $j\in B^c$:
\begin{align*} \Delta_B E_{0j} = -\Delta_B (\omega_{jj}) &= -\omega_{**}=E_{0*}. \end{align*}
\item For $i, j \in B^c$:
\begin{align*} \Delta_B E_{ij}= \Delta_B (2\omega_{ij} -\omega_{ii} - \omega_{jj}) &= 2\omega_{**} -\omega_{**} - \omega_{**} = 0. \end{align*}
\end{itemize}
This shows the lemma.
\end{proof}

\begin{remark}\label{rem:not a map}
Note that one might think that the lemma says that there is a map of cyclic cooperads
\[
\Hy^* \to \BV^*_{\leq 2} 
\]
obtained by setting to zero the generators $\psi_i$ and $\delta_A$ of $\Hy^*$. This is however not true, because the cocomposition of $E_{ij}$ in $\Hy^*$ does not produce arity two elements, while the cocomposition in $\BV^*_{\leq 2}$ does.
This defect can be repaired by extending the above morphism to an $\infty$-morphism, adding one higher homotopy. We shall not do this here since we want to avoid $\infty$-cooperadic constructions. However, the interested reader can read off the $\infty$-morphism from the map $\Phi$ between the corresponding Feynman transforms introduced in Section \ref{sec:Phi map}.  
\end{remark}

\section{\texorpdfstring{$\HyCom$}{HyCom} and \texorpdfstring{$\BV$}{BV} graph complexes and proof of Theorem \ref{thm:hycom bv}}

\subsection{Combinatorial description of graph complexes}
We begin by describing the weight 2 part of the graph complex $\Feyn(\HyCom^*)$.
Since the above variant $\Hy^*$ is quasi-isomorphic to $\HyCom^*$ in weights $\leq 2$ we have a quasi-isomorphism 
\[
\gr_2\Feyn(\Hy^*) \to \gr_2\Feyn(\HyCom^*).
\]
The graph complex $\gr_2\Feyn(\Hy^*)$ consists of linear combinations of graphs with one special vertex decorated by $\Hy^*$, while all other vertices are undecorated.
We indicate the special vertex by a double circle.
The special vertex can carry three different decorations. 
First, it can be decorated by a $\psi$-class, that we shall depict by an arrow
\begin{align*} \begin{tikzpicture}[scale=1] \node[ext,accepting, label=90:{$\scriptstyle \psi_i$}] (v) at (0,0){}; \node at (.45,.15) {$\scriptstyle i$}; \draw (v) edge +(0:.5) edge +(60:.5) edge +(-60:.5) edge +(120:.5) edge +(180:.5) edge +(-120:.5); \end{tikzpicture} &=: \begin{tikzpicture}[scale=1] \node[ext,accepting] (v) at (0,0){}; \draw (v) edge[->-] +(0:.5) edge +(60:.5) edge +(-60:.5) edge +(120:.5) edge +(180:.5) edge +(-120:.5); \end{tikzpicture} \end{align*}

  Second, there can be a decoration by $E_{ij}$, that we depict by two arrows.
  \begin{align}\label{equ_Eij pic} \begin{tikzpicture} \node[ext, accepting] (v) at (0,0){}; \node (vi) at (130:.7){}; \node (vj) at (50:.7){}; \node (vk) at (90:.5){$E_{ij}$}; \draw (v) edge (vi) edge (vj) (v) edge +(-30:.5) edge +(-90:.5) edge +(-150:.5); \end{tikzpicture} \ \  & =: \ \  \begin{tikzpicture} \node[ext, accepting] (v) at (0,0){}; \node (vi) at (130:.7){$\scriptstyle i$}; \node (vj) at (50:.7){$\scriptstyle j$}; \draw (v) edge[->-] (vi) edge[->-] (vj) (v) edge +(-30:.5) edge +(-90:.5) edge +(-150:.5) ; \end{tikzpicture} \end{align}

Note that the special vertex must have valence at least three.

  Finally, the special vertex may be decorated by $\delta_{A}$.
  Pictorially, we replace the special vertex with two vertices connected by a marked edge.
  \begin{align}\label{equ:delta two vert} \begin{tikzpicture}[scale=1] \node[ext,accepting, label=90:{$\scriptstyle \delta_{A}$}] (v) at (0,0){}; \draw (v) edge +(0:.5) edge +(60:.5) edge +(-60:.5) edge +(120:.5) edge +(180:.5) edge +(-120:.5); \end{tikzpicture} &=: \begin{tikzpicture}[scale=1] \node[ext] (v) at (0,0){}; \node[ext] (w) at (-.7,0){}; \draw (v) edge[crossed] (w) edge +(0:.5) edge +(60:.5) edge +(-60:.5) (w) edge +(120:.5) edge +(180:.5) edge +(-120:.5); \draw[pbrace] (.6,.5) -- (.6,-.5); \node at (.9,0) {$A$}; \end{tikzpicture} \end{align}
Here the subset $A$ of half-edges is connected to one vertex, and the complement to the other.

The differential consists of two parts,
\[
d = d_s + d_{\Hy^*},  
\]
where $d_s$ acts by splitting vertices, using the cooperadic cocomposition, and $d_{\Hy^*}$ is induced from the differential on $\Hy^*$.
Concretely, since the differential on $\Hy^*$ is only nontrivial on the $E_{ij}$-generator, we have the pictorial description
\begin{equation}\label{equ:dhyc pic}
  d_{\Hy^*}:
\begin{tikzpicture} \node[ext, accepting] (v) at (0,0){}; \node (vi) at (130:.7){}; \node (vj) at (50:.7){}; \draw (v) edge[->-] (vi) edge[->-] (vj) (v) edge +(-30:.5) edge +(-70:.5) edge +(-110:.5) edge +(-150:.5) ; \end{tikzpicture}
\mapsto 
\begin{tikzpicture} \node[ext, accepting] (v) at (0,0){}; \node (vi) at (130:.7){}; \node (vj) at (50:.7){}; \draw (v) edge (vi) edge[->-] (vj) (v) edge +(-30:.5) edge +(-70:.5) edge +(-110:.5) edge +(-150:.5) ; \end{tikzpicture}
+
\begin{tikzpicture} \node[ext, accepting] (v) at (0,0){}; \node (vi) at (130:.7){}; \node (vj) at (50:.7){}; \draw (v) edge[->-] (vi) edge (vj) (v) edge +(-30:.5) edge +(-70:.5) edge +(-110:.5) edge +(-150:.5) ; \end{tikzpicture}
-
\sum\, 
\begin{tikzpicture} \node[ext] (v) at (0,0){}; \node[ext] (w) at (0.7,0){}; \draw (v) edge[crossed] (w) edge +(130:.7) (w) edge +(50:.7) (v) edge +(-110:.5) edge +(-150:.5) (w) edge +(-30:.5) edge +(-70:.5) ; \end{tikzpicture}
.
\end{equation}
Furthermore, from the description \eqref{equ:Delta Eij}-\eqref{equ:Delta delta A} of the cooperadic cocomposition we can see that the operation $d_s$ acts as follows:
\begin{align}\label{equ:ds1} \begin{tikzpicture}[baseline=-.65ex] \node[int] (v) at (0,0) {}; \draw (v) edge +(-.3,-.3) edge +(-.3,0) edge +(-.3,.3) edge +(.3,-.3) edge +(.3,0) edge +(.3,.3); \end{tikzpicture} &\mapsto \pm \sum \begin{tikzpicture}[baseline=-.65ex] \node[int] (v) at (0,0) {}; \node[int] (w) at (0.5,0) {}; \draw (v) edge (w) (v) edge +(-.3,-.3) edge +(-.3,0) edge +(-.3,.3) (w) edge +(.3,-.3) edge +(.3,0) edge +(.3,.3); \end{tikzpicture} \\ \label{equ:dsdelta} \begin{tikzpicture}[scale=1] \node[ext] (v) at (0,0){}; \node[ext] (w) at (-.7,0){}; \draw (v) edge[crossed] (w) edge +(0:.5) edge +(60:.5) edge +(-60:.5) (w) edge +(120:.5) edge +(180:.5) edge +(-120:.5); \end{tikzpicture} &\mapsto \sum \begin{tikzpicture}[scale=1] \node[int] (v0) at (0.7,0){}; \node[ext] (v) at (0,0){}; \node[ext] (w) at (-.7,0){}; \draw (v) edge[crossed] (w) edge (v0) edge +(-60:.5) (v0) edge +(0:.5) edge +(60:.5) (w) edge +(120:.5) edge +(180:.5) edge +(-120:.5); \end{tikzpicture} +\sum \begin{tikzpicture}[scale=1] \node[int] (w0) at (-1.4,0){}; \node[ext] (v) at (0,0){}; \node[ext] (w) at (-.7,0){}; \draw (v) edge[crossed] (w) edge +(-60:.5) edge +(0:.5) edge +(60:.5) (w) edge (w0) edge +(120:.5) (w0) edge +(180:.5) edge +(-120:.5); \end{tikzpicture} - \begin{tikzpicture}[scale=1] \node[int] (v) at (0,0){}; \node[ext, accepting] (w) at (-.7,0){}; \draw (v) edge +(0:.5) edge +(60:.5) edge +(-60:.5) (w) edge[->-] (v) edge +(120:.5) edge +(180:.5) edge +(-120:.5); \end{tikzpicture} - \begin{tikzpicture}[scale=1] \node[ext, accepting] (v) at (0,0){}; \node[int] (w) at (-.7,0){}; \draw (v) edge[->-] (w) edge +(0:.5) edge +(60:.5) edge +(-60:.5) (w) edge +(120:.5) edge +(180:.5) edge +(-120:.5); \end{tikzpicture} \\
 \label{equ:dsEij} \begin{tikzpicture} \node[ext, accepting] (v) at (0,0){}; \node (vi) at (130:.7){}; \node (vj) at (50:.7){}; \draw (v) edge[->-] (vi) edge[->-] (vj) (v) edge +(-30:.5) edge +(-70:.5) edge +(-110:.5) edge +(-150:.5) ; \end{tikzpicture} &\mapsto -\sum\, \begin{tikzpicture} \node[ext, accepting] (v) at (0,0){}; \node[int] (w) at (130:.7) {}; \node (vi) at (130:1.4){}; \node (vj) at (50:.7){}; \draw (v) edge[->-] (w) edge[->-] (vj) (w) edge (vi) edge +(-110:.5) edge +(-150:.5) (v) edge +(-30:.5) edge +(-70:.5) ; \end{tikzpicture} - \sum\, \begin{tikzpicture} \node[ext, accepting] (v) at (0,0){}; \node[int] (w) at (50:.7) {}; \node (vi) at (130:.7){}; \node (vj) at (50:1.4){}; \draw (v) edge[->-] (w) edge[->-] (vi) (w) edge (vj) edge +(-30:.5) edge +(-70:.5) (v) edge +(-110:.5) edge +(-150:.5) ; \end{tikzpicture} - \sum\, \begin{tikzpicture} \node[ext, accepting] (v) at (0,0){}; \node[int] (w) at (0,-.7) {}; \node (vi) at (130:.7){}; \node (vj) at (50:.7){}; \draw (v) edge (w) edge[->-] (vi) edge[->-] (vj) (w) edge +(-110:.5) edge +(-70:.5) (v) edge +(-150:.5) edge +(-30:.5) ; \end{tikzpicture} \end{align}

Finally, we turn to 
\[
    \gr_2\Feyn(\BV^*).
\]
The combinatorial description of this graph complex is similar to that of $\gr_2\Feyn(\Hy^*)$
above. The collection of dg vector spaces $\gr_2\Feyn(\BV^*)$ consists of linear combinations of graphs with one special vertex.
This special is decorated by $E_{ij}$, using the basis $E_{ij}$ of $\gr_2\BV^*$ introduced in Section \ref{sec:BV gr2 basis}.

We use the graphical encoding \eqref{equ_Eij pic} to indicate such a decoration on the special vertex.
One important fact to note is, however, that in $\gr_2\Feyn(\BV^*)$ the special vertex may have valence 2, while in $\gr_2\Feyn(\Hy^*)$ it needs to have valence $\geq 3$.
In particular, in the analog version of \eqref{equ:dsEij} for $\gr_2\Feyn(\BV^*)$ there are also terms for which the special vertex is bivalent.

\subsection{A morphism}\label{sec:Phi map}

We next describe a morphism of dg modular sequences
\[
\Phi: 
\gr_2\Feyn(\Hy^*) 
\to
\gr_2\Feyn(\BV^*).
\]
This morphism is defined as follows:
\begin{itemize}
\item If $\Gamma\in \gr_2\Feyn(\Hy^*)(\!(g,n)\!)$ is a graph whose special vertex is decorated by $\delta_A$, then we set $\Phi(\Gamma)=0$.
\item If $\Gamma\in \gr_2\Feyn(\Hy^*)(\!(g,n)\!)$ is a graph whose special vertex is decorated by $E_{ij}$, then we set $\Phi(\Gamma)=\Gamma$, where we identify the basis element $E_{ij}$ of $\gr_2\Hy^*$ with the corresponding basis element $E_{ij}$ of $\gr_2\BV^*$, see \ref{sec:BV gr2 basis}.
\item Finally, let $\Gamma\in \gr_2\Feyn(\Hy^*)(\!(g,n)\!)$ be a graph whose special vertex is decorated by $\psi_{i}$, with $i$ corresponding to some half-edge at the special vertex.
Then we set $\Phi(\Gamma)=\Gamma'$, with $\Gamma'$ obtained from $\Gamma$ by the replacement
\begin{equation}\label{equ: phi on psi}
      \begin{tikzpicture} \node[ext, accepting] (v) at (0,0){}; \node (w1) at (1,0) {}; \draw (v) edge[->-] (w1) (v) edge +(-.5,.5) edge +(-.5,0) edge +(-.5,-.5) ; \end{tikzpicture}
      \, 
      \mapsto
      \, 
      \begin{tikzpicture} \node[ext, accepting] (v) at (0,0){}; \node[int] (w1) at (-.7,0) {}; \node (w2) at (.7,0) {}; \draw (v) edge[->-] (w1) edge[->-] (w2) (w1) edge +(-.5,.5) edge +(-.5,0) edge +(-.5,-.5) ; \end{tikzpicture}
    \end{equation}
That is, we replace the special vertex in $\Gamma$ by an ordinary vertex and add a new bivalent special vertex to the graph on the edge of the half-edge $i$. As usual, we put the new edge first in the ordering.

\end{itemize}

\begin{lemma}
The map $\Phi:\gr_2\Feyn(\Hy^*) \to \gr_2\Feyn(\BV^*)$ 
defined above is a morphism of dg modular sequences.
\end{lemma}
\begin{proof}
  It is clear that $\Phi$ intertwines the symmetric group actions. 
  We shall check that it also intertwines the differentials.
  Since the differential on $\BV^*$ is zero we have to check that for any graph $\Gamma\in \gr_2\Feyn(\Hy^*)(\!(g,n)\!)$
  \[
  \Phi(d_s\Gamma + d_{\Hy^*}\Gamma ) =  d_s\Phi(\Gamma).
  \]
The splitting differential $d_s$ can act either by splitting the special vertex, or a non-special vertex.
Since $\Phi(\Gamma)$ is the same graph as $\Gamma$ away from the special vertex, it is sufficient to consider the pieces of the differential that acts on the special vertex.
  We naturally have to consider three different cases, according to the decoration of the special vertex of $\Gamma$.
  \begin{itemize}
    \item Suppose the special vertex of $\Gamma$ is decorated by $\delta_A$. Then $\Phi(\Gamma)=0$ and $d_{\Hy^*}\Gamma=0$, so we have to check that $\Phi(d_s\Gamma)=0$.
    Among the terms $d_s\Gamma$ there are those with with their special vertex decorated by some $\delta_{A'}$, and those for which the special vertex is decorated by some $\psi_i$, see \eqref{equ:dsdelta}.
    The former terms are trivially sent to zero, so we only need to consider the latter. Pictorially, 
    \[
      d_s\Gamma = 
    \begin{tikzpicture} \node[ext, accepting] (v) at (0,0){}; \node[int] (w1) at (1,0) {}; \draw (v) edge[->-] (w1) (v) edge +(-.5,.5) edge +(-.5,0) edge +(-.5,-.5) (w1) edge +(.5,.5) edge +(.5,0) edge +(.5,-.5) ; \end{tikzpicture}
        +
        \begin{tikzpicture}[xscale=-1] \node[ext, accepting] (v) at (0,0){}; \node[int] (w1) at (1,0) {}; \draw (v) edge[->-] (w1) (v) edge +(-.5,.5) edge +(-.5,0) edge +(-.5,-.5) (w1) edge +(.5,.5) edge +(.5,0) edge +(.5,-.5) ; \end{tikzpicture}
        +(\cdots),
    \] 
    where we only draw a neighborhood of the special vertex, not the whole graph.
    Applying $\Phi$ we produce the terms 
    \[
      \begin{tikzpicture} \node[ext, accepting] (v) at (0,0){}; \node[int] (w1) at (-.7,0) {}; \node[int] (w2) at (.7,0) {}; \draw (v) edge[->-] node[above]{$\scriptstyle 1$} (w1) edge[->-] (w2) (w1) edge +(-.5,.5) edge +(-.5,0) edge +(-.5,-.5) (w2) edge +(.5,.5) edge +(.5,0) edge +(.5,-.5) ; \end{tikzpicture}
          +
          \begin{tikzpicture} \node[ext, accepting] (v) at (0,0){}; \node[int] (w1) at (-.7,0) {}; \node[int] (w2) at (.7,0) {}; \draw (v) edge[->-] (w1) edge[->-] node[above]{$\scriptstyle 1$} (w2) (w1) edge +(-.5,.5) edge +(-.5,0) edge +(-.5,-.5) (w2) edge +(.5,.5) edge +(.5,0) edge +(.5,-.5) ; \end{tikzpicture}=0,
    \]
    with the ``1'' over the edge indicating that this edge comes first in the order of edges.
    Both terms cancel due to sign, given that the order of the two edges adjacent to the bivalent vertex is opposite.
    \item Suppose that the special vertex of $\Gamma$ is decorated by $E_{ij}$. Then $\Phi(\Gamma)=\Gamma$ and the difference
    \[
      d_s\Phi(\Gamma)-\Phi(d_s\Gamma)
    \]
    consists only of those terms for which the special vertex is bivalent, namely:
    \[
      \begin{tikzpicture} \node[int] (v) at (0,0){}; \node[ext, accepting] (w) at (-.5,0.5){}; \draw (v) edge +(-.5,-.5) edge +(0,-.5) edge +(.5,-.5) edge +(.5,.5) (w) edge[->-] (v) edge[->-] +(-.5,.5); \end{tikzpicture}
      +
      \begin{tikzpicture} \node[int] (v) at (0,0){}; \node[ext, accepting] (w) at (.5,0.5){}; \draw (v) edge +(-.5,-.5) edge +(0,-.5) edge +(.5,-.5) edge +(-.5,.5) (w) edge[->-] (v) edge[->-] +(.5,.5); \end{tikzpicture}
    \]
    (See also Remark \ref{rem:not a map}.)
    But these terms are equal to $\Phi(d_{\Hy^*}\Gamma)$, see \eqref{equ:dhyc pic} for a picture of $d_{\Hy^*}\Gamma$ and then \eqref{equ: phi on psi} for the action of $\Phi$.
    \item Finally, suppose the special vertex is decorated by $\psi_i$.
    Then $d_{\Hy^*}\Gamma=0$ and $\Phi(\Gamma)$ is obtained from $\Gamma$ by making the special vertex non-special, but adding a bivalent vertex.
    Comparing the formulas for splitting the $\psi_i$-decorated vertex, the terms in $\Phi(d_s\Gamma)$ match those in $d_s\Phi(\Gamma)$ from splitting the new ordinary vertex. \qedhere
  \end{itemize}
\end{proof}

\subsection{Proof of Theorem \ref{thm:hycom bv}}
Given the map $\Phi$ of the previous subsection, Theorem \ref{thm:hycom bv} follows from the following result.
\begin{prop}
    The map $\Phi: \gr_2\Feyn(\Hy^*)(\!(g,n)\!)\to \gr_2\Feyn(\BV^*)(\!(g,n)\!)$ is a quasi-isomorphism for all $(g,n)\neq (1,0),(0,2)$.
\end{prop}
\begin{proof}

    Let $(g,n)\neq (1,0),(0,2)$.
    It is easy to see that $\Phi$ is a surjective map.
    Hence we have to check that the kernel of $\Phi$ is acyclic, $H(\ker \Phi)=0$.
    We first describe the kernel explicitly.
    A basis of the kernel is given by (i) graphs $\Gamma$ whose special vertex is decorated by $\delta_A$ and (ii) linear combinations of two graphs, with $\psi$-decorations along one edge in a symmetric combination:
    
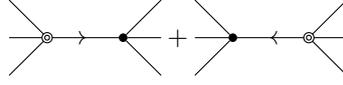
\[
      \begin{tikzpicture} \node[ext, accepting] (v) at (0,0){}; \node[int] (w1) at (1,0) {}; \draw (v) edge[->-] (w1) (v) edge +(-.5,.5) edge +(-.5,0) edge +(-.5,-.5) (w1) edge +(.5,.5) edge +(.5,0) edge +(.5,-.5) ; \end{tikzpicture}
      +
      \begin{tikzpicture}[xscale=-1] \node[ext, accepting] (v) at (0,0){}; \node[int] (w1) at (1,0) {}; \draw (v) edge[->-] (w1) (v) edge +(-.5,.5) edge +(-.5,0) edge +(-.5,-.5) (w1) edge +(.5,.5) edge +(.5,0) edge +(.5,-.5) ; \end{tikzpicture}  
    \]
    The basis elements (ii) may be considered as given by graphs with a marked edge instead of a vertex, the edge corresponding to the symmetric combination of $\psi$-decorations as above.

    Next, we want to check that $H(\ker \Phi)=0$.
    To this end, let $\Gamma\in \Feyn(\Hy^*)(\!(g,n)\!)$ be a graph with $k$ vertices. We say that the number of effective vertices of $\Gamma$ is $k+1$ if the special vertex is decorated by $\delta_A$ and $k$ otherwise.
    We filter $\ker \Phi$ by the number of effective vertices and consider the associated spectral sequence.
    On the first page we see only the last two terms of the differential \eqref{equ:dsdelta} replacing the $\delta_A$-decorated special vertex by a pair of vertices with a symmetric $\psi$-decoration.
    
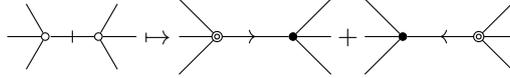
\[    
      \begin{tikzpicture}[scale=1] \node[ext] (v) at (0,0){}; \node[ext] (w) at (-.7,0){}; \draw (v) edge[crossed] (w) edge +(0:.5) edge +(60:.5) edge +(-60:.5) (w) edge +(120:.5) edge +(180:.5) edge +(-120:.5); \end{tikzpicture} 
      \mapsto 
      \begin{tikzpicture} \node[ext, accepting] (v) at (0,0){}; \node[int] (w1) at (1,0) {}; \draw (v) edge[->-] (w1) (v) edge +(-.5,.5) edge +(-.5,0) edge +(-.5,-.5) (w1) edge +(.5,.5) edge +(.5,0) edge +(.5,-.5) ; \end{tikzpicture}
      +
      \begin{tikzpicture}[xscale=-1] \node[ext, accepting] (v) at (0,0){}; \node[int] (w1) at (1,0) {}; \draw (v) edge[->-] (w1) (v) edge +(-.5,.5) edge +(-.5,0) edge +(-.5,-.5) (w1) edge +(.5,.5) edge +(.5,0) edge +(.5,-.5) ; \end{tikzpicture}
    \]
    This map clearly is a bijection between the basis elements (i) and (ii) in the kernel. Hence our spectral sequence abuts to 0 and we have shown that $H(\ker \Phi)=0$ as desired.
\end{proof}

\subsection{Blown-up picture and proof of Proposition \ref{prop:FeynBV_symmetric_product}}
\label{sec:blownup}

There is a different combinatorial way in which we may depict generators of $\gr_2\Feyn(\BV^*)(\!(g,n)\!)$, which we call the ``blown-up picture'', following \cite{PayneWillwacher11}.
Concretely, we may remove the special vertex from the graph, and make the incident half-edges into external legs. The decoration $E_{ij}$ is remembered by marking those legs corresponding to $i$ and $j$ with a symbol $\omega$, and the other legs by $\epsilon$.
For example:

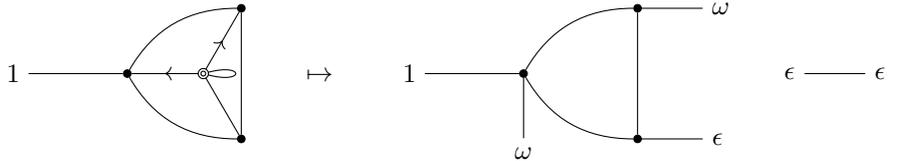
\[
    \begin{tikzpicture}[scale=1] \node[ext,accepting] (v1) at (0,0){}; \node[int] (v2) at (180:1){}; \node[int] (v3) at (60:1){}; \node[int] (v4) at (-60:1){}; \draw (v1) edge[->-](v2) edge[->-] (v3) edge (v4) edge[loop right] (v2) (v2) edge[bend left] (v3) edge[bend right] (v4) -- +(180:1.3) (v3) edge (v4); \node (w) at (180:2.5) {$1$}; \end{tikzpicture}
\quad \quad \mapsto \quad \quad 
\begin{tikzpicture}[scale=1, node distance=.8] \node[int] (v2) at (180:1){}; \node[int] (v3) at (60:1){}; \node[int] (v4) at (-60:1){}; \node[below=of v2] (e1) {$\omega$}; \node[right=of v3] (e2) {$\omega$}; \node[right=of v4] (e3) {$\epsilon$}; \node (e4) at (2.5,0) {$\epsilon$}; \node[right=of e4] (e5) {$\epsilon$}; \draw (v2) edge (e1) edge[bend left] (v3) edge[bend right] (v4) -- +(180:1.3) (v3) edge (e2) edge (v4) (v4) edge (e3) (e4) edge (e5); \node (w) at (180:2.5) {$1$}; \end{tikzpicture}\, .
\]

The same graph complexes have appeared in the algebraic topology literature, and compute direct summands of the cohomology of embedding spaces of copies of $\R^m$ and one sphere into $\R^N$, see \cite{FTW2}.

In the special case $n=0$ it is known that the above complexes can be simplified further, see \cite{TWspherical}.
To this end, let 
\[
Y_g \subset \gr_2\Feyn(\BV^*)(\!(g,0)\!)    
\]
be the subcomplex spanned by graphs that have
(i) no $\epsilon$-decorated legs in the blown-up picture and (ii) no blown-up connected component that is a tadpole graph
\[
    D=\begin{tikzpicture}[baseline=-.65ex,every loop/.style={}] \node (v) at (0,-.3) {$\omega$}; \node [int] (w) at (0,.3) {}; \draw (v) edge (w) (w) edge[loop] (w); \end{tikzpicture}.
\]
In the standard picture (i) means that the special vertex has valence 2.
Then we have:

\begin{prop}\label{prop:pre feyn bv symm}
For $g\geq 2$ the inclusion $Y_g \to \gr_2\Feyn(\BV^*)(\!(g,0)\!)$ is a quasi-isomorphism.
\end{prop}
\begin{proof}
We work in the blown-up picture, and compute the cohomology of $\gr_2\Feyn(\BV^*)(\!(g,0)\!)$.
We filter $\gr_2\Feyn(\BV^*)(\!(g,0)\!)$ by the number of (blown-up) connected components and consider the associated spectral sequence.
The differential on the first page leaves the number of blown-up components the same, so that the resulting complex is a sum of products of a connected graph complex.
The cohomology of the connected complex has been computed in \cite[Theorem 3.1, subcase ``$m,n$ even'']{TWspherical}, and shown to be equal to that of the subcomplex without $\epsilon$-legs, minus the tadpole graph $D$.
Hence, the cohomology of the $E^1$-page is the same as the cohomology of $Y_g$. 
This means that the $E^2$-pages of the spectral sequences associated to $\gr_2\Feyn(\BV^*)(\!(g,0)\!)$ and of its subcomplex $Y_g$ agree. All higher differentials in the spectral sequence vanish since there is no piece of the differential that can reduce the number of blown-up components of a graph without $\epsilon$-legs, nor is there a piece of the differential creating an $\epsilon$ leg.
Hence our spectral sequence abuts at the $E^2$-page.
Furthermore, since $\gr_2\Feyn(\BV^*)(\!(g,0)\!)$ is finite dimensional, the spectral sequences converge to the cohomology.
Hence we conclude that the inclusion  $Y_g \to \gr_2\Feyn(\BV^*)(\!(g,0)\!)$ is a quasi-isomorphism.
\end{proof}

\begin{proof}[Proof of Proposition \ref{prop:FeynBV_symmetric_product}]

By Proposition \ref{prop:pre feyn bv symm} we may consider the quasi-isomorphic subcomplex $Y_g$.
Note that every graph generating $Y_g$
has either one blown-up component with two legs, or two blown-up components with one leg each.
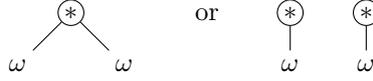
\begin{align*} \begin{tikzpicture} \node[ext] (v) at (0,0) {$*$}; \node (o1) at (-.7,-.7) {$\omega$}; \node (o2) at (.7,-.7) {$\omega$}; \draw (v) edge (o1) edge (o2); \end{tikzpicture} \quad\quad\text{or} \quad \quad \begin{tikzpicture} \node[ext] (v) at (0,0) {$*$}; \node[ext] (v2) at (1,0) {$*$}; \node (o1) at (0,-.7) {$\omega$}; \node (o2) at (1,-.7) {$\omega$}; \draw (v) edge (o1) (v2) edge (o2); \end{tikzpicture} \end{align*}
Accordingly, the complex $Y_g$ splits into a direct sum of subcomplexes
\begin{align*} Y_g &= Y_g^1 \oplus Y_g^2   . \end{align*}
Furthermore, $Y_g^2$ is isomorphic to the genus $g$ piece of the symmetric product 
\[
    \Sym^2(\oplus_{h\geq 2}\Feyn(\Com^*)(\!(h,1)\!)^{no-tp}[-1])[-1],
\]
where we restrict the sum to $h\geq 2$ to account for the component $D$ not being present in generators of $Y_g^2$.

Similarly, $Y_g^2$ is isomorphic to the part of $\Feyn(\Com^*)(\!(g,2)\!)[-1]$ that is antisymmetric under the exchange of the two hairs. Hence Proposition \ref{prop:FeynBV_symmetric_product} follows.
\end{proof}

\subsection{Proofs of Corollaries \ref{cor:Hbdy Com comparison} and \ref{cor:hbdy growth}}
Corollary \ref{cor:Hbdy Com comparison} follows by restricting the third statement of Corollary \ref{cor:main hbdy} to $n=0$ and applying Theorem \ref{thm:hycom bv} and Proposition \ref{prop:FeynBV_symmetric_product}.

Next consider Corollary \ref{cor:hbdy growth}.
This is deduced from statement (2) of Corollary \ref{cor:main hbdy} and from Corollary \ref{cor:Hbdy Com comparison}, using nontrivial classes in the commutative graph cohomology
$H(\Feyn(\Com^*)(\!(g,1)\!))$. 
The argument is identical to that leading to \cite[Corollary 1.3]{PayneWillwacher}, and we recall here the main steps.
Firstly, as in loc.\ cit.\ we know that the dimensions of
\begin{align*} H^{2g}(\Feyn(\Com^*)(\!(g,0)\!)) \text{ and } H^{2g+3}(\Feyn(\Com^*)(\!(g,1)\!)) \end{align*}
are at least exponentially growing as $g\to \infty$.
Hence so is the dimension of $H^{4g-k}(\HMod_{g,0}^0)$ for $k=6,9$ by statement (2) of Corollary \ref{cor:main hbdy}.

Secondly, we also know that the dimensions of
\begin{equation}
\label{eq_feyncom_g_1}
 H^{2g}(\Feyn(\Com^*)(\!(g,1)\!)) \text{ and }  
 H^{2g+3}(\Feyn(\Com^*)(\!(g,1)\!))
\end{equation}
are at least exponentially growing as $g\to \infty$ and that there is a nontrivial cohomology class in 
\begin{equation}
\label{eq_feyncom_10_1}
	H^{27}(\Feyn(\Com^*)(\!(10,1)\!)).
\end{equation}
As in Proposition \ref{prop:FeynBV_symmetric_product}, let $\SW_{g}^k$ be the part of total genus $g$ and degree $k$ of the symmetric product 
\[
\SW = \Sym^2\left( \bigoplus_{g\geq 2} H^\bullet\left(\Feyn(\Com^*)(\!(g,1)\!)\right)[-1] \right).
\]
Then by Corollary \ref{cor:Hbdy Com comparison}, $(\SW_g^{k-1})^*$ is a summand in $ \gr_{6g-8} H^{6g-6-k}(\HMod_{g,0}^0)$.
Now combining the non-trivial classes from \eqref{eq_feyncom_g_1} and \eqref{eq_feyncom_10_1} as factors in $\SW$ yields at least exponential growth of $H^{4g-k}(\HMod_{g,0}^0)$ for $k=9,12,15,16,19$.\footnote{Let $\feyncom = \Feyn(\Com^*)$.
These numbers are obtained by considering the following summands of $\SW_g^{k-1}$:
\begin{gather*} H^{2g'}\left(\feyncom(\!(g',1)\!)\right)[-1] \oplus H^{2g'}\left(\feyncom(\!(g',1)\!)\right)[-1] \subset \SW_{2g'}^{4g'+2}, \quad H^{2g'}\left(\feyncom(\!(g',1)\!)\right)[-1] \oplus H^{2g'+3}\left(\feyncom(\!(g',1)\!)\right)[-1] \subset \SW_{2g'}^{4g'+5}, \\
H^{2g'+3}\left(\feyncom(\!(g',1)\!)\right)[-1] \oplus H^{2g'+3}\left(\feyncom(\!(g',1)\!)\right)[-1] \subset \SW_{2g'}^{4g'+8}, \quad H^{2g'}\left(\feyncom(\!(g',1)\!)\right)[-1] \oplus H^{27}\left(\feyncom(\!(10,1)\!)\right)[-1] \subset \SW_{g'+10}^{2g'+29}, \\
H^{2g'+3}\left(\feyncom(\!(g',1)\!)\right)[-1] \oplus H^{27}\left(\feyncom(\!(10,1)\!)\right)[-1] \subset \SW_{g'+10}^{2g'+32}. \end{gather*}}

\section{Comparison of \texorpdfstring{$H^\bullet(\MM)$}{H*(M)} and \texorpdfstring{$\HyCom$}{HyCom} graph complexes and proof of Theorem \ref{prop:HyCom MM comparison}}

\subsection{Graph complex computing \texorpdfstring{$\gr_2 H_c^\bullet(\M_{g,n})$}{gr2(Hc*(Mgn))} -- recollection from \cite{PayneWillwacher}}\label{sec:Xgn def}

Sam Payne and the last-named author \cite{PayneWillwacher} have identified a graph complex $X_{g,n}$ that computes the weight 2 part of the compactly supported cohomology of the moduli spaces of curves $\gr_2 H_c^\bullet(\M_{g,n})$.
The complex has a very efficient description in terms of the Feynman transform of $\BV^*$ as follows.

Recall that a tadpole in a graph is an edge connecting a vertex to itself.
Graphs contributing to the Feynman transform may have tadpoles.
Specifically, a graph $\Gamma\in \gr_2\Feyn(\BV^*)(\!(g,n)\!)$ may have tadpoles of 4 different types:
\begin{itemize}
\item Tadpoles at internal vertices.
\[
  \text{Type 0:}
\begin{tikzpicture} \node[int] (v) at (0,0) {}; \draw (v) edge[loop] (v) edge +(-.5,-.5) edge +(0,-.5) edge +(.5,-.5); \end{tikzpicture}
\]
\item Tadpoles at the special vertex, which can have either of three different types:
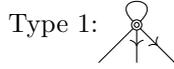
\begin{align*} \text{Type 1:}& \begin{tikzpicture} \node[ext, accepting] (v) at (0,0) {}; \draw (v) edge[loop] (v) edge +(-.5,-.5) edge[->-] +(0,-.5) edge[->-] +(.5,-.5); \end{tikzpicture} & \text{Type 2:}& 
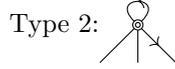
\begin{tikzpicture} \node[ext, accepting] (v) at (0,0) {}; \draw (v) edge[loop, ->-] (v) edge +(-.5,-.5) edge +(0,-.5) edge[->-] +(.5,-.5); \end{tikzpicture} & \text{Type 3:}& 
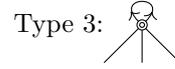
\begin{tikzpicture} \node[ext, accepting] (v) at (0,0) {}; \draw (v) edge[loop, -><-] (v) edge +(-.5,-.5) edge +(0,-.5) edge +(.5,-.5); \end{tikzpicture} \end{align*}
In other words, if the decoration at the special vertex is $E_{ij}$ then tadpoles of type 1 (resp. type 2, type 3) are such that none (resp. one, two) of the half-edges $i$ and $j$ are part of the tadpole edge.
\end{itemize}

Let $I_{g,n}\subset \gr_2\Feyn(\BV^*)(\!(g,n)\!)$ be the subspace spanned by graphs that have at least one tadpole of type 0, 2 or 3.

\begin{lemma}
The subspace $I_{g,n}\subset \gr_2\Feyn(\BV^*)(\!(g,n)\!)$ is closed under the differential, that is, 
\[
d I_{g,n} \subset I_{g,n}.  
\]
\end{lemma}
\begin{proof}
One checks that the differential cannot remove any of the tadpoles of type 0,2,3. At worst, it can transform a tadpole of type 2 into one of type 0.
\end{proof}

We then define, for $g,n$ such that $2g+n\geq 3$, the quotient complex
\[
X_{g,n} :=   \gr_2\Feyn(\BV^*)(\!(g,n)\!) / I_{g,n}.
\]

We recall the main result from \cite{PayneWillwacher}.
\begin{thm}
For all $(g,n)\neq (1,1)$ such that $2g+n\geq 3$ we have that 
\[
  \gr_2 H_c^\bullet(\M_{g,n}) \cong X_{g,n}.
\]
\end{thm}

\subsection{Proof of Proposition \ref{prop:HyCom MM comparison}}\label{sec:HyCom MM proof}
To show Proposition \ref{prop:HyCom MM comparison} we first introduce an auxiliary graph complex $Z_{g,n}$ using the tadpole terminology of the previous section.
Concretely, let 
\[
  Z_{g,n} \subset \gr_2\Feyn(\BV^*)(\!(g,n)\!)
\]
be the graded subspace spanned by graphs that do not have any tadpole, of any of the types 0-3 above.
The differential (vertex splitting) cannot create a tadpole if there was none before, so that $Z_{g,n}$ is a dg subspace.

\begin{lemma}\label{lem:Zgn qiso}
The inclusion $Z_{g,n} \to \gr_2\Feyn(\BV^*)(\!(g,n)\!)$ is a quasi-isomorphism for all $(g,n)$ except $(g,n)=(1,0), (1,1)$.
\end{lemma}
\begin{proof}
  The proof is similar to that of \cite[Lemma 5]{AWZ}.
We have to check that the quotient $$Q_{g,n}:=\gr_2\Feyn(\BV^*)(\!(g,n)\!)/Z_{g,n} $$
is acyclic.
Note that a basis of $Q_{g,n}$ is given by graphs with at least one tadpole, of any type.
For such a graph $\Gamma$ let us say that a vertex is a thin tadpole vertex if it is non-special, carries a tadpole, and has valence $3$.
\[
\begin{tikzpicture} \node[int] (v) at (0,0) {}; \draw (v) edge[loop] (v) edge +(0,-.5); \end{tikzpicture}
\]
Let $k(\Gamma)$ be the number of other vertices in $\Gamma$.
We filter $Q_{g,n}$ by the numbers $k(\Gamma)$ and consider the associated spectral sequence.
The first page is identified with $(Q_{g,n},d_s')$, where $d_s'$ are those terms of the vertex splitting differential that create a thin tadpole vertex.
Concretely, $d_s'$ acts non-trivially only on vertices carrying a tadpole:
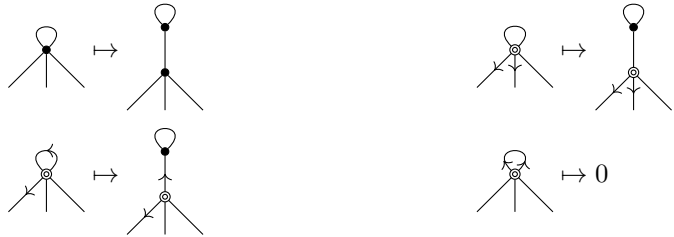
\begin{align*} \begin{tikzpicture} \node[int] (v) at (0,0) {}; \draw (v) edge[loop] (v) edge +(0,-.5) edge +(-.5,-.5) edge +(.5,-.5); \end{tikzpicture} &\mapsto \begin{tikzpicture} \node[int] (v) at (0,.3) {}; \node[int] (w) at (0,-.3) {}; \draw (v) edge[loop] (v) edge (w) (w) edge +(0,-.5) edge +(-.5,-.5) edge +(.5,-.5); \end{tikzpicture} & \begin{tikzpicture} \node[ext, accepting] (v) at (0,0) {}; \draw (v) edge[loop] (v) edge[->-] +(0,-.5) edge[->-] +(-.5,-.5) edge +(.5,-.5); \end{tikzpicture} &\mapsto \begin{tikzpicture} \node[int] (v) at (0,.3) {}; \node[ext, accepting] (w) at (0,-.3) {}; \draw (v) edge[loop] (v) edge (w) (w) edge[->-] +(0,-.5) edge[->-] +(-.5,-.5) edge +(.5,-.5); \end{tikzpicture} \\
 \begin{tikzpicture} \node[ext, accepting] (v) at (0,0) {}; \draw (v) edge[loop,->-] (v) edge +(0,-.5) edge[->-] +(-.5,-.5) edge +(.5,-.5); \end{tikzpicture} &\mapsto \begin{tikzpicture} \node[int] (v) at (0,.3) {}; \node[ext, accepting] (w) at (0,-.3) {}; \draw (v) edge[loop] (v) (w) edge[->-] (v) (w) edge +(0,-.5) edge[->-] +(-.5,-.5) edge +(.5,-.5); \end{tikzpicture} & \begin{tikzpicture} \node[ext, accepting] (v) at (0,0) {}; \draw (v) edge[loop,-><-] (v) edge +(0,-.5) edge +(-.5,-.5) edge +(.5,-.5); \end{tikzpicture} &\mapsto 0 \end{align*}
There is an obvious homotopy for this operation (by contracting the unique non-tadpole edge at a thin tadpole vertex). Hence one can show that $E_1:=H(Q_{g,n},d_s')$ is spanned by graphs that do not have tadpoles of types 0,1,2, using the tadpole terminology of Section \ref{sec:Xgn def}.
In other words, $E_1$ is spanned by graphs that have a type 3 tadpole, but no other tadpole.
The differential on this page of the spectral sequence is given by splitting vertices.

We decompose $E_1$ further into a direct sum 
\[
  E_1 = \begin{tikzcd}[column sep = 0em]
    E_1'\ar[loop above]{}
    & \oplus & E_1'' \ar[bend right, swap]{ll}{f} \ar[loop above]{}
  \end{tikzcd}
\]
with $E_1'$ spanned by graphs for which the special vertex has valence 3, and $E_1''$ spanned by graphs for which the special vertex has valence $\geq 4$. Note that the special vertex cannot have valence 2 since we required $(g,n)\neq (1,0)$.
The piece of the differential $f$ acts as follows:
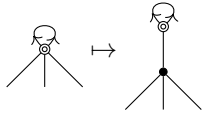
\[
  \begin{tikzpicture} \node[ext, accepting] (v) at (0,0) {}; \draw (v) edge[loop,-><-] (v) edge +(0,-.5) edge +(-.5,-.5) edge +(.5,-.5); \end{tikzpicture}
  \mapsto 
  \begin{tikzpicture} \node[ext, accepting] (v) at (0,.3) {}; \node[int] (w) at (0,-.3) {}; \draw (v) edge[loop,-><-] (v) edge (w) (w) edge +(0,-.5) edge +(-.5,-.5) edge +(.5,-.5); \end{tikzpicture}
\]
This morphism is clearly a bijection on basis elements, with the inverse given by contracting the unique non-tadpole edge at the special vertex. Here we are using that $(g,n)\neq (1,1)$, which implies that there must be a vertex at the other end of that edge.
But if $f$ is a bijection, then by \cite[Lemma 2.1]{PayneWillwacher} the complex $E_1$ is acyclic, and we conclude that $H(Q_{g,n})=0$ as desired.
\end{proof}

Note that we have morphisms of dg vector spaces
\[
  Z_{g,n} \xrightarrow{\sim} \gr_2\Feyn(\BV^*)(\!(g,n)\!)
  \to X_{g,n}.
\]
Next, we shall split $X_{g,n}$ into two graded subspaces,
\begin{equation}\label{equ:Xgn decomp}
  X_{g,n} = 
  \begin{tikzcd}[column sep = 0em]
    X_{g,n}'\ar[loop above]{}{d_0}
    & \oplus & X_{g,n}'' \ar[bend right, swap]{ll}{d_1} \ar[loop above]{}{d_0}
  \end{tikzcd},
\end{equation}
with $X_{g,n}'$ spanned by graphs that have a no tadpole, and $X_{g,n}''$ spanned by graphs that have a tadpole.
The arrows indicate the pieces $d_0,d_1$ of the differential.
Note that 
\begin{equation}\label{equ:X Z iso}
  X_{g,n}'\cong Z_{g,n}  
\end{equation}
is naturally identified with the image of $Z_{g,n}$ in $X_{g,n}$. Furthermore, the graphs generating $X_{g,n}''$ have exactly one tadpole, and this tadpole is of type 1.
Hence there is an isomorphism of graded vector spaces
\begin{equation}\label{equ:T iso}
  T : Z_{g-1,n}\cong X_{g-1,n}'\to X_{g,n}''
\end{equation}
defined on a graph by adding one tadpole of type 1.

We next consider the decomposition \eqref{equ:Xgn decomp} as a two-step filtration on $X_{g,n}$ and study the associated spectral sequence.
The $E^1$-page of this spectral sequence has the form 
\[
\begin{tikzcd}[column sep = 0em]
  H(X_{g,n}',d_0)
  & \oplus & H(X_{g,n}'',d_0) \ar[bend right, swap]{ll}{d_1}.
\end{tikzcd}.
\]
Since we are considering a two-step filtration, all differentials on higher pages must vanish, and the spectral sequence converges to cohomology on the $E^2$ page.
That is, we have that 
\[
  H(X_{g,n}) \cong \coker d_1 \oplus \ker d_1.
\]
To arrive at the first statement of Proposition \ref{prop:HyCom MM comparison} we just use the identifications \eqref{equ:X Z iso} and \eqref{equ:T iso} and Lemma \ref{lem:Zgn qiso}.
These show that the $E^1$-page of our spectral sequence is isomorphic to the cone of the morphism
\[
  \nabla_{g,n} : 
  H(\gr_2\Feyn(\BV^*)(\!(g-1,n)\!))
  \to 
  H(\gr_2\Feyn(\BV^*)(\!(g,n)\!))[2]
\]
obtained by concatenating $T$ and $d_1$.

The final statement of Proposition \ref{prop:HyCom MM comparison} asserts that the morphism $\nabla_{g,n}$ is zero for $n=0$.
But this is the content of \cite[Theorem 1.2]{PayneWillwacher}.
\hfill\qed

\subsection{Proof of Corollary \ref{cor:Hbdy Mg comparison new}}
To show Corollary \ref{cor:Hbdy Mg comparison new} first note that 
\[
\gr_2H^k(\Feyn(H^\bullet(\bar \M))(\!(g,0)\!))
\cong 
\gr_2 H^k_c(\M_{g})
\cong 
\gr_{6g-8} H^{6g-6-k}(\M_g)^* ,
\]
see for example the introduction of \cite{PayneWillwacher}.
The last statement of Proposition \ref{prop:HyCom MM comparison} then states that 
\[
  \gr_{6g-8} H^{6g-6-k}(\M_g)
  =
  \gr_2 H^{-k}(\Feyn(\HyCom)(\!(g,0)\!)) 
\oplus 
\gr_2 H^{1-k}(\Feyn(\HyCom)(\!(g-1,0)\!)).
\]
From part (3) of Corollary \ref{cor:main hbdy} we then conclude that 
\[
  \gr_{6g-8} H^{6g-6-k}(\M_g)
  \cong
  \gr_{6g-8} H^{6g-6-k}(\HMod_{g,0}^0) 
  \oplus 
  \gr_{6g-14} H^{6g-11-k}(\HMod_{g-1,0}^0)
\]
as desired.
\section{Euler characteristics}
\label{sec:euler}
We finish this article by considering the weight-graded Euler characteristics of
$\Feyn_\kk(D\BV^*)$, $\Feyn_\kk(\tGrav)$, $\Feyn(\BV)$ and $\Feyn(\HyCom)$.

For a permutation $\sigma \in S_n$, we define a
symmetric polynomial $p^\sigma \in \Lambda_n = \Q[x_1,\ldots,x_n]^{S_n}$ given by
$p^\sigma = p_{\lambda_1} \cdots p_{\lambda_\ell}$, 
where $p_k = \sum_{i=1}^n x_i^k$
is the $k$-th power sum symmetric polynomial and $(\lambda_1,\ldots,\lambda_\ell)$ is the \emph{cycle type} of $\sigma$.
To an $S_n$-module $V$ on which $\sigma\in S_n$ acts as $\rho(\sigma)\in \mathrm{End}(V)$ we may associate the \emph{Frobenius characteristic} defined by
\begin{align*} \ch(V) = \frac{1}{n!} \sum_{\sigma \in S_n} \mathrm{tr}(\rho(\sigma)) p^\sigma. \end{align*}
If $V$ is a differential graded $S_n$-module, recall that $V^k$ denotes the degree $k$ part of $V$.
The \emph{equivariant Euler characteristic} of $V$ is the following symmetric polynomial in $\Lambda_n$
$$\chi^{S_n}(V) = 
\sum_{k} (-1)^k \ch( V^k ),
$$
which is a homotopy invariant of $V$.
In case there is an additional weight grading defined on $V$ we set 
\[
  \chi_t^{S_n}(V)
  :=
  \sum_{W \geq 0} t^W
  \chi^{S_n}( \gr_W V ) \in \Lambda_n[\![t]\!].
\]

Let $\Lambda$ be the \emph{ring of symmetric functions} over the rationals. We may define it as the power series ring $\Lambda = \Q[\![p_1,p_2,\ldots]\!]$ generated by symbols $p_1,p_2,\ldots$ of degree $1,2,\ldots$ respectively, where $p_k$ stands in for the power symmetric functions of degree $k$. 
The \emph{Euler-Frobenius-Poincar\'e characteristic} of a dg symmetric sequence $\cP$ with a weight grading is the following power series in $\Lambda[\![t]\!]$
\begin{align*} \Ch_t(\cP) =  \sum_{n\geq 0} \chi^{S_n}_t( \gr_W \cP(\!(n)\!) ). \end{align*}

The
\emph{Necklace polynomial} is defined as
 $M_n(t) = \frac{1}{n} \sum_{d|n} \mu(n/d) t^d$, where
$\mu$ is the number-theoretical M\"obius function.
\begin{prop}\label{prop:Chts}
The (cyclic) Euler-Frobenius-Poincar\'e characteristics of $\tGrav$, $D\BV^*$, $\HyCom$ and $\BV$ are 
\begin{align}         \label{equ:ChtGrav} \Ch_t( \tGrav ) &= -\frac{1}{t^2(1-t^4)} \left( (1+p_1) \prod_{\ell \geq 1} (1+p_\ell)^{M_\ell(t^2)} -1 - (1 + t^2) p_1 \right) + \frac{1}{1 - t^2} \frac{p_1^2}{2} -\frac{1}{1 + t^2} \frac{p_2}{2} \\
\label{equ:ChtDBV}         \Ch_t( D\BV^* ) &= \Ch_t( \tGrav ) - \frac{t^2}{1-t^2} \frac{p_1^2}{2} - \frac{t^2}{1+t^2} \frac{p_2}{2} \\
\label{equ:ChtHyCom} \Ch_t(\HyCom) &= \Ch_t(\Feyn_{\kk}(\tGrav)(\!(0,-)\!)) \\
\label{equ:ChtBV} \Ch_t(\BV) &=  \frac 1{1-t^4} (1+t^2(1-t^2)p_1) \left( \prod_{\ell\geq 1}(1+t^{2\ell}(1-t^{2\ell})p_\ell)^{M_\ell(t^{-2})} -1 \right)-p_1. \end{align}
\end{prop}
\begin{proof}
The formula for $\Ch_t( \tGrav )$ is due to Getzler \cite[Theorem~5.7]{Getzler0}, and has just been adapted to our grading conventions.

For the second formula note that we have 
\[
  \Ch_t(D\BV^*)
  = 
  \Ch_t(D'\BV^*)
  +
  \chi_t^{S_2}(D\BV^*(\!(2)\!))
  \stackrel{\text{Thm.~\ref{thm:DCV}}}{=}
  \Ch_t( \tGrav )
  +
  \chi_t^{S_2}(D\BV^*(\!(2)\!)).
\]
The weight $W=2k$-part of $D\BV^*(\!(2)\!)$ is one-dimensional and spanned by $x_k:=\underbrace{\Dc\circ \cdots \circ\Dc}_{k\times}$ in degree $2k-1$. Let $\tau\in S_2$ be the transposition. It acts trivially on $\Dc$, but reverses the order of the $\Dc$'s and the $k-1$ odd compositions so that
\[
\tau \cdot x_k =(-1)^{k(k-1)/2}(-1)^{(k-1)(k-2)/2} x_k
=
-(-1)^{k}x_k.
\]
Hence we get 
\[
  \chi_t^{S_2}(D\BV^*(\!(2)\!))
  =
  \frac12
  \sum_{k\geq 1}
  t^{2k} (-p_1^2 + (-1)^{k} p_2)
  =
  - \frac{t^2}{1-t^2} \frac{p_1^2}{2}
 - \frac{t^2}{1+t^2} \frac{p_2}{2}.
\]

Equation \eqref{equ:ChtHyCom} follows from Getzler's result that $D\tGrav \simeq \HyCom^*$.

Next consider the formula for $\Ch_t(\BV)$.
For $n\geq 3$ note that we have the principal $(S^1)^n$-bundle
\begin{equation}\label{equ:pre ChtBV bundle}
\BV(\!(n)\!) \to \M_{0,n} = \BV(\!(n)\!)/\left((S^1)^n\right).  
\end{equation}
It is also $S_n$-equivariant.
Hence we have that 
\[
\chi_t^{S_n}(\BV(\!(n)\!)) 
=
\chi_t^{S_n}(H_\bullet(\M_{0,n}))
\circ \chi_t^{S_1}(H_\bullet(S^1)),
\]
where on the right-hand side we use the plethysm product on symmetric functions \cite[Section 7.2]{GK}, and we equip all objects with the weight grading by twice the homological degree.
Obviously, 
\[
\chi_t^{S_1}(H_\bullet(S^1))
=
(1-t^2)p_1.
\]
Following Getzler, define the cyclic sequence 
\[
  \gm(\!(n)\!)
  =
  \begin{cases}
    H_\bullet(\M_{0,n}) & \text{for $n\geq 3$} \\
    0 & \text{otherwise}
  \end{cases},
\]
equipped with the weight grading by twice the homological degree.
From the above it then follows that  
\begin{equation}\label{equ:ChtBVpre}
  \Ch_t(\BV)
  = 
  \Ch_t(\gm) \circ ((1-t^2)p_1 )
  +\chi_t^{S_2}(\BV(\!(2)\!))
  =
  \Ch_t(\gm) \circ ((1-t^2)p_1 )
  +
  \frac12 (1-t^2)(p_1^2+p_2).
\end{equation}
Getzler \cite[Theorem~5.7]{Getzler0} showed that 
\begin{equation}\label{equ:Chtm}
  \Ch_t(\gm)=
  \frac 1{1-t^4} (1+t^2p_1)
  \left( \prod_{\ell\geq 1}(1+t^{2\ell}p_\ell)^{M_\ell(t^{-2})} 
  -1-
  (1+t^{2})p_1 
 \right)
 -
 \frac{1}{1-t^2} \frac{p_1^2}{2}
 - 
 \frac{1}{1+t^2} \frac{p_2}{2}.
\end{equation}
Inserting this into \eqref{equ:ChtBVpre}
we obtain \eqref{equ:ChtBV}. %
\end{proof}

Explicitly, we have the following leading terms, expressed through Schur polynomials:
\begin{align*}   \Ch_t(\tGrav) &= s_{3} \\& -s_{2,2} + s_{4} t^{2} +s_{3,1,1} -s_{3,2} t^{2}+ s_{5} t^{4} \\& -s_{2,2,1,1} - s_{3,3} - s_{4,1,1} +( s_{3,2,1} + s_{4,1,1} ) t^{2} -s_{4,2} t^{4} + s_{6} t^{6} \\& +s_{3,1,1,1,1} + s_{3,2,1,1} + s_{3,2,2} + s_{4,2,1} + s_{5,2} \\&\quad\quad +( -s_{2,2,2,1} - s_{3,2,1,1} - s_{3,3,1} - s_{4,1,1,1} - s_{4,2,1} - s_{4,3} - s_{5,1,1} ) t^{2} \\&\quad\quad +( s_{3,3,1} + s_{4,2,1} + s_{5,1,1} ) t^{4} -s_{5,2} t^{6} + s_{7} t^{8} +O(x^8)         \\
 \Ch_t(\BV) &= s_{2} -s_{2} t^{2} \\&+s_{3} +( -s_{2,1} - s_{3} ) t^{2}+( s_{1,1,1} + s_{2,1} ) t^{4} -s_{1,1,1} t^{6} \\&+ s_{4} +( -s_{2,2} - s_{3,1} - s_{4} ) t^{2}+( 2s_{2,1,1} + s_{2,2} + 2s_{3,1} ) t^{4} +( -s_{1,1,1,1} - 3s_{2,1,1} - 2s_{3,1} ) t^{6} \\&\quad\quad+( s_{1,1,1,1} + s_{2,1,1} + s_{2,2} + s_{3,1} ) t^{8} +( -s_{2,2} ) t^{10} \\&+ s_{5} +( -s_{3,2} - s_{4,1} - s_{5} ) t^{2} +( s_{2,2,1} + 3s_{3,1,1} + 2s_{3,2} + 2s_{4,1} ) t^{4} \\&\quad\quad+( -3s_{2,1,1,1} - 3s_{2,2,1} - 6s_{3,1,1} - 3s_{3,2} - 3s_{4,1} ) t^{6} \\&\quad\quad+( 2s_{1,1,1,1,1} + 5s_{2,1,1,1} + 5s_{2,2,1} + 5s_{3,1,1} + 5s_{3,2} + 3s_{4,1} + s_{5} ) t^{8} \\&\quad\quad+( -2s_{1,1,1,1,1} - 3s_{2,1,1,1} - 5s_{2,2,1} - 3s_{3,1,1} - 4s_{3,2} - 2s_{4,1} - s_{5} ) t^{10} \\&\quad\quad+( s_{2,1,1,1} + 2s_{2,2,1} + 2s_{3,1,1} + s_{3,2} + s_{4,1} ) t^{12}+( -s_{3,1,1} ) t^{14} \\&+ s_{6} +( -s_{4,2} - s_{5,1} - s_{6} ) t^{2}+( 2s_{3,2,1} + s_{3,3} + 3s_{4,1,1} + 2s_{4,2} + 2s_{5,1} ) t^{4} \\&\quad\quad+( -3s_{2,2,1,1} - s_{2,2,2} - 4s_{3,1,1,1} - 7s_{3,2,1} - 4s_{3,3} - 8s_{4,1,1} - 4s_{4,2} - 3s_{5,1} ) t^{6} \\&\quad\quad+( 5s_{2,1,1,1,1} + 8s_{2,2,1,1} + 5s_{2,2,2} + 11s_{3,1,1,1} + 16s_{3,2,1} \\
&\quad\quad\quad+ 5s_{3,3} + 10s_{4,1,1} + 10s_{4,2} + 5s_{5,1} + s_{6} ) t^{8} \\&\quad\quad+( -2s_{1,1,1,1,1,1} - 10s_{2,1,1,1,1} - 12s_{2,2,1,1} - 10s_{2,2,2} - 15s_{3,1,1,1} - 23s_{3,2,1} \\&\quad\quad\quad - 5s_{3,3} - 11s_{4,1,1} - 14s_{4,2} - 6s_{5,1} - 2s_{6} ) t^{10} \\&\quad\quad+( 3s_{1,1,1,1,1,1} + 7s_{2,1,1,1,1} + 14s_{2,2,1,1} + 8s_{2,2,2} + 13s_{3,1,1,1} + 20s_{3,2,1} + 7s_{3,3} \\&\quad\quad\quad + 12s_{4,1,1} + 10s_{4,2} + 5s_{5,1} + s_{6} ) t^{12} \\&\quad\quad+( -s_{1,1,1,1,1,1} - 3s_{2,1,1,1,1} - 9s_{2,2,1,1} - 2s_{2,2,2} - 7s_{3,1,1,1} \\&\quad\quad\quad - 12s_{3,2,1} - 5s_{3,3} - 8s_{4,1,1} - 4s_{4,2} - 3s_{5,1} ) t^{14} \\&\quad\quad+( s_{2,1,1,1,1} + 2s_{2,2,1,1} + s_{2,2,2} + 3s_{3,1,1,1} + 4s_{3,2,1} + s_{3,3} \\&\quad\quad\quad + 2s_{4,1,1} + 2s_{4,2} + s_{5,1} ) t^{16} \\&\quad\quad +( -s_{2,2,2} - s_{3,1,1,1} - s_{4,2} ) t^{18} +O(x^7)       \\
\Ch_t(\HyCom) &= s_{3} \\& + s_{4} + s_{4} t^{2} \\& + s_{5} + ( s_{4,1} + s_{5} ) t^{2} + s_{5} t^{4} \\& + s_{6} + ( s_{4,2} + s_{5,1} + 2 s_{6} ) t^{2} + ( s_{4,2} + s_{5,1} + 2 s_{6} ) t^{4} + s_{6} t^{6} \\& + s_{7} + ( s_{4,3} + s_{5,2} + 2 s_{6,1} + 2 s_{7} ) t^{2} + ( s_{4,2,1} + 2 s_{4,3} + 3 s_{5,2} + 3 s_{6,1} + 4 s_{7} ) t^{4} \\ &\quad\quad + ( s_{4,3} + s_{5,2} + 2 s_{6,1} + 2 s_{7} ) t^{6} + s_{7} t^{8} \\& + s_{8} + ( s_{4,4} + s_{5,3} + 2 s_{6,2} + 2 s_{7,1} + 3 s_{8} ) t^{2} \\ &\quad\quad + ( s_{4,2,2} + 2 s_{4,3,1} + 3 s_{4,4} + 2 s_{5,2,1} + 5 s_{5,3} + s_{6,1,1} + 7 s_{6,2} + 6 s_{7,1} + 6 s_{8} ) t^{4} \\ &\quad\quad + ( s_{4,2,2} + 2 s_{4,3,1} + 3 s_{4,4} + 2 s_{5,2,1} + 5 s_{5,3} + s_{6,1,1} + 7 s_{6,2} + 6 s_{7,1} + 6 s_{8} ) t^{6} \\ &\quad\quad + ( s_{4,4} + s_{5,3} + 2 s_{6,2} + 2 s_{7,1} + 3 s_{8} ) t^{8} + s_{8} t^{10} +O(x^9)    \end{align*}

There are no known closed formulas for the Euler characteristics of the Feynman transforms of the above cyclic operads.
However, in low weights closed formulas can sometimes be obtained.
For example, since $\gr_0\Feyn(\BV)\cong \Feyn(\Com)$ we know the Euler characteristic of $\gr_0\Feyn(\BV)$ by the results of \cite{CFGP,MBeuler}.
Furthermore, the computation of \cite{PayneWillwacherEuler} can be slightly adapted to yield the following result for the weight 2 part:
\begin{prop}
The equivariant Euler characteristic of $\gr_2\Feyn(\BV)$ satisfies
\begin{equation}\label{equ:Euler FeynBV}
\begin{aligned}
 &\sum_{g,n} u^{g+n}\chi^{S_n}(\gr_2\Feyn(\BV)(\!(g,n)\!))
 \\&= -\frac12 u Z_1
  \left[
    \left(-\frac 1 {Z_1} +\sum_{\ell\geq 1} \frac{\mu(\ell)}{\ell}
    \left(\log(\ell u^\ell Z_\ell) 
    + 
    \psi_0(- Z_\ell) \right) \right)^2
  + 
  \sum_{\ell\geq 1} \frac{\mu(\ell)}{\ell}
  \left(\log(2\ell u^{2\ell} Z_{2\ell}) 
  + 
  \psi_0(- Z_{2\ell}) \right) 
  \right.
\\&\quad\quad\quad
  \left.
     -\frac 1 {(Z_1)^2} + \sum_{\ell\geq 1} \frac{\mu(\ell)^2}{\ell^2}
    \psi_1(-Z_\ell)
  \right]
\end{aligned}
\end{equation}
with 
\[
  Z_\ell := \frac 1 \ell \sum_{d\mid\ell}\mu(\ell/d) (u^{-d} + p_d))
\]
and 
\begin{equation}
  \label{equ:psi def}
  \psi_0(z) := -\sum_{j=1}^\infty \frac{B_j}{j} \frac 1 {(-z)^j} \mbox{ \ \ \ and \ \ \ }
  \psi_1(z) := -\sum_{j=0}^\infty B_j \frac 1 {(-z)^{j+1}},
\end{equation}
where $B_j$ are the Bernoulli numbers.
\end{prop}
\begin{proof}[Proof sketch]
First note that in the ``blown-up'' viewpoint, see Section \ref{sec:blownup}, the complex $\gr_2\Feyn(\BV)(\!(g,n)\!)$ is a complex of graphs of genus $g$ with external hairs, exactly two of which are labelled by $\omega$, $n$ are numbered and possibly some are labelled by $\epsilon$.
The Euler characteristic of virtually the same graded vector space of graphs has been computed in \cite[Section 4.3]{PayneWillwacherEuler}, there denoted $\widetilde{fX}^{\mathit{conn},2}$.
The only differences are that (i) in loc.~cit.~numbered hairs are considered edges and carry degree 1 and (ii) internal vertices in loc.~cit.~ may not carry tadpoles.
Difference (i) amounts to multiplying the $S_n$-representation $\Feyn(\!(g,n)\!)$ by the sign representation in degree $n$.
This change is accounted for on the characters by replacing $p_k\to -p_k$.
Difference (ii) can be accounted for by replacing $\BV_0$ by $\mathsf{Pois}_0$ in Proposition 3.2 of loc. cit.
By section 2.4 of loc. cit. the effect is to replace the terms $(1-u^\ell)u^\ell p_\ell$ in the formula of loc.~ cit.~by $u^\ell p_\ell$.
Thus we obtain our \eqref{equ:Euler FeynBV}.
\end{proof}

One may expand the above generating function on the computer.
The result for low $(g,n)$ is displayed in Table \ref{tbl:BV equiv}.
Furthermore, setting $p_k=0$ for all $k$ we obtain 
\begin{multline*}
  \sum_{g} u^{g}\chi(\gr_2\Feyn(\BV)(\!(g,0)\!))
  =
  u-u^8+3 u^9-u^{10}+3 u^{11}-4 u^{12}+4 u^{13}-12 u^{14}-2 u^{15}-12 u^{16}
\\
+16 u^{17}
+33 u^{18}
  +127 u^{19}+115 u^{20}-267
   u^{21}-463 u^{22}+1718 u^{23}-2586 u^{24}-45721 u^{25}+4560 u^{26}
\\
+742210 u^{27}+65910 u^{28}
   -13604215 u^{29}-10317
   u^{30}+O\left(u^{31}\right)
\end{multline*}

By Theorem \ref{thm:hycom bv} and an easy computation in $(g,n)=(1,1)$
we then find that 
\[
  \chi^{S_n}(\gr_2\Feyn(\HyCom)(\!(g,n)\!))
  =\begin{cases}
  \chi^{S_n}(\gr_2\Feyn(\BV)(\!(g,n)\!)) & \text{for $(g,n)\neq (1,1)$} \\
  0 & \text{for $(g,n)= (1,1)$}
\end{cases}.
\]

Finally, by Theorem \ref{thm:dbv_hycom} and a small computation in $(g,n)=(1,0),(1,1),(0,2)$ we find that 
\[
  \chi^{S_n} \left(\gr_{6g-8+2n} \AFeyn_\kk(D\BV^*)(\!(g,n)\!) \right)  
  =\begin{cases}
    \chi^{S_n}(\gr_2\Feyn(\BV)(\!(g,n)\!)) & \text{for $(g,n)\neq (1,1),(1,0),(0,2)$} \\
    p_1 & \text{for $(g,n)= (1,1)$} 
  \end{cases}.
\]

If sufficiently many terms of the Euler--Frobenius--Poincar\'e characteristic, $\Ch_t(\cC)$, of a cyclic (co)-operad are known, then we can use the procedure laid out in \cite[Section~8]{GK} to compute the Euler characteristic of $\Feyn(\cC)$ or $\Feyn_\kk(\cC)$. We used an adapted version of the program published with \cite{BVer} to compute the Euler characteristics of discussed Feynman transforms in arity $0$ up to a couple of orders in genus and in weight.

The weight-graded Euler characteristics of $\Feyn_\kk(D\BV^*) (\!(g,0)\!)$,
which by Theorem~\ref{thm:giansiracusa} agree with the weight-graded Euler characteristics of the handlebody group, $\HMod_{g,0}^0$,
are listed in Table~\ref{tab:DBV}. The top-weight contribution,
which by Theorem~\ref{thm:main} agrees with the Euler characteristic of 
$\Feyn(\Com)$,
is colored in blue. By Theorem~\ref{thm:dbv_hycom} the top$-2$-weight contribution is given by $\chi(\gr_2 \Feyn(\HyCom)(\!(g,0)\!))= \chi(\gr_2 \Feyn(\BV)(\!(g,0)\!))$, which itself can be computed explicitly as discussed above. The $0$-weight Euler characteristic agrees, by Theorem~\ref{thm:main}, with the Euler characteristic of $\Out(F_g)$ (see \cite{BVer} for a listing of $\chi(\Out(F_g))$).

Similarly, the weight-graded Euler characteristics of $\Feyn_\kk(\tGrav)$
are listed in Table~\ref{tab:Grav}. 
The $0$-weight Euler characteristic agrees with the Euler characteristic of $\Out(F_g)$ as $\gr_0 \tGrav \cong D \Com^*$ and $H(\Feyn_\kk(D \Com^*)(\!(g,0)\!)) = H(\Out(F_g))$.
The Euler characteristic is also known in the top and the 
top$-2$-weight by Theorem~\ref{thm:Feyn grav top}.
See Remark \ref{rem:feyngrav top 2} for a formula for the entries in weight top$-2$.

Tables~\ref{tab:BV} and \ref{tab:HyCom} list the 
weight-graded
Euler characteristics of $\Feyn(\BV)$ and $\Feyn(\HyCom)$ respectively. 
In both cases, the weight $0$ part is equal to the Euler characteristic of 
$\Feyn(\Com)$, because $\gr_0(\overline{\BV}) = \gr_0(\HyCom) = \Com$.
The argument for Lemma~\ref{lem:Feyngrav weights} can also be 
applied to $\HyCom$. It follows that $\Feyn(\HyCom)(\!(g,n)\!)$ is concentrated in even weights $0,\dots, 4g-6+2n$. 
This naive top-weight bound is indicated in blue in Table~\ref{tab:HyCom}.
However, in the $g=0$ case, the contributions to the top weight 
vanish due to the symmetries of tadpole graphs.
The top$-2$ weight Euler characteristic also appears to vanish. 

The auxiliary files to the arXiv version of this article contain 
larger versions of Tables~\ref{tab:DBV}, \ref{tab:Grav}, \ref{tab:BV} and \ref{tab:HyCom} with the filenames \texttt{tab\_DBV.csv}, \texttt{tab\_Grav.csv}, \texttt{tab\_BV.csv} and \texttt{tab\_HyCom.csv} respectively. 
These files can be opened with standard spreadsheet software.

In fact, we note that strictly speaking one of the Tables \ref{tab:DBV} and \ref{tab:BV} is redundant, due to the following general fact.

\begin{prop}\label{prop:BVDBV Euler}
For all $(g,n)\neq (0,1)$ and all $W$ we have that 
\[
  \chi^{S_n}\left(\gr_{W} \AFeyn(D\BV^*)(\!(g,n)\!)
  \right)
  =
  \chi^{S_n}\left(\gr_{6g-6+2n-W} \Feyn(\BV)(\!(g,n)\!)
  \right).
\]
\end{prop}
\begin{proof}
Let $\gm(\!(n)\!)=H_\bullet(\M_{0,n})$ be as in the proof of Proposition \ref{prop:Chts}. Define the auxiliary cyclic pseudo-operad $\gm_S$ such that
\[
  \gm_S(\!(n)\!)=
  \begin{cases}
    H_\bullet(\M_{0,n}\times (S^1)^n)
    =
    \gm(\!(n)\!) \otimes H_\bullet(S^1)^{\otimes n} & \text{for $n\geq 3$}\\
    \bar H_\bullet(S_1) = \Q\Delta & \text{for $n=2$} \\
    0 &\text{otherwise}
  \end{cases}.
\] 
We equip $\gm_S$ with the almost trivial cyclic pseudo-operad structure, defined such that the compositions with the binary operations are defined through the natural action of $\bar H_\bullet(S_1)$ on $H_\bullet(S^1)$, and all other compositions vanish.
We equip $\gm_S$ with the weight grading by twice the homological degree.
Recall that $\overline{\BV}$ 
is the augmentation ideal of the cyclic operad $\BV$.
From the proof of Proposition \ref{prop:Chts} we see that
\[
  \Ch_t(\gm_S)=\Ch_t(\overline{\BV}).
  \]
Indeed, the Euler characteristic ``cannot see'' the nontriviality of the bundle \eqref{equ:pre ChtBV bundle}.
Hence we also have that 
\[
  \chi^{S_n}\left(\gr_{6g-6-W} \Feyn(\BV)(\!(g,n)\!)
  \right)
  =
  \chi^{S_n}\left(\gr_{6g-6+2n-W} \Feyn(\gm_S)(\!(g,n)\!)
  \right).
\]
We are going to show that the right-hand side is equal to $ \chi^{S_n}\left(\gr_{W} \AFeyn(D\BV^*)(\!(g,n)\!) \right)$.
To this end we compute $H(\Feyn(\gm_S))$ explicitly.
The graphs spanning the complex $\Feyn(\gm_S)$ can have vertices of valences $\geq 2$.
Any such graph we may think of as a $\geq 3$-valent graph, to which strings of bivalent vertices have been added along its edges.
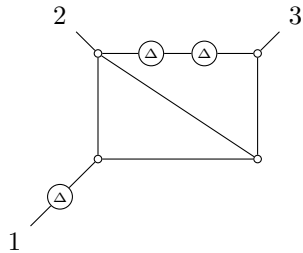
\[
\begin{tikzpicture} \node[ext] (v1) at (-.7,-.7) {}; \node[ext] (v2) at (+1.4,-.7) {}; \node[ext] (v3) at (-.7,+.7) {}; \node[ext] (v4) at (+1.4,+.7) {}; \node[ext] (b1) at (0,+.7) {$\scriptscriptstyle \Delta$}; \node[ext] (b11) at (0.7,+.7) {$\scriptscriptstyle \Delta$}; \node[ext] (b2) at (-1.2,-1.2) {$\scriptscriptstyle \Delta$}; \node (n1) at (-1.8,-1.8) {$1$}; \node (n2) at (-1.2,+1.2) {$2$}; \node (n3) at (+1.9,+1.2) {$3$}; \draw (v1) edge (v2) edge (v3) (b1) edge (v3) edge (b11) (b11) edge (v4) (b2) edge (v1) edge (n1) (v4) edge (v2) edge (n3) (v2) edge (v3) (v3) edge (n2); \end{tikzpicture}
\]
Here we can assume w.l.o.g.~that all the bivalent vertices are decorated by $\Delta$, whereas each $\geq 3$-valent vertex $v$ is decorated by some element of
\[
\gm_S(\!(n_v)\!)=\gm(\!(n_v)\!) \otimes 
H_\bullet(S^1)^{\otimes n_v}.
\]
But for each internal edge in the $\geq 3$-valent graph, the two factors of $H_\bullet(S^1)$ from the decorations on the endpoints of the edge and all possible additions of bivalent vertices on the edge yield, up to degree shift, one copy of the two-sided bar construction 
\[
  B_2:=
  \bigoplus_{k\geq 0}
  H_\bullet(S^1)
  \otimes (\bar H_\bullet(S_1)[1])^{\otimes k}
  \otimes H_\bullet(S_1).
  \]
Similarly, the contributions of the two-valent vertices on the legs can be combined into a one-sided bar construction 
\[
  B_1:=
  \bigoplus_{k\geq 0}
  H_\bullet(S^1)
  \otimes (\bar H_\bullet(S_1)[1])^{\otimes k}.
  \]

  It is well-known that 
  \begin{align*} H(B_2) &= \Q\oplus \Q \Delta_s & &\text{ and } & H(B_1) &= \Q, \end{align*}
  with 
  \[
    \Delta_s = 
    1\otimes \Delta +\Delta \otimes 1.
  \]

  From these observations
  \[
  H(\Feyn(\gm_S)) 
  \]
  can be seen as generated by $\gm$-decorated graphs with two types of internal edges: One corresponds to the normal edges used in the Feynman transform, of degree $-1$ and weight 0.
  The other type of edge is of degree $-2$ and weight $+2$, and represents the decoration with $\Delta_s$.

On the other hand, we have seen in Lemma \ref{lem:feynp dbv} that for $(g,n)\neq (1,0)$ 
\[
H(\AFeyn_\kk(D\BV^*)(\!(g,n)\!))
= H(\AFeyn_\kk'(D\BV^*)(\!(g,n)\!)).
\]
Furthermore, we have that 
\[
  \chi^{S_n}\left(\gr_W\AFeyn_\kk'(D\BV^*)(\!(g,n)\!) \right)
=
\chi^{S_n}\left(\gr_W\AFeyn_\kk'(H(D\BV^*))(\!(g,n)\!) \right).
\]

Hence it suffices to check that there is an $S_n$-equivariant isomorphism of vector spaces 
\begin{equation}\label{equ:vs iso Euler}\phi\colon : \gr_{W}  H^{k}(\Feyn(\gm_S) (\!(g,n)\!))
  \xrightarrow{\cong}
\gr_{6g-6+2n-W}\AFeyn_\kk'(H(D\BV^*))(\!(g,n)\!)^{6g-6+2n+k}.
\end{equation}
In fact, the vector spaces on both sides are generated by the same types of graphs, namely genus $g$ graphs with $n$ legs and two types of edges. The vertices are either decorated by $\gm$ or $\tGrav$.
Let $\Gamma$ be such an $\gm$-decorated graph on the left-hand side of \eqref{equ:vs iso Euler}.
Then we define $\phi$ by sending $\Gamma$ to the graph $\phi(\Gamma)=\tilde \Gamma$ on the right-hand side of \eqref{equ:vs iso Euler}, built from $\Gamma$ by the following replacements. 
\begin{itemize}
  \item
  Every decoration by $\gm(\!(n_v)\!)^{-k_v}\cong H_{k_v}(\M_{0,n_v})$ on some vertex $v$ of $\Gamma$ is replaced by a decoration in $\tGrav(\!(n_v)\!)^{2n_v-6-k_v}\cong H_c^{2n_v-6-k_v}(\M_{0,n_v})$ on the corresponding vertex of $\tilde \Gamma$, using the identification 
  \[
    H_{k_v}(\M_{0,n_v}) \cong H_c^{2n_v-6-k_v}(\M_{0,n_v}).
  \]
\item Every ``normal'' edge in $\Gamma$ of degree $-1$ and weight 0 is replaced by a marked ($\Dc$-decorated) edge of degree $+1$ and weight $2$ in $\tilde \Gamma$.
\item Every $\Delta_s$-edge in $\Gamma$ of degree $-2$ and weight $0$ is replaced by a normal edge, of degree and weight 0 in $\tilde \Gamma$.
\end{itemize} 

It is clear that the map above defines a bijection on basis elements. We hence just need to verify the degree and weight counts.
Suppose that $\Gamma$ has $V$ vertices, $e_1$ normal edges and $e_2$ marked ($\Delta_s$-)edges. 
Then it has cohomological degree and weight 
\begin{align*} k &= -e_0 -2e_1 -\sum_v k_v \text{ and} \\
 W &= 2e_1 + 2\sum_v k_v. \end{align*}
On the other hand, the image $\tilde \Gamma$ of $\Gamma$ has degree and weight 
\begin{align*} \tilde k &= e_0 + \sum_v (2n_v - 6- k_v) = e_0 + 2n+ 4(e_0+e_1) -6V -\sum_v k_v = k + 2n +6g-6 \text{ and} \\
 \tilde W &= 2e_0 + \sum_v (2n_v-6-2k_v) =6g-6 +2n -W, \end{align*}
as desired.
\end{proof}

One can ask to what extent the identification of Proposition \ref{prop:BVDBV Euler} holds not only on the level of Euler characteristics, but also on the level of the cohomology, and more specifically whether generally
\[
    \gr_{W} H^{-k}(\Feyn(\BV)(\!(g,n)\!))
    \stackrel{?}\simeq
    \gr_{6g-6+2n-W} H^{6g-6+2n-k}(\AFeyn(D\BV^*)(\!(g,n)\!) ) 
    .
\]
Our main results imply that this identification of cohomology holds at least for $W=0,2$. The general case we have to leave to future work.

\begin{table}
\small{
\begin{tabular}{c|r r r r r r r r r r r r r r r r r r r}
$g$ \textbackslash $W$& 0&2&4&6&8&10&12&14&16&18&20&22&24&26&28&30&32&34&36 \\
\hline
$2$& $1$&$-1$&$0$&$0$ \cellcolor{col2!30}&0&0&0&0&0&0&0&0&0&0&0&0&0&0&0 \\
$3$& $1$&$0$&$-1$&$-1$&$0$&$0$&$1$ \cellcolor{col2!30}&0&0&0&0&0&0&0&0&0&0&0&0 \\
$4$& $2$&$-1$&$0$&$-3$&$1$&$1$&$1$&$-1$&$0$&$0$ \cellcolor{col2!30}&0&0&0&0&0&0&0&0&0 \\
$5$& $1$&$0$&$1$&$1$&$-3$&$-2$&$-1$&$0$&$1$&$1$&$0$&$0$&$1$ \cellcolor{col2!30}&0&0&0&0&0&0 \\
$6$& $2$&$0$&$2$&$-5$&$-3$&$1$&$9$&$2$&$-12$&$-1$&$3$&$3$&$1$&$-1$&$0$&$-1$ \cellcolor{col2!30}&0&0&0 \\
$7$& $1$&$0$&$6$&$8$&$-20$&$-9$&$11$&$-10$&$-19$&$43$&$32$&$-32$&$-25$&$2$&$7$&$3$&$1$&$0$&$1$ \cellcolor{col2!30} \\
\end{tabular}
}
\caption{Weight-graded Euler characteristics
$\chi(\gr_W\Feyn_\kk(D\BV^*) (\!(g,0)\!) )$.
\label{tab:DBV}
}
\end{table}

\begin{table}
\small{
\begin{tabular}{c|r r r r r r r r r r r r r r}
$g$ \textbackslash $W$& 0&2&4&6&8&10&12&14&16&18&20&22&24&26 \\
\hline
$2$& $1$&$1$ \cellcolor{col2!30}&0&0&0&0&0&0&0&0&0&0&0&0 \\
$3$& $1$&$2$&$3$&$1$ \cellcolor{col2!30}&0&0&0&0&0&0&0&0&0&0 \\
$4$& $2$&$3$&$8$&$11$&$6$&$1$ \cellcolor{col2!30}&0&0&0&0&0&0&0&0 \\
$5$& $1$&$5$&$16$&$32$&$48$&$31$&$9$&$1$ \cellcolor{col2!30}&0&0&0&0&0&0 \\
$6$& $2$&$3$&$28$&$80$&$167$&$245$&$197$&$72$&$13$&$1$ \cellcolor{col2!30}&0&0&0&0 \\
$7$& $1$&$6$&$31$&$165$&$475$&$952$&$1457$&$1297$&$616$&$142$&$17$&$1$ \cellcolor{col2!30}&0&0 \\
$8$& $1$&$-3$&$60$&$246$&$1079$&$3068$&$6205$&$9497$&$9438$&$5323$&$1608$&$250$&$22$&$1$ \cellcolor{col2!30} \\
\end{tabular}
}
\caption{Weight-graded Euler characteristics
$\chi(\gr_W \Feyn_\kk(\tGrav) (\!(g,0)\!) )$.
\label{tab:Grav}
}
\end{table}

\begin{table}
\small{
\begin{tabular}{c|r r r r r r r r r r r r r r r r r r r}
$g$ \textbackslash $W$& 0&2&4&6&8&10&12&14&16&18&20&22&24&26&28&30&32&34&36 \\
\hline
$2$& $0$&$0$&$-1$&$1$ \cellcolor{col2!30}&0&0&0&0&0&0&0&0&0&0&0&0&0&0&0 \\
$3$& $1$&$0$&$0$&$-1$&$-1$&$0$&$1$ \cellcolor{col2!30}&0&0&0&0&0&0&0&0&0&0&0&0 \\
$4$& $0$&$0$&$-1$&$1$&$1$&$1$&$-3$&$0$&$-1$&$2$ \cellcolor{col2!30}&0&0&0&0&0&0&0&0&0 \\
$5$& $1$&$0$&$0$&$1$&$1$&$0$&$-1$&$-2$&$-3$&$1$&$1$&$0$&$1$ \cellcolor{col2!30}&0&0&0&0&0&0 \\
$6$& $-1$&$0$&$-1$&$1$&$3$&$3$&$-1$&$-12$&$2$&$9$&$1$&$-3$&$-5$&$2$&$0$&$2$ \cellcolor{col2!30}&0&0&0 \\
$7$& $1$&$0$&$1$&$3$&$7$&$2$&$-25$&$-32$&$32$&$43$&$-19$&$-10$&$11$&$-9$&$-20$&$8$&$6$&$0$&$1$ \cellcolor{col2!30} \\
\end{tabular}
}
\caption{\label{tbl:euler 3} Weight-graded Euler characteristics
$\chi(\gr_W(\Feyn(\BV)) (\!(g,0)\!) )$.
\label{tab:BV}
}
\end{table}

\begin{table}
\small{
\begin{tabular}{c|r r r r r r r r r r r r r r r r}
$g$ \textbackslash $W$& 0&2&4&6&8&10&12&14&16&18&20&22&24&26&28&30 \\
\hline
$2$& $0$&$0$ \cellcolor{col2!30}&0&0&0&0&0&0&0&0&0&0&0&0&0&0 \\
$3$& $1$&$0$&$0$&$0$ \cellcolor{col2!30}&0&0&0&0&0&0&0&0&0&0&0&0 \\
$4$& $0$&$0$&$-1$&$1$&$0$&$0$ \cellcolor{col2!30}&0&0&0&0&0&0&0&0&0&0 \\
$5$& $1$&$0$&$0$&$1$&$-1$&$-2$&$0$&$0$ \cellcolor{col2!30}&0&0&0&0&0&0&0&0 \\
$6$& $-1$&$0$&$-1$&$0$&$0$&$7$&$16$&$1$&$0$&$0$ \cellcolor{col2!30}&0&0&0&0&0&0 \\
$7$& $1$&$0$&$1$&$2$&$6$&$4$&$-59$&$-153$&$-29$&$0$&$0$&$0$ \cellcolor{col2!30}&0&0&0&0 \\
$8$& $0$&$-1$&$1$&$4$&$7$&$-24$&$-58$&$513$&$1945$&$659$&$19$&$0$&?&? \cellcolor{col2!30}&0&0 \\
$9$& $0$&$3$&$3$&$11$&$-4$&$-50$&$115$&$437$&$-5365$&$-28408$&?&?&?&?&?&? \cellcolor{col2!30} \\
\end{tabular}
}
\caption{\label{tbl:euler 4} Weight-graded Euler characteristics
$\chi(\gr_W(\Feyn(\HyCom)) (\!(g,0)\!) )$.
\label{tab:HyCom}
}
\end{table}

\begin{table}
  \small{
  \begin{tabular}{T|T|T|T|T|U|M}
    $g$\textbackslash $n$ & 0 & 1 & 2 & 3 & 4 & 5\\ 
\hline
0 & $ 0 $ & $ 0 $ & $ -s_{2} $ & $ 0 $ & $ s_{4} $ & $ -s_{3,2} $ \\  
\hline
 1 & $ s_{} $ & $ s_{1} $ & $ -s_{2} $ & $ 0 $ & $ -s_{2,1,1} + s_{4} $ & $ s_{3,1,1} + s_{4,1} $ \\  
\hline
 2 & $ 0 $ & $ 0 $ & $ -s_{2} $ & $ -s_{2,1} - s_{3} $ & $ -s_{2,1,1} + 2s_{4} $ & $ -s_{2,1,1,1} - s_{2,2,1} + 3s_{4,1} + 2s_{5} $ \\  
\hline
 3 & $ 0 $ & $ s_{1} $ & $ s_{1,1} $ & $ -s_{2,1} - 3s_{3} $ & $ -s_{2,1,1} + 2s_{2,2} - s_{3,1} $ & $ -4s_{2,1,1,1} - 2s_{2,2,1} - 6s_{3,1,1} + 2s_{3,2} + 3s_{4,1} + 4s_{5} $ \\  
\hline
 4 & $ 0 $ & $ s_{1} $ & $ s_{1,1} $ & $ 2s_{1,1,1} - s_{2,1} - 3s_{3} $ & $ 3s_{1,1,1,1} - s_{2,2} - 5s_{3,1} - 3s_{4} $ & $ 5s_{1,1,1,1,1} - 2s_{2,1,1,1} - 12s_{3,1,1} + 3s_{3,2} + 2s_{4,1} + 7s_{5} $ \\  
\hline
 5 & $ 0 $ & $ 2s_{1} $ & $ 2s_{1,1} + 3s_{2} $ & $ 2s_{1,1,1} + 4s_{2,1} - s_{3} $ & $ 2s_{1,1,1,1} + 5s_{2,1,1} - s_{2,2} - 9s_{3,1} - 12s_{4} $ & $ 7s_{1,1,1,1,1} + 8s_{2,1,1,1} + 11s_{2,2,1} - 12s_{3,1,1} + 3s_{3,2} - 12s_{4,1} - s_{5} $ \\  
\hline
 6 & $ 0 $ & $ -s_{1} $ & $ -2s_{1,1} + 2s_{2} $ & $ 4s_{2,1} + 7s_{3} $ & $ 6s_{1,1,1,1} + 13s_{2,1,1} - 7s_{2,2} + s_{3,1} - 9s_{4} $ & $ 13s_{1,1,1,1,1} + 29s_{2,1,1,1} + 7s_{2,2,1} + 3s_{3,1,1} - 23s_{3,2} - 34s_{4,1} - 20s_{5} $ \\  
\hline
 7 & $ 0 $ & $ 0 $ & $ 2s_{2} $ & $ -7s_{1,1,1} + 8s_{2,1} + 10s_{3} $ & $ -9s_{1,1,1,1} + 20s_{2,1,1} + 8s_{2,2} + 29s_{3,1} $ & $ -12s_{1,1,1,1,1} + 30s_{2,1,1,1} + 38s_{3,1,1} - 45s_{3,2} - 59s_{4,1} - 57s_{5} $ \\  
\hline
 8 & $ -s_{} $ & $ -4s_{1} $ & $ -8s_{1,1} - 7s_{2} $ & $ -13s_{1,1,1} - 10s_{2,1} + 14s_{3} $ & $ -17s_{1,1,1,1} - 8s_{2,1,1} - 2s_{2,2} + 53s_{3,1} + 42s_{4} $ & $ -31s_{1,1,1,1,1} + 30s_{2,1,1,1} - 19s_{2,2,1} + 129s_{3,1,1} - 23s_{3,2} + 28s_{4,1} - 37s_{5} $ \\  
\hline
 9 & $ 3s_{} $ & $ 5s_{1} $ & $ -6s_{2} $ & $ -15s_{1,1,1} - 9s_{2,1} - 3s_{3} $ & $ -61s_{1,1,1,1} - 72s_{2,1,1} + 15s_{2,2} + 36s_{3,1} + 47s_{4} $ & $ -102s_{1,1,1,1,1} - 50s_{2,1,1,1} + 36s_{2,2,1} + 239s_{3,1,1} + 170s_{3,2} + 221s_{4,1} + 32s_{5} $ \\  
\hline
 10 & $ -s_{} $ & $ -4s_{1} $ & $ -6s_{1,1} - 16s_{2} $ & $ -3s_{1,1,1} - 55s_{2,1} - 32s_{3} $ & $ -23s_{1,1,1,1} - 140s_{2,1,1} - 20s_{2,2} - 43s_{3,1} + 75s_{4} $ & $ -99s_{1,1,1,1,1} - 277s_{2,1,1,1} - 119s_{2,2,1} + 5s_{3,1,1} + 220s_{3,2} + 417s_{4,1} + 241s_{5} $ \\  
\hline
 11 & $ 3s_{} $ & $ 13s_{1} $ & $ 30s_{1,1} + 11s_{2} $ & $ 29s_{1,1,1} - 28s_{2,1} - 83s_{3} $ & $ 36s_{1,1,1,1} - 56s_{2,1,1} + 68s_{2,2} - 131s_{3,1} - 39s_{4} $ & $ -148s_{1,1,1,1,1} - 863s_{2,1,1,1} - 479s_{2,2,1} - 1001s_{3,1,1} - 44s_{3,2} + 7s_{4,1} + 287s_{5} $ \\  
\hline
 12 & $ -4s_{} $ & $ -2s_{1} $ & $ 19s_{1,1} + 10s_{2} $ & $ 103s_{1,1,1} - 14s_{2,1} - 78s_{3} $ & $ 258s_{1,1,1,1} + 13s_{2,1,1} - 98s_{2,2} - 458s_{3,1} - 206s_{4} $ & $ 404s_{1,1,1,1,1} - 411s_{2,1,1,1} - 253s_{2,2,1} - 1614s_{3,1,1} - 232s_{3,2} - 375s_{4,1} + 432s_{5} $ \\  
\end{tabular}
  }
  \caption{\label{tbl:BV equiv}
  Equivariant Euler characteristic of the weight 2 part
  $\chi^{S_n}(\gr_2(\Feyn(\BV)) (\!(g,n)\!) )$.
  }
\end{table}

\appendix

\section{Giansiracusa's result for handlebodies with marked points}
\label{sec:giansiracusa}

We note that Theorem \ref{thm:giansiracusa} is essentially shown by Giansiracusa \cite{Giansiracusa}, but he only considers the case of handlebodies with marked disks, i.e., $\HMod_{g,0}^m$.
We sketch here how to extract the case of handlebodies with marked points from Giansiracusa's result.

Generally, let $\Hbdy$ be the handlebody modular operad, so that 
$$
\Hbdy(\!(g,m)\!) \simeq B\HMod_{g,0}^m,
$$
see \cite[Section 4.3]{Giansiracusa}.

First, we recall the following main result of loc.\ cit.
\begin{thm}[Theorem A and Proposition 4.3.2 of \cite{Giansiracusa}]\label{thm:giansiracusa 2}
The handlebody modular operad is weakly equivalent to the homotopy modular closure $\LMod_!\Hbdy_0$ of its genus zero part $\Hbdy_0$ for all $(g,n)\neq (1,0)$.
Furthermore, the cyclic operad $\Hbdy_0$ is homotopy equivalent to the framed little 2-disks operad.
\end{thm}

Next, the following relation between the handlebody groups seems to be well known, though we have not found a citeable reference in the literature.
\begin{lemma}\label{lem:hbdy S1 quotient}
  We have that 
  \[
  B\HMod_{g,n}^m \simeq 
  \Hbdy(\!(g,n+m)\!) /\!/ ( \Hbdy(\!(0,2)\!)^n ),
  \]
  where $\Hbdy(\!(0,2)\!)\simeq S^1$ acts by composition at $n$ of the $n+m$ marked disks. 
\end{lemma}
\begin{proof}
There is a short exact sequence of groups 
\begin{equation}
  \label{eq_SES_groups}
  1\to \mathbb Z^n \to \HMod_{g,0}^{n+m} \to \HMod_{g,n}^m \to 1,   
\end{equation}
  with the $\mathbb Z^n$ corresponding to Dehn twists around $n$ of the marked disks
(\cite[Corollary 5.9 and Corollary 6.2]{Hensel2020}, see also \cite[Proof of Corollary 3.7]{HainautPetersen}).

Each such Dehn twist $\phi$ can be obtained as follows:
Take $\mathcal{H}_0^2$ a genus zero handlebody with two marked disks. Its handlebody group $\HMod_{0,0}^2$ is infinite cyclic, generated by a meridional Dehn twist $\phi'$ (with the definition as in \cite[Example 5.1]{Hensel2020}). 
Now attach one of the two marked disks of $\mathcal{H}_0^2$ to the corresponding marked disk of the original genus $g$ handlebody with $n+m$ marked disks. The result is again a genus $g$ handlebody with $n+m$ marked disks and the image of $\phi'$ in its mapping class group is exactly the Dehn twist $\phi$. This gives an identification of $\HMod_{0,0}^2$ with the corresponding copy of $\Z$ in the kernel of \eqref{eq_SES_groups}.

Using the Borel Construction, the short exact sequence of groups in \eqref{eq_SES_groups} gives rise to a map of classifying spaces
  \[
  B\HMod_{g,0}^{n+m} \to B\HMod_{g,n}^m,   
  \]
(a stack in the sense of \cite{Geo:Topologicalmethodsgroup})
all of whose fibres are given by a model for $B(\HMod_{0,0}^2)^n\cong B\Z^n$ (see \cite[Theorem 7.1.10]{Geo:Topologicalmethodsgroup}).
This implies the claim.
\end{proof}

\begin{proof}[Sketch of proof of Theorem \ref{thm:giansiracusa}]
Let $g$, $n$, $m$ be such that $2g+n+m\geq 3$.
By Lemma \ref{lem:hbdy S1 quotient} we have that 
\[
H_\bullet(B\HMod_{g,n}^m) 
\cong  H_\bullet( \HBdy(\!(g,n+m)\!) /\!/ ( \HBdy(\!(0,2)\!)^n ) ).
\]
Since $(g,n+m)\neq (1,0)$ we can apply Theorem \ref{thm:giansiracusa 2} to obtain 
\[
  H_\bullet( \HBdy(\!(g,n+m)\!) /\!/ ( \HBdy(\!(0,2)\!)^n ) )
  \cong 
  H_\bullet( \LMod_!\Hbdy_0(\!(g,n+m)\!) /\!/ ( \LMod_!\Hbdy_0(\!(0,2)\!)^n ) )  .
\]
Arguing similarly to \cite[Section 7.1]{Giansiracusa} it follows from formality of the framed little disks operad that 
\begin{align*} &H_\bullet( \LMod_!\Hbdy_0(\!(g,n+m)\!) /\!/ ( \LMod_!\Hbdy_0(\!(0,2)\!)^n ) ) \\&\cong H_\bullet( \LMod_!H_\bullet(\Hbdy_0)(\!(g,n+m)\!) /\!/ ( \LMod_!H_\bullet(\Hbdy_0)(\!(0,2)\!)^{\otimes n} ) ) \\&= H_\bullet( \LMod_!\BV(\!(g,n+m)\!) /\!/ ( \LMod_!\BV(\!(0,2)\!)^n ) ). \end{align*}
However, the homotopy modular closure is computed by the Feynman transform of the dg dual, see \cite[Corollary 1.3, 1.4]{Ward19}, and hence 
\begin{align*} &H_\bullet\left( \LMod_!\BV(\!(g,n+m)\!) /\!/ ( \LMod_!\BV(\!(0,2)\!)^{\otimes n} ) \right) \\
 \cong ~& H_\bullet\left( \Feyn(D\BV)(\!(g,n+m)\!) /\!/ ( \Feyn(D\BV)(\!(0,2)\!)^{\otimes n} ) \right). \end{align*}
Finally, $\Feyn(D\BV)$ is a quasi-free modular operad, and in particular a quasi-free $\Feyn(D\BV)(\!(0,2)\!)^{\otimes n} )$-module, using that $2g+n+m\geq 3$.
 Hence the homotopy quotient can be replaced by an ordinary quotient.
\begin{align*} &H_\bullet\left( \Feyn(D\BV)(\!(g,n+m)\!) /\!/ ( \Feyn(D\BV)(\!(0,2)\!)^{\otimes n} ) \right) \\
 \cong ~&H_\bullet\left( \Feyn(D\BV)(\!(g,n+m)\!) / ( \Feyn(D\BV)(\!(0,2)\!)^{\otimes n} ) \right). \end{align*}
But in the special case $m=0$ this is just the definition of the amputated Feynman transform of $D\BV$.
\end{proof}

\end{document}